\documentclass[a4paper, 10pt]{amsart}
\usepackage[margin=3.5cm]{geometry}
\usepackage{amsmath}
\usepackage{epsfig}
\usepackage{graphicx}
\usepackage{eufrak}
\usepackage{dsfont}
\usepackage[english]{babel}
\usepackage{color}

\usepackage{palatino, mathpazo}
\usepackage{amscd}      
\usepackage{amssymb}
\usepackage{dsfont}
\usepackage{xypic}      
\LaTeXdiagrams          
\usepackage[all,v2]{xy}
\xyoption{2cell} \UseAllTwocells \xyoption{frame} \CompileMatrices
\usepackage{latexsym}
\usepackage{epsfig}
\usepackage{enumerate}

\usepackage{latexsym}
\usepackage{epsfig}
\usepackage{amsfonts}
\usepackage{enumerate}
\usepackage{mathrsfs}

\newbox\mybox
\def\overtag#1#2#3{\setbox\mybox\hbox{$#1$}\hbox to
  0pt{\vbox to 0pt{\vglue-#3\vglue-\ht\mybox\hbox to \wd\mybox
      {\hss$\ss#2$\hss}\vss}\hss}\box\mybox}
\def\undertag#1#2#3{\setbox\mybox\hbox{$#1$}\hbox to 0pt{\vbox to
    0pt{\vglue#3\vglue\ht\mybox\hbox to \wd\mybox
      {\hss$\ss#2$\hss}\vss}\hss}\box\mybox}
\def\lefttag#1#2#3{\hbox to 0pt{\vbox to 0pt{\vss\hbox to
      0pt{\hss$\ss#2$\hskip#3}\vss}}#1}
\def\righttag#1#2#3{\hbox to 0pt{\vbox to 0pt{\vss\hbox to
      0pt{\hskip#3$\ss#2$\hss}\vss}}#1}
\let\ss\scriptstyle

\def\Dot{\lower.2pc\hbox to 2pt{\hss$\bullet$\hss}}
\def\Circ{\lower.2pc\hbox to 2pt{\hss$\circ$\hss}}
\def\Vdots{\raise5pt\hbox{$\vdots$}}
\newcommand\lineto{\ar@{-}}
\newcommand\dashto{\ar@{--}}
\newcommand\dotto{\ar@{.}}

\allowdisplaybreaks[3]

\newtheorem{prop}{Proposition}[section]
\newtheorem{lem}[prop]{Lemma}

\newtheorem{cor}[prop]{Corollary}
\newtheorem{thm}[prop]{Theorem}
\newtheorem{rmk}[prop]{Remark}
\newtheorem{example}{Example}

\newtheorem{defn}[prop]{Definition}

\newtheorem{con}[prop]{Conjecture}
\newtheorem{condition}[prop]{Condition}

            {\nolinebreak $\Box$ \end{trivlist}}

\newcommand{\noprint}[1]{}

\renewcommand{\tilde}{\widetilde}

\newcommand{\sm}{\mbox{\tiny sm}}
\newcommand{\eq}{\mbox{\tiny eq}}
\newcommand{\Ext}{\mbox{Ext}}

\newcommand{\Hom}{\mbox{Hom}}

\newcommand{\scM}{\mathscr{M}}
\newcommand{\scT}{\mathscr{T}}

\newcommand{\XX}{{\mathfrak X}}

\renewcommand{\SS}{{\mathfrak S}}

\newcommand{\zz}{{\mathbb Z}}

\newcommand{\aaa}{{\mathbb A}}
\newcommand{\ttt}{{\mathbb T}}

\renewcommand{\ll}{{\mathbb L}}
\newcommand{\qq}{{\mathbb Q}}
\newcommand{\pp}{{\mathbb P}}
\newcommand{\cc}{{\mathbb C}}
\newcommand{\rr}{{\mathbb R}}

\newcommand{\Gm}{{{\mathbb G}_{\mbox{\tiny\rm m}}}}

\newcommand{\sT}{{\mathcal T}}

\newcommand{\sE}{{\mathcal E}}

\newcommand{\sL}{{\mathcal L}}

\newcommand{\sS}{{\mathcal S}}

\newcommand{\sO}{{\mathcal O}}

\newcommand{\sX}{{\mathcal X}}
\newcommand{\sY}{{\mathcal Y}}
\newcommand{\sM}{{\mathcal M}}

\newcommand{\sZ}{{\mathcal Z}}

\newcommand{\bfc}{\mathbf{c}}

\DeclareMathOperator{\Def}{Def}

\DeclareMathOperator{\id}{id}

\DeclareMathOperator{\Sch}{Sch}

\DeclareMathOperator{\Hilb}{Hilb}

\DeclareMathOperator{\Aut}{Aut}

\DeclareMathOperator{\Isom}{Isom}

\DeclareMathOperator{\ind}{ind}
\DeclareMathOperator{\Ob}{Ob}

\DeclareMathOperator{\vir}{vir}

\DeclareMathOperator{\vd}{vd}
\DeclareMathOperator{\Ch}{Ch}

\DeclareMathOperator{\Tot}{Tot}

\DeclareMathOperator{\Td}{Td}

\DeclareMathOperator{\QG}{QG}
\DeclareMathOperator{\Vol}{Vol}
\DeclareMathOperator{\CM}{CM}
\DeclareMathOperator{\quintic}{quintic}
\DeclareMathOperator{\sextic}{sextic}
\DeclareMathOperator{\lci}{lci}
\DeclareMathOperator{\Gor}{Gor}

\DeclareMathOperator{\norm}{norm}

\DeclareMathOperator{\Diag}{Diag}

\DeclareMathOperator{\germs}{germs}

\newcommand{\C}{\ensuremath{\mathscr{C}}}

\newcommand{\rk}{\mathop{\rm rk}}

\newcommand{\tr}{\mathop{\rm tr}\nolimits}

\newcommand{\ob}{\mathop{\rm ob}}

\newcommand{\spec}{\mathop{\rm Spec}\nolimits}

\newcommand{\proj}{\mathop{\rm Proj}\nolimits}
\newcommand{\tor}{\mathop{\rm tor}\nolimits}

\numberwithin{equation}{subsection}

\newcommand {\mat}      [1] {\left(\begin{array}{#1}}
\newcommand {\rix}          {\end{array}\right)}

\title[Virtual fundamental class for moduli of surfaces of general type]{The virtual fundamental class for the moduli space of surfaces of general type}
\author{Yunfeng Jiang}
\address{Department of Mathematics\\ University of Kansas\\ 405 Snow Hall 1460 Jayhawk Blvd\\Lawrence KS 66045 USA} 
\email{y.jiang@ku.edu}

\begin{document}
\sloppy 
\begin{abstract}
We prove that  the moduli stack of index-one covers of semi-log-canonical surfaces of general type is isomorphic to the KSBA moduli stack of stable general type surfaces. Using the index-one covering Deligne-Mumford stack of a semi-log-canonical surface, we define the $\lci$ cover. The $\lci$ cover, as a Deligne-Mumford stack, has only locally complete intersection singularities.
We then construct the moduli stack of $\lci$ covers so that it admits a proper map to the moduli stack of surfaces of general type. 

The smoothing  of lci covering Deligne-Mumford stacks can not control all the non lci slc singularities.  
Thus, we  construct the moduli stack of birational lci enhancement Deligne-Mumford stacks for all non lci slc singularities.  This moduli stack is in general nonseparated. We construct a separated {\em algebraic  space} from this moduli stack.

Next, we construct a perfect obstruction theory on each of these moduli  stacks and a virtual fundamental class in the Chow group. We then pushforward the virtual fundamental class to the KSBA moduli space. Thus, our construction proves Donaldson's conjecture  on the existence of a virtual fundamental class for KSBA moduli spaces.
A tautological invariant is defined by integrating a power of the first Chern class of the CM line bundle over the virtual fundamental class. This serves as a generalization of the tautological invariants defined by integrating tautological classes over the moduli space $\overline{M}_g$ of stable curves to the moduli space of stable surfaces.
\end{abstract}

\maketitle

\tableofcontents
\addtocontents{toc}{\protect\setcounter{tocdepth}{1}}

\section{Introduction}

The main goal of this paper is to construct a virtual fundamental class for the KSBA moduli spaces of semi-log-canonical (s.l.c.) surfaces of general type. More precisely, we define the moduli stack of $\lci$ covers and prove that there is a proper morphism from this stack to the KSBA moduli space.  

The smoothing of $\lci$ covers can not handle all non $\lci$ slc singularities.  We also construct the moduli stack of birational lci enhancement Deligne-Mumford stacks.  Birational lci enhancements have contractions and may have flop-related models which causes the moduli stack nonseparated.  We construct  its underlying separated algebraic  space.  
We then construct a perfect obstruction theory and a virtual fundamental class on each of the moduli stacks constructed in this paper.
We begin with background and main results. 

\subsection*{KSBA moduli space}
An s.l.c. surface $S$ is a projective $\qq$-Gorenstein surface $S$ satisfying Serre's condition $S_2$ and having only s.l.c. singularities.  It is called stable if the dualizing sheaf $\omega_S$ is ample. 
Fixing the invariants \(K^2:=K_S^2\) and \(\chi:=\chi(\mathcal{O}_S)\) for a projective s.l.c. surface \(S\), and an integer \(N>0\), we let \(M_N:=\overline{M}_{K^2, \chi, N}\) be the KSBA moduli stack which  represents the  moduli functor of $\qq$-Gorenstein deformation families 
$$\{\sS\to T\}$$ 
of stable s.l.c. surfaces $S$ such that 
$\omega_{\sS/T}^{[N]}$ is invertible. Here for any integer $r$, we have 
$\omega_{\sS/T}^{[r]}:=(\omega_{\sS/T}^{\otimes r})^{\vee\vee}=i_{*}\omega_{\sS^{0}/T}^{\otimes r}$, 
where $i: \sS^{0}\hookrightarrow \sS$ is the inclusion of the Gorenstein locus of $\sS/T$, which is the locus where the relative dualizing sheaf $\omega_{\sS/T}$ is invertible; 
see \cite[\S 3.1]{Hacking} and \cite[\S 5.4]{Kollar-Shepherd-Barron}.
For a sufficiently large and divisible \(N\),  \cite[Corollary 5.7]{Kollar-Shepherd-Barron} and \cite[Theorem 1.1, Definition 6.2, Remark 6.3]{KP15} proved that the stack \(M:=\overline{M}_{K^2, \chi, N}\) is a proper Deligne-Mumford stack with a projective coarse moduli space.

\subsection*{Index one cover}\label{subsec_indec_one}

The KSBA moduli stack $M_N$ does not admit a perfect obstruction theory, largely because a stable s.l.c. surface $S$ is not locally
complete intersection (lci).   It is not even Gorenstein, so Serre duality doesn’t work.  We  first address Gorenstein by enhancing any KSBA-stable surface $S$ with a canonical  stack structure which is Gorenstein.
This can be obtained  by the  construction of index one cover on local  singularity germs $(S,x)$,  and there is a  canonical Deligne-Mumford stack $\SS$ with coarse moduli space $S$. We call $\SS$ the index one covering Deligne-Mumford stack.  

\subsection*{LCI cover}

The index one covering Deligne-Mumford stack $\SS$ may contain non lci singularities. 
From the classification of s.l.c. singularities in \cite[Theorem 4.23, Theorem 4.24]{Kollar-Shepherd-Barron} and \S \ref{subsec_slc_classification}, 
the  only  non-lci  singularities of $\SS$ are simple elliptic singularities, cusps, and degenerate cusp singularity germs $(S,x)$ with embedding dimension $\ge 5$. 

For these singularities $(S,x)$, we define their $\lci$ covers $(\widetilde{Z},x)$ with transformation group $G$. The lci cover is constructed from the fundamental group of the link $\Sigma$ of the singularity germ, where $\Sigma$ is 
the boundary of a small neighborhood of the singularity and  is a compact, oriented, real 3-manifold.  
The $\lci$ cover construction is canonical on each analytic germ of the singularities considered above. Therefore, the local $\lci$ covering stacks $[\widetilde{Z}/G]$ glue to form a Deligne-Mumford stack. We have 

\begin{thm}\label{prop_lci_covering_DM_intro}(Proposition \ref{prop_lci_cover_DM})
    For any s.l.c. surface $S$ in the KSBA moduli space $M_N$, there is a canonical   Deligne-Mumford stack 
    $$
\pi^{\lci}: \SS^{\lci}\to  S $$
with  degree one such that $\SS^{\lci}$ only has $\lci$ singularities. We call  $\SS^{\lci}$ the $\lci$ covering Deligne-Mumford stack.  
\end{thm}
The morphism $\pi^{\lci}: \SS^{\lci}\to S$ factors through the index one covering Deligne-Mumford stack $\pi: \SS\to S$.
The stack $\SS^{\lci}$  has only lci singularities and its coarse moduli space is $S$. Consequently, its dualizing sheaf $\omega_{\SS^{\lci}}$ is invertible.

\subsection*{Moduli stack of index one covers}

Both index one covering Deligne-Mumford stacks $\SS$ and lci covering  Deligne-Mumford stacks $\SS^{\lci}$ can be put into flat families. 
For a $\qq$-Gorenstein deformation family $\sS/T$  of s.l.c. surfaces, 
Hacking \cite[\S 3.2]{Hacking} and \S \ref{subsec_Q-Gorenstein_deformation} constructed a canonical flat family  $\SS/T$ of index one covering Deligne-Mumford stacks with coarse moduli space $\sS/T$.
Let 
$$\sM_N^{\ind}:=\overline{\sM}^{\ind}_{K^2, \chi,N}: \Sch_{\mathbf{k}}\to \text{Groupoids}$$
be the moduli functor 
$$T\mapsto \{f: \SS\to T\}$$ 
which sends $T$ to the isomorphism class of flat families  
$\{\SS\to T\}$
of index one covering Deligne-Mumford stacks.  

\begin{thm}\label{thm_index_one_DM_intro}(Theorem \ref{thm_index_one_covering_stack}, see also \cite{AH}) 
The moduli functor $\sM^{\ind}_N$ is represented by a Deligne-Mumford stack $M_N^{\ind}:=\overline{M}^{\ind}_{K^2, \chi,N}$. 
There exists an isomorphism between Deligne-Mumford stacks 
$$f: M_N^{\ind}\to M_N=\overline{M}_{K^2, \chi, N}.$$
If $N$ is sufficiently large and divisible, then $M^{\ind}:=M_N^{\ind}$ is a proper Deligne-Mumford stack and the isomorphism $f: M^{\ind}\to M$ induces an isomorphism on the projective coarse moduli spaces.
\end{thm}
The moduli stack $M^{\ind}$ admits an obstruction theory in the sense of Behrend-Fantechi \cite{BF} and Li-Tian \cite{LT}, see 
Theorem \ref{thm_OT_slc}. The obstruction theory is not perfect because an index one covering Deligne-Mumford stack $\SS$ is not lci, in general.

\subsection*{Perfect obstruction theory}
In this paper we construct the moduli stack $M_N^{\lci}=\overline{M}^{\lci}_{K^2, \chi, N}$ of lci covering Deligne-Mumford stacks in Theorem \ref{prop_lci_covering_DM_intro}, togehter with a proper morphism $f^{\lci}: M_N^{\lci}\to M_N$, and a perfect obstruction theory on $M_N^{\lci}$. We first present the main result.  

Let $N$ be sufficiently large divisible and 
$M^{\lci}:=\overline{M}^{\lci}_{K^2, \chi, N}$.
There is a  universal family
$p^{\lci}: \scM^{\lci}\rightarrow M^{\lci}$ which is a projective, flat, and relatively Gorenstein morphism.  
Let $\omega^{\lci}:=\omega_{\scM^{\lci}/M^{\lci}}[2]$. 
For an lci covering Deligne-Mumford stack $\SS^{\lci}$,  let $T_{\SS^{\lci}}$ be the tangent sheaf. 
Roughly speaking we take  $H^1(\SS^{\lci}, T_{\SS^{\lci}})$
as the deformation space, and $H^2(\SS^{\lci}, T_{\SS^{\lci}})$ as the obstruction space.  The lci condition (plus the finite automorphisms of $\SS^{\lci}$) implies all the other $H^i(\SS^{\lci}, T_{\SS^{\lci}})$ vanish for $i\neq 1, 2$.  Then 
taking Serre duals and working in the family with the universal lci covering Deligne-Mumford stack  $p^{\lci}: \scM^{\lci}\rightarrow M^{\lci}$  we define the virtual cotangent bundle of $M^{\lci}$ to be
$$E^{\bullet}_{M^{\lci}}=Rp^{\lci}_{*}(\ll^{\bullet}_{\scM^{\lci}/M^{\lci}}\otimes \omega^{\lci})[-1],$$
where $\ll^{\bullet}_{\scM^{\lci}/M^{\lci}}$ is the relative cotangent complex of $p^{\lci}$, and $\omega_{\scM^{\lci}/M^{\lci}}$ is the relative dualizing sheaf 
of $p^{\lci}$ which is a line bundle.   Thus,  from Theorem \ref{thm_BF_prop6.1} (see also \cite[Proposition 6.1]{BF}), 
the Kodaira-Spencer map $\ll^{\bullet}_{\scM^{\lci}/M^{\lci}}\to (p^{\lci})^{*}\ll_{M^{\lci}}^{\bullet}[1]$ induces an obstruction theory 
\begin{equation}\label{eqn_OT_univ_intro}
\phi^{\lci}: E^{\bullet}_{M^{\lci}}\to \ll_{M^{\lci}}^{\bullet}
\end{equation}
on $M^{\lci}$, where $\ll_{M^{\lci}}^{\bullet}$ is the cotangent complex of $M^{\lci}$.  

\begin{thm}\label{thm_POT_M_intro}(Theorem \ref{thm_POT_Muniv})
Suppose that $N$ is sufficiently large and divisible. 
Let $M=\overline{M}_{K^2, \chi, N}$ be the moduli stack of stable s.l.c. surfaces of general type with invariants $K^2, \chi, N$, and $f^{\lci}: M^{\lci}\to M$ be  the moduli stack of $\lci$ covers over $M$.  
Then the  obstruction theory  $\phi^{\lci}: E^{\bullet}_{M^{\lci}}\to \ll_{M^{\lci}}^{\bullet}$ in (\ref{eqn_OT_univ_intro})
is a perfect obstruction theory in the sense of Behrend-Fantechi \cite{BF}.    
\end{thm}

\subsection*{Proof of main results}

We construct the local lci cover of a singularity germ $(S,x)$ using the fundamental group of the link $\Sigma$ of $x$. 
Theorem \ref{prop_lci_covering_DM_intro} is proved from gluing the local lci covering Deligne-Mumford stacks, see \S\ref{subsubsec_universal_discriminant}, \S \ref{subsubsec_discriminant} and \S \ref{subsubsec_lci_slc}. 

Theorem \ref{thm_index_one_DM_intro} is proved from the canonical construction of index one covering Deligne-Mumford stacks.  The moduli stack $M_N^{\ind}$ of index one covers is a new formulation of the KSBA moduli stack of stable s.l.c. surfaces.   The advantage of moduli stack of index one covers is that all the index one covering Deligne-Mumford stacks in $M_N^{\ind}$ are Gorenstein. 

To prove Theorem \ref{thm_POT_M_intro}, we first have the following result for the flat families of lci covering Deligne-Mumford stacks.  

\begin{thm}\label{thm_lci_cover_link_smoothing_T_intro}(Theorem \ref{thm_lci_cover_link_smoothing_T})
    Let $\overline{f}: \sS\to T$ be a flat smoothing family of stable slc surfaces over a scheme $T$ and $f: \SS\to T$ be the associated index one covering Deligne-Mumford stack.  Suppose that the non-lci singularity germs $(\SS_0,x)$  of the central fiber are not simple elliptic singularities of degree $5$, $6$ or $7$, or cusp and degenerate cusp  singularities that can not lifted to an equivariant smoothing of an $\lci$ cusp or an lci cover in the algebraic sense,  then $f: \SS\to T$ can be lifted to a flat family $f^{\lci}: \SS^{\lci}\to T$ of lci covering Deligne-Mumford stacks.  Moreover, $\sS$ is  the coarse moduli space of both $\SS$ and $\SS^{\lci}$.
\end{thm}

\begin{rmk}
    Although \cite[Theorem 7.2]{Jiang_cusp} proved that for any cusp singularity $(S,0)$, if it admits a smoothing, then it always admits an equivariant  smoothing by an $\lci$ cusp,  the method uses Inoue-Hirzebruch surfaces and works analytically.  It is not clear at the moment if it works algebraically. 
\end{rmk}

Theorem \ref{thm_lci_cover_link_smoothing_T_intro}
implies that the simple elliptic singularities of degrees $5, 6, 7$ and some cusp and degenerate cusp  singularities (see Theorem \ref{thm_minimally_elliptic_universal}, and examples in that section) are special.
The case of degree $5$ doesn't give  higher obstruction spaces, see \cite{Jiang_2021}.
For other cases, a one-parameter smoothing (where the degree $6$ and $7$ simple elliptic singularities are given by  degree $6$ and $7$ del Pezzo cones) has a canonical singularity whose link is simply connected. Because of this simple connectivity, we don't currently  know how to take lci covers.
Theorem \ref{thm_lci_cover_link_smoothing_T_intro}
is proved by a careful study of smoothing families of simple elliptic and cusp and degenerate cusp singularities, see Theorem \ref{thm_minimally_elliptic_universal2}, Theorem \ref{thm_elliptic_singularity_lci_lifting}, and Theorem \ref{thm_minimally_elliptic_universal}.

The moduli stack of lci covers is constructed from the following functor   
$$
\sM_N^{\lci}:=\overline{\sM}^{\lci}_{K^2, \chi, N}: \Sch_{\mathbf{k}}\to \text{Groupoids}
$$
which sends 
$$
T\mapsto \{\SS^{\lci}\to T\}
$$
where $\{\SS^{\lci}\to T\}$ is  the isomorphism class of  flat families  of stable $\lci$ covering Deligne-Mumford stacks. 

Any such family \(\SS^{\lci}\to T\) induces a $\mathbb{Q}$-Gorenstein deformation family $\mathcal{S} \to T$ of s.l.c. surfaces. We denote by \( M_N=\overline{M}_{K^2, \chi, N} \) the corresponding moduli functor induced from \( M_N^{\lci} \).
Koll\'ar's result in \cite[Theorem 2.6]{Kollar90} implies that the moduli functor \( M_N \) is coarsely represented by a projective scheme. 
There is  a stratification
\[
\sM_1^{\lci}\subset \sM_2^{\lci}\subset \cdots
\]
We denote the union (or limit) of this stratification by \( \sM^{\lci}:=\overline{\sM}^{\lci}_{K^2, \chi, N} \), where \( N \) is taken to be sufficiently large and divisible.   

\begin{thm}\label{thm_lci_cover_DM_intro}(Theorem \ref{thm_universal_covering_stack})
The moduli functor  $\sM_N^{\lci}$ is represented by a Deligne-Mumford stack $M_N^{\lci}:=\overline{M}^{\lci}_{K^2,\chi,N}$, and there exists a proper 
morphism between Deligne-Mumford stacks 
$$f^{\lci}: M_N^{\lci}\to M_N.$$  

If $N$ is sufficiently large and divisible, and $M_N$ is proper, then the stack $M^{\lci}:=M_N^{\lci}$ is a proper 
Deligne-Mumford stack and the morphism 
$f^{\lci}: M^{\lci}\to M$ is a proper  morphism which induces a proper  morphism  on their projective coarse moduli spaces.
\end{thm}

The moduli stack $M_N^{\lci}$ in Theorem \ref{thm_lci_cover_DM_intro} is proved using the flat families of lci covering Deligne-Mumford stacks.   There is a  universal family  and  Theorem \ref{thm_POT_M_intro} is obtained from Behrend-Fantechi's criterion for a perfect obstruction theory.

\subsection*{Moduli stack of birational lci enhancements}

Theorem \ref{thm_lci_cover_link_smoothing_T_intro} implies that the smoothing of $\lci$ covering Deligne-Mumford stacks can not handle all the non-lci s.l.c. singularities. We define moduli stack of birational $\lci$ enhancements to cover all such singualrities. 

We take crepant resolution on the one-parameter smoothing of such singularities. This technique  applies to any simple elliptic, cusp, and even degenerate cusp singularity.  Let $\overline{f}: (\sS,0)\to T$ be a   $\qq$-Gorenstein family of a simple elliptic  singularity of degree $6$, or $7$, or some cusps and degenerate cusp singularities that can not be lifted to an equivariant smoothing of an $\lci$ cusp and an $\lci$ degenerate cusp algebraically.  We take the crepant resolution $\tilde{f}: (\sX, 0)\to T$ on this singularity first for the one-parameter family, and then obtain the family $f^{\lci}: \XX^{\lci}\to T$ of lci covering Deligne-Mumford stacks by the lci cover construction for other singularities.
Families $f^{\lci}: \XX^{\lci}\to T$ over any base $T$ are obtained from one-parameter families. 
We call such a family a birational lci enhancement and the central fiber $\XX^{\lci}_0$ a birational lci enhancement Deligne-Mumford stack. 
A key feature distinguishing it from the link-covering construction is that the coarse moduli space of a birational lci enhancement admits a proper morphism to the original slc surface.  A more general generalization of this method will be studied in \cite{Jiang_Bubble} motivated by collapsing of K\"ahler-Einstein metrics. 

Two  smoothing families of birational lci enhancement Deligne-Mumford stacks are related by flops.  
Two  birational lci enhancement Deligne-Mumford stacks $\XX^{\lci}_1$ and $\XX^{\lci}_2$ are called $\mathbb{S}$-equivalent if a flat  family $\XX^{\circ, \lci}\to T\setminus\{0\}$ over $T\setminus\{0\}$ can be extended to families over $T$ by adding $\XX^{\lci}_1$ and $\XX^{\lci}_2$.  
This will cause the moduli stack of birational  lci enhancements  non-separated. 

We  define the moduli functor   
$$
\sM_N^{\mathrm{Bir}}:=\overline{\sM}^{\mathrm{Bir}}_{K^2, \chi, N}: \Sch_{\mathbf{k}}\to \text{Groupoids}; \quad T\mapsto \{\XX^{\lci}\to T\}$$
where $\{\XX^{\lci}\to T\}$ is  the isomorphism class of  flat families  of stable birational  $\lci$ enhancement Deligne-Mumford stacks.  Then without the $\mathbb{S}$-equivalence condition the moduli functor gives a moduli stack $M^{\mathrm{Bir}}_N$ which is in general ``nonseparated ", but there  still exists a  morphism 
$f^{\mathrm{Bir}}: M^{\mathrm{Bir}}_N\to M_N$ to the KSBA moduli space.  In this case there is a unique universal family 
$p^{\mathrm{Bir}}: \scM^{\mathrm{Bir}}\to M^{\mathrm{Bir}}_N$. 
The $\mathbb{S}$-equivalence on $M^{\mathrm{Bir}}_N$ defines a ``separated" algebraic  space. This can be explained by the morphism $\pi: \mathbf{D}\to \aaa^1$, where $\mathbf{D}$ is a line with double origins.
This morphism is not a morphism  as  good moduli spaces, but  a morphism as algebraic spaces. 

We have 

\begin{thm}\label{thm_good_model}(Theorem \ref{thm_stack_Bir_enhancement}, Theorem \ref{thm_stack_good_enhancement}, Appendix \ref{sec_POT_equivalence})
The moduli stack $M^{\mathrm{Bir}}_N$ is an algebraic Deligne-Mumford stack of finite type. 
There exists a separated algebraic  space $$\pi: M^{\mathrm{Bir}}_N\to M_N^{\mathrm{alg}}.$$
If $N$ is sufficiently large and divisible, $M_N^{\mathrm{alg}}$ is compact. 

The universal family 
$p^{\mathrm{Bir}}: \scM^{\mathrm{Bir}}\to M^{\mathrm{Bir}}_N$ induces a family 
    $$
    p^{\mathrm{alg}}: \mathscr{M}^{\mathrm{alg}}\to M_N^{\mathrm{alg}}.
$$
 Then 
there exists a perfect obstruction theory 
\begin{equation}\label{eqn_OT_univ_intro_canonical}
\phi^{\mathrm{alg}}: E^{\bullet}_{M_N^{\mathrm{alg}}}\to \ll_{M_N^{\mathrm{alg}}}^{\bullet}
\end{equation}
where 
$$E^{\bullet}_{M_N^{\mathrm{alg}}}=Rp^{\mathrm{alg}}_{*}(\ll^{\bullet}_{\mathscr{M}^{\mathrm{alg}}/M^{\mathrm{alg}}}\otimes \omega^{\mathrm{alg}})[-1].$$
\end{thm}

\subsection*{Virtual fundamental class}

The standard theory in \cite{BF} induces a virtual fundamental class on the moduli stacks $M_N^{\lci}$, $M_N^{\mathrm{Bir}}$, and $M_N^{\mathrm{alg}}$.  

We present for 
$[M^{\lci}]^{\vir}\in A_{\vd}(M^{\lci})$, 
where the virtual dimension is given by
$$\vd=\dim(H^1(S, T_S))-\dim(H^2(S, T_S))$$
for a smooth surface $S\in M_N$, and we have  $\vd=10\chi-2K^2$.

Let $f^{\lci}: M^{\lci}\to M$ be the proper morphism between these two Deligne-Mumford stacks.   The morphism 
$f^{\lci}$ is not necessarily representable, but induces a proper  morphism  on the coarse moduli spaces. 
From \cite[Definition 3.6 (iii)]{Vistoli}, 
\begin{defn}\label{defn_virtual_class_M_intro}
We define 
    \begin{equation}\label{eqn_virtual_class_M}
[M]^{\vir}:=f^{\lci}_*\left([M^{\lci}]^{\vir} \right)\in A_{\vd}(M)
\end{equation} 
to be the virtual fundamental class of the moduli stack 
$M$. 
\end{defn} 
    
An s.l.c. surface $S$ with only Kawamata-log-terminal (k.l.t.) singularities  is a  projective surface whose singularities, except codimension one simple normal crossing singularities, are only  quotient singularities.  We have:

\begin{thm}\label{thm_POT_Mind_intro}(Theorem \ref{thm_POT_Mind})
Let $M$ be the moduli stack of stable surfaces of general type with invariants $K^2, \chi, N$. 
If the moduli stack $M$ consists of s.l.c. surfaces with only k.l.t. singularities, 
then the  moduli    stack $M^{\lci}$ of $\lci$ covers is the same as the moduli stack $M^{\ind}$, which is isomorphic to the moduli stack $M$.
Moreover, (\ref{eqn_OT_univ_intro}) gives a perfect obstruction theory  for the moduli   stack $M^{\ind}$, and (\ref{eqn_virtual_class_M}) gives a virtual fundamental class for $M^{\ind}$.  
\end{thm}

Finally we have 

\begin{cor}\label{cor_lci_virtual_class_intro}(Corollary \ref{cor_POT_MG}) 
If the moduli stack $M$ consists only of l.c.i. surfaces, then the moduli stack $M^{\lci}$ of lci covers and the moduli stack $M^{\ind}$ of index one covers are all the same as the moduli stack $M$.  Moreover $M$ admits a perfect obstruction theory 
$$\phi: E^{\bullet}_{M}\to \ll_{M}^{\bullet}$$
in the sense of \cite{BF}, 
where 
$$E^{\bullet}_{M}=Rp^{}_{*}(\ll^{\bullet}_{\scM/M}\otimes \omega^{\bullet})[-1],$$
$\omega^{\bullet}:=\omega_{\scM/M}[2]$, 
and $\ll^{\bullet}_{\scM/M}$ is the relative cotangent complex of the universal family $p: \scM\to M$.  Therefore, the perfect obstruction theory induces a virtual fundamental class 
$[M]^{\vir}\in A_{\vd}(M)$. 
This proves Donaldson's conjecture for the existence of virtual fundamental class in his example \cite[\S 5]{Donaldson}.
\end{cor}

\subsection*{Tautological invariants}

In \cite{Donaldson}, Donaldson studied the Fredholm topology and enumerative geometry of surfaces of general type and proposed the following two premises:

(1) There exists a virtual fundamental class \([\overline{M}_{K^2, \chi, N}]^{\text{vir}} \in H_{*}(\overline{M}_{K^2, \chi, N},\mathbb{Q})\), constructed using the theory of Behrend-Fantechi \cite{BF} and Li-Tian \cite{LT}.

(2) The Miller-Mumford-Morita (MMM) classes can be extended to \(H^{*}(\overline{M}_{K^2, \chi, N},\mathbb{Q})\).

Donaldson calculated the tautological invariant defined by integrating the MMM-classes over this conjectural virtual fundamental class in an example in \cite[\S 5]{Donaldson}.  Theorem \ref{thm_POT_M_intro} and Definition \ref{defn_virtual_class_M_intro} prove Donaldson's conjecture (1) and Corollary \ref{cor_lci_virtual_class_intro} confirms the virtual fundamental class calculation in Donaldson's example.

For Donaldson's conjecture (2), Patakfalvi and Xu proved the ampleness of the CM line bundle on \( M= \overline{M}_{K^2, \chi, N}\) in \cite{PXu15}. The first Chern class of the CM line bundle is the first Kappa class \cite{Alexeev_2023}.
From Theorem \ref{thm_POT_M_intro} and Equation (\ref{eqn_virtual_class_M}), we define the tautological invariant as: 
\begin{defn}\label{defn_tautological_invariant_intro}
(Definition \ref{defn_tautological_invariants}) 
We define the tautological invariant by
$$I_{\CM}=\int_{[M]^{\vir}}(c_1(L_{\CM}))^{\vd}.$$
\end{defn}

It is therefore interesting to compute these tautological invariants. We include Donaldson's example in \S\ref{sec_examples}. It is more interesting to study the tautological invariants for  the KSBA moduli spaces of log surfaces of general type in \cite{Alexeev2}.  The perfect obstruction theory on this moduli space is quite subtle.  We hope to return to the virtual fundamental class of KSBA moduli space of log surface pairs in future work. 

In \cite{Alexeev_2023}, Alexeev computed Kappa classes and tautological invariants for several moduli spaces of surfaces of general type, including moduli spaces of product-quotient curves, Burniat surfaces and Campedelli surfaces. The moduli spaces in these examples from \cite{Alexeev_2023} are all smooth. In \cite{AJ2024}, the authors will study the virtual fundamental class for the moduli space of Burniat surfaces of degrees \( 5 \) and \( 4 \).

\subsection*{Convention}\label{subsec_class_T_intro}
We work over the field of complex numbers  $\mathbf{k}=\cc$ throughout  the paper, although some parts work for any algebraically closed field $\mathbf{k}$ of characteristic zero.  
For the notion of algebraic stack and Deligne-Mumford stack, we follow the book \cite{LMB}, \cite{DM} and  \cite{Stack_Project}. 
 All Deligne-Mumford stacks are quasi-projective which,  from A. Kretch's equivalence condition, means that they can be embedded into a smooth projective Deligne-Mumford stack.  The Chow group $A_*(M):=A_*(M, \qq)$ of the Deligne-Mumford stack $M$ is under $\qq$-coefficients as in \cite{Vistoli}.

Class $T$-singularities are either rational double point or two dimensional cyclic quotient singularities of the form 
$\spec \mathbf{k}[x,y]/\mu_{r^2 s}$,
where $\mu_{r^2 s}=\langle\alpha\rangle$ and there exists a primitive $r^2 s$-th root of unity 
$\eta$ such that the action is given by:
$\alpha(x,y)=(\eta x, \eta^{dsr-1}y)$
and $(d, r)=1$.     When $s=1$, these are called Wahl  singularities. 

Recall a normal surface singularity $(S,x)$ is a rational  singularity if the exceptional divisor of the minimal resolution  is a tree of rational curves.
Simple elliptic surface singularities, cusp or degenerate cusp surface singularities  were defined in \cite[Definition 4.20]{Kollar-Shepherd-Barron}.   A {\em simple elliptic singularity} is a normal Gorenstein surface singularity such that the exceptional divisor of the 
minimal resolution is a smooth elliptic curve. 
A normal Gorenstein surface singularity is called a {\em cusp} if the exceptional divisor of the minimal resolution is a cycle of smooth rational curves or a rational nodal curve.  
A {\em degenerate cusp} is a non-normal  Gorenstein surface singularity $S$. If $f: X\to S$ is a minimal semi-resolution, then the exceptional divisor is a cycle of smooth rational curves or a rational nodal curve.  
In this case $S$ has no pinch points and the irreducible components of $S$ have cyclic quotient singularities.

\subsection*{Outline} 
Here is a short  outline for this paper.  In \S \ref{sec_POT}  basic materials about perfect obstruction theory in \cite{BF} and  \cite{LT} are reviewed.   \S \ref{sec_surface_general_type} reviews the moduli stack of semi-log-canonical surfaces, and constructs the moduli stack of semi-log-canonical surfaces.  
In \S \ref{sec_moduli_index_one_cover} we construct  
the moduli stack of  index one covers over the moduli stack of s.l.c. surfaces. 
We define the moduli stack of $\lci$ covers over the moduli stack of s.l.c. surfaces in \S \ref{sec_universal_covering_DM}; and   in \S \ref{sec_semi-POT_MG} we construct the perfect obstruction theory.  
In \S \ref{sec_moduli_fake} we construct the moduli stack of birational lci enhancements. 
In \S \ref{sec_CM_tautological_invariants} we construct  the CM line bundle on the moduli stack of s.l.c. surfaces.   We define  the tautological invariant by integrating the power of the first Chern class of the CM line bundle over the virtual fundamental class. Finally, in \S \ref{sec_examples} we calculate some examples: the moduli stack of quintic surfaces, and   Donaldson's example on sextic surfaces in $\pp^3$ with a finite group action. We also give a short discussion on  the moduli stack  $\overline{M}_{24,11}$ of numerical minimal general type sextic surfaces with $K_S^2=24, \chi(\sO_S)=11$.   We discuss  the virtual fundamental class for this moduli stack, although we can not fully understand its construction. Finally in \S \ref{sec_POT_equivalence} we prove the perfect obstruction theory does not depend on the choice of representatives in the $\mathbb{S}$-equivalence class. 

\subsection*{Acknowledgments}

Y. J. would like to thank Professor Donaldson for suggesting this project,  sharing  his paper on enumerative invariants, and several valuable discussions in March 2023 at Simons Center of Stony Brook.   We thank Professor Richard Thomas for his interest in this work, continuous discussion on the lci covering Deligne-Mumford stacks, valuable suggestion on the introduction, and valuable inputs on the proof of equivalence of perfect obstruction theories.  This paper would  not be in its present shape without his help.

Y. J. thanks Professor Kai Behrend  for teaching him perfect obstruction theories,  nice  comments for the paper, and valuable discussions on  the detailed proof of the main results.  
Y. J. thanks  Professor Jonathan Wahl for the email correspondence on universal abelian cover of surface singularities, and Professor J\'anos Koll\'ar for the examples of smoothing  simple elliptic singularities of degree $6$ and $7$. 
Y. J. thanks Yuchen Liu, Song Sun, Chenyang Xu and Ziquan Zhuang for the correspondence and valuable discussion  on 
semi-log-canonical surfaces
and  the moduli stack of index one covers.
This work is partially supported by  NSF DMS-2401484, and a Simon Collaboration Grant.


\section{Preliminaries on perfect obstruction theory}\label{sec_POT}

We review the basic construction of perfect obstruction theory in \cite{BF} and  \cite{LT}.

\subsection{Perfect obstruction theory}\label{subsec_POT}

Let $M$ be a quasi-projective Deligne-Mumford stack, which is an algebraic stack over $\mathbf{k}$ in the sense of \cite{Artin} and \cite{LMB} with unramified diagonal. 
Let $\ll^{\bullet}_{M}$ be the cotangent complex of $M$ in the sense of \cite{Illusie1} and  \cite{Illusie2}.

\begin{defn}\label{defn_obstruction_theory}(\cite[Definition 4.4]{BF})
An obstruction theory for $M$ is a morphism
$$\phi: E_{M}^{\bullet}\to \ll^{\bullet}_{M}$$
in the derived category $D(\sO_M)$ such that 
\begin{enumerate}
\item $E_{M}^{\bullet}\in D(\sO_M)$ satisfies the condition that $h^{i}(E_{M}^{\bullet})=0$ for all 
$i>0$, and $h^{i}(E_{M}^{\bullet})$ is coherent for $i=0, -1$.
\item  $\phi$ induces an isomorphism on $h^0$ and an epimorphism on $h^{-1}$.
\end{enumerate}
\end{defn}

\begin{defn}\label{defn_perfect_obstruction_theory}(\cite[Definition 5.1]{BF})
An obstruction theory $\phi: E_{M}^{\bullet}\to \ll^{\bullet}_{M}$  for $M$ is  called  $perfect$ if 
$E_{M}^{\bullet}$ is of perfect amplitude contained in $[-1,0]$. 
\end{defn}

\subsection{Bundle stack}\label{subsec_bundle_stack}

Any complex  $E_{M}^{\bullet}\in D(\sO_M)$ defines an algebraic stack $h^1/h^0((E_{M}^{\bullet})^{\vee})$
over $M$ as follows: locally around an \'etale chart $U\to M$, $(E_{M}^{\bullet})^{\vee}|_{U}$ is a complex written as
$$(E_{M}^{\bullet})^{\vee}|_{U}=\Big[E_0\to E_1\to\cdots\Big].$$
The stack $h^1/h^0((E_{M}^{\bullet})^{\vee})(U)$ is the groupoid of pairs $(P,f)$ where $P$ is an $E_0$-torsor (principle homogeneous $E_0$-bundle) on $U$ and 
$f: P\to E_1|_{U}$ is an $E_0$-equivariant morphism of sheaves on $U$.  Thus $h^1/h^0((E_{M}^{\bullet})^{\vee})$ is a fiber category fiberd by groupoids which is 
an algebraic $M$-stack (called an abelian cone stack). 

If  $E_{M}^{\bullet}\in D(\sO_M)$  is perfect; i.e., of perfect amplitude contained in $[-1,0]$, then  $h^1/h^0((E_{M}^{\bullet})^{\vee})$ is a vector bundle stack, since \'etale locally around $U\to M$, $(E_{M}^{\bullet})^{\vee}|_{U}$ is a complex of vector bundles 
$(E_{M}^{\bullet})^{\vee}|_{U}=\Big[E_0\to E_1\Big]$.  The stack is $h^1/h^0((E_{M}^{\bullet})^{\vee})|_{U}=[E_1/E_0]$.

\subsection{Intrinsic normal cone}\label{subsec_intrinsic_normal_cone}

Let $M$ be a quasi-projective Deligne-Mumford stack.  \'Etale locally there exists a diagram
\[
\xymatrix{
U\ar[r]^{f} \ar[d]_{i}& Y\\
M,&
}
\]
where $i: U\to M$ is an \'etale morphism and $f: U\to Y$ is a closed immersion into a smooth scheme $Y$. 
There is a cone stack $[C_{U/Y}/T_{Y}|_{U}]$
where  $C_{U/Y}$ is the normal cone, and $T_{Y}|_{U}$ acts on the normal cone $C_{U/Y}$. 
Whenever we have a morphism $\chi: (U^\prime, Y^\prime)\to (U, Y)$ of the local embeddings, which means there exists a commutative diagram
\[
\xymatrix{
U^\prime\ar[r]^{f^\prime} \ar[d]_{\phi_U}& Y^\prime\ar[d]^{\phi_{Y}}\\
U\ar[r]^{f}& Y,
}
\]
where $\phi_U$ is \'etale and $\phi_Y$ is smooth, we have that 
$\left(C_{U/Y}\hookrightarrow N_{U/Y}\right)|_{U^\prime}$ is the quotient of $\left(C_{U^\prime/Y^\prime}\hookrightarrow N_{U^\prime/Y^\prime}\right)$ by the action of $f^{\prime*}T_{Y^\prime/Y}$.  
Here $N_{U/Y}$ is the normal sheaf of $U$ to $Y$. 
Hence the isomorphism 
$$\widetilde{\chi}: \Big[N_{U^\prime/Y^\prime}/f^{\prime*}T_{Y^\prime}\Big]\cong \Big[N_{U/Y}/f^{*}T_{Y}\Big]|_{U^\prime}$$ 
identifies the closed subcone stacks
$$\widetilde{\chi}: \Big[C_{U^\prime/Y^\prime}/f^{\prime*}T_{Y^\prime}\Big]\cong \Big[C_{U/Y}/f^{*}T_{Y}\Big]|_{U^\prime}.$$
The stacks $\Big[N_{U/Y}/f^{*}T_{Y}\Big]$ glue to give the stack 
$h^1/h^0((\ll_{M}^{\bullet})^{\vee})$, which is called the intrinsic normal sheaf; and the stacks 
 $\Big[C_{U/Y}/f^{*}T_{Y}\Big]$ glue to give the stack $\bfc_{M}$, which is called the {\em intrinsic normal cone} of $M$.

\subsection{Infinitesimal obstruction theory}\label{subsec_infinitesimal_obstruction}

We review a bit for the infinitesimal deformation and obstruction theory for  a later use. 

Let $T\to \overline{T}$ be a square-zero extension of scheme with ideal $J$; i.e., $J^2=0$. For the Deligne-Mumford stack $M$,  let  $g: T\to M$ be a morphism, then there 
is a canonical morphism
\begin{equation}\label{eqn_infinitesimal_1}
g^* \ll^{\bullet}_{M}\to  \ll^{\bullet}_{T}\to  \ll^{\bullet}_{T/\overline{T}}
\end{equation}
in $D(\sO_T)$ by functoriality properties of the cotangent complex.  One has $\tau_{\geq 1}\ll^{\bullet}_{T/\overline{T}}=J[1]$, so the homomorphism (\ref{eqn_infinitesimal_1}) can be taken as an element 
$$\omega(g)\in \Ext^1(g^* \ll^{\bullet}_{M}, J).$$
Basic fact about deformation theory says that an extension $\overline{g}: \overline{T}\to M$ of $g$ exists if and only if $\omega(g)=0$, and if $\omega(g)=0$ the extensions form a torsor under 
$\Ext^0(g^* \ll^{\bullet}_{M}, J)=\Hom(\Omega_M, J)$. 

Let $\phi:  E_M^{\bullet}\to \ll^{\bullet}_{M}$ be an obstruction theory.  Then \cite[Proposition 2.6]{BF} tells us that 
$$\phi^{\vee}: h^1/h^0((\ll_{M}^{\bullet})^{\vee})\to h^1/h^0((E_{M}^{\bullet})^{\vee})$$
is a closed immersion.  Since the intrinsic normal cone $\bfc_M\hookrightarrow h^1/h^0((\ll_{M}^{\bullet})^{\vee})$ is embedded into the intrinsic normal sheaf, we have that $\phi^{\vee}(\bfc_M)\hookrightarrow h^1/h^0((E_{M}^{\bullet})^{\vee})$ is a closed subcone stack. 
If $T\to \overline{T}$ is a square zero extension of $\mathbf{k}$-schemes with ideal sheaf $J$ and $g: T\to M$ is a morphism, then 
$\omega(g)\in \Ext^1(g^*\ll^{\bullet}_{M}, J)$ and we denote by 
$\phi^*\omega(g)\in  \Ext^1(g^*E^{\bullet}_{M}, J)$ the image of the obstruction $\omega(g)$ in $\Ext^1(g^*E^{\bullet}_{M}, J)$.
 
We have the following result in \cite{BF}.

\begin{thm}\label{thm_obstruction_theory_equ_conditions}(\cite[Theorem 4.5]{BF})
Let $M$ be a Deligne-Mumford stack. The following statements are equivalent:
\begin{enumerate}
\item $\phi:  E_M^{\bullet}\to \ll^{\bullet}_{M}$ is an obstruction theory.
\item $\phi^{\vee}: h^1/h^0((\ll_{M}^{\bullet})^{\vee})\to h^1/h^0((E_{M}^{\bullet})^{\vee})$ is a closed immersion of cone stacks over $M$. 
\item For any $(T, \overline{T}, g)$ as above, the obstruction $\phi^*\omega(g)\in  \Ext^1(g^*E^{\bullet}_{M}, J)$ vanishes if and only if an extension 
$\overline{g}$ of $g$ to $\overline{T}$ exists; and if $\phi^*\omega(g)=0$, the extensions form a torsor under $\Ext^0(g^*E^{\bullet}_{M}, J)=\Hom(g^*h^0(E_M^{\bullet}),J)$. 
\end{enumerate}
\end{thm}

\begin{rmk}
\cite[Theorem 4.5]{BF} has a fourth equivalent condition by using the stack $h^1/h^0(\ll_{T/\overline{T}}^{\bullet})=C(J)$ and the morphism
$\ob(g): C(J)\to g^*\ll^{\bullet}_{M}$.   Since we don't use this in this paper,  we refer the detailed discussion to \cite[Theorem 4.5]{BF}.
\end{rmk}

\subsection{Virtual fundamental class}\label{subsec_virtual_class}

We construct the virtual fundamental class as in \cite[\S 5]{BF} for a perfect obstruction theory
$\phi:  E_M^{\bullet}\to \ll^{\bullet}_{M}$.
First the intrinsic normal cone 
$$\bfc_M\hookrightarrow h^1/h^0((\ll_{M}^{\bullet})^{\vee})\hookrightarrow h^1/h^0((E_{M}^{\bullet})^{\vee})$$
is a closed subcone stack of the vector bundle stack $h^1/h^0((E_{M}^{\bullet})^{\vee})$. 
Then intersection theory of Artin stacks in \cite{Kretch} gives the virtual fundamental class 
$$[M]^{\vir}=0^{!}_{ h^1/h^0((E_{M}^{\bullet})^{\vee})}(\bfc_M)\in A_{\rk(E_M^{\bullet})}(M);$$
i.e., the intersection of the intrinsic normal cone $\bfc_M$ with the zero section of the bundle stack 
$ h^1/h^0((E_{M}^{\bullet})^{\vee})$. 
Readers may like to construct the virtual fundamental class by intersection theory on Deligne-Mumford stacks. For this, we take a 
global resolution of $E_M^{\bullet}$ (\cite[Lemma 2.5]{Behrend}) given by
$$E=\Big[ E^{-1}\to E^0\Big]$$ 
of two term vector bundles such that $E_M^{\bullet}\sim E$. Then we let $E_i:=(E^{-i})^{\vee}$ and form 
$E^{\vee}=\Big[ E_{0}\to E_1\Big]$.  We have the following Cartesian diagram
\[
\xymatrix{
C\ar[r]\ar[d]& E_1\ar[d]\\
\bfc_{M}\ar[r]& [E_1/E_0],
}
\]
where $C\subset E_1$ is a subcone inside the vector bundle $E_1$ which can be taken as the lift of the intrinsic normal cone $\bfc_M$. 
Then the virtual fundamental class 
$$[M]^{\vir}=0^{!}_{E_1}(C)\in A_{\rk(E)}(M)$$
is the intersection of the cone $C$ with the zero section of the vector bundle $E_1$. 
The construction of the virtual fundamental class $[M]^{\vir}$ is a fundamental tool to define enumerative invariants in algebraic geometry for various of 
moduli spaces $M$, see \cite{Behrend2}, \cite{Thomas}, \cite{PT} and \cite{TT1}.

\subsection{Moduli space of projective Deligne-Mumford stacks}\label{subsec_moduli_varieties}

We recall one result in \cite[\S 6]{BF} for the obstruction theory of the moduli space of projective varieties. 

Let $p: \scM\to M$ be a projective, flat morphism between two Deligne-Mumford stacks.  The morphism $p$ is called relatively Gorenstein if the relative dualizing 
complex $\omega_{\scM/M}^{\bullet}$ is a line bundle $\omega^{\bullet}$. 
Let $\ll_{\scM/M}^{\bullet}$ be the relative cotangent complex of $p$.  We construct the following complex
$$E_M^{\bullet}:=Rp_*\left(\ll_{\scM/M}^{\bullet}\otimes \omega^{\bullet}\right)[-1].$$
The Kodaira-Spencer map  $\ll_{\scM/M}^{\bullet}\to p^*\ll_M^{\bullet}[1]$ induces a map
$$\phi: E_M^{\bullet}\to \ll_M^{\bullet}.$$

\begin{thm}\label{thm_BF_prop6.1}(\cite[Proposition 6.1]{BF})
Let $p: \scM\to M$ be a projective, flat and relatively Gorenstein morphism of  Deligne-Mumford stacks. Assume that the family $\scM$ is universal at every point 
of $M$.  Then $\phi: E_M^{\bullet}\to \ll_M^{\bullet}$ is an obstruction theory for $M$.  Moreover, if 
$E_M^{\bullet}$ is perfect;  i.e., of perfect amplitude contained in $[-1,0]$, then $\phi$ is a perfect obstruction theory for $M$. 
\end{thm}
\begin{proof}
The proof is in \cite[Proposition 6.1]{BF}. We provide the proof here for completeness and  a later use. 

We show an equivalence condition as in Theorem \ref{thm_obstruction_theory_equ_conditions}. Consider a scheme $T$ and let 
$f: T\to M$ be a morphism, then we have the following  Cartesian diagram
\[
\xymatrix{
\scT\ar[r]^{g}\ar[d]_{q}& \scM\ar[d]^{p}\\
T\ar[r]^{f}& M
}
\]
given by the fiber product. 
Let $T\to \overline{T}$ be a square-zero extension with ideal sheaf $J$, then the obstruction to extending 
$\scT$ to a flat family over $\overline{T}$ lies in 
$\Ext^2(\ll^{\bullet}_{\scT/T}, q^*J)$. If the extensions exist, they form a torsor under 
$\Ext^1(\ll^{\bullet}_{\scT/T}, q^*J)$. 
The flatness of $p$ implies that $\ll^{\bullet}_{\scT/T}=g^*\ll_{\scM/M}^{\bullet}$, we have that 
$$\Ext_{\sO_{\scT}}^k(\ll^{\bullet}_{\scT/T}, q^*J)=\Ext_{\sO_{\scM}}^k(\ll^{\bullet}_{\scM/M}, Rg_*q^*J)=
\Ext_{\sO_{\scM}}^k(\ll^{\bullet}_{\scM/M}, p^*Rf_*J)$$
and also 
$$\Ext_{\sO_{\scM}}^k(\ll^{\bullet}_{\scM/M}, p^*Rf_*J)
=\Ext_{\sO_{\scM}}^k(\ll^{\bullet}_{\scM/M}\otimes \omega^{\bullet}, p^{!}Rf_*J)
=\Ext_{\sO_{M}}^{k-1}(E^{\bullet}_{M}, Rf_*J)
=\Ext_{\sO_{T}}^{k-1}(f^*E^{\bullet}_{M}, J).$$
Here we use $p^{!}Rf_*J=p^*Rf_*J\otimes \omega^{\bullet}$.

The family $\scM$ is universal, which means that the fibers of $p$ have finite automorphism groups. Therefore, 
$E_M^{\bullet}$ satisfies that $h^i(E_M^{\bullet})=0$ for $i>0$ and  $h^i(E_M^{\bullet})$ is coherent for 
$i=0, -1$.  The morphism $\phi:  E_M^{\bullet}\to  \ll_M^{\bullet}$ induces morphisms 
$$\phi_k: \Ext_{\sO_{\scT}}^k(\ll^{\bullet}_{\scT/T}, q^*J)=\Ext_{\sO_{T}}^{k-1}(f^*E^{\bullet}_{M}, J)
\to \Ext_{\sO_{T}}^{k-1}(f^*\ll^{\bullet}_{M}, J).$$
Then if $M$ is a moduli stack, then $\phi_1$ is an isomorphism and 
$\phi_2$ is injective. So from Theorem \ref{thm_obstruction_theory_equ_conditions}, $\phi$ is an obstruction theory. 

If $E_M^{\bullet}$ is perfect which is of perfect amplitude contained in $[-1,0]$, then 
$\phi$ is a perfect obstruction theory from Definition \ref{defn_perfect_obstruction_theory}. 
\end{proof}

\begin{rmk}
If $p$ is smooth and the relative fiber is of dimension $\leq 2$, then it is not hard to see that $E_M^{\bullet}$ is a perfect obstruction theory.
In the case that the relative fibers are all smooth projective surfaces, the cohomology $H^*(M, (Rp_*(\ll^{\bullet}_{\scM/M}\otimes \omega))^{\vee})$ calculates the cohomology $H^*(S, T_S)$ for each fiber $S$ for the morphism $p$. 
Let us further assume that all the surfaces in the fibers are of general type which means $S$ has a finite automorphism group.  Then $M$ is a Deligne-Mumford stack.  
The cohomology 
$H^1(S, T_S)$ classifies  the deformations for the surface $S$; and $H^2(S, T_S)$ classifies the obstructions.  Since there are no higher dimensional cohomology spaces, the obstruction theory is perfect. 

In this paper,  we apply Theorem \ref{thm_BF_prop6.1} in the more general setting for the moduli stack where $p: \scM\to M$ is the universal family of the moduli of surfaces with semi-log-canonical singularities which is called the KSBA compactification of the moduli space of surfaces of general type. 
\end{rmk}

\section{Moduli stack of surfaces of general type}\label{sec_surface_general_type}

In this section we review the moduli stack of surfaces of general type with only semi-log-canonical (s.l.c.) singularities.  The moduli space of varieties of general type has been studied for decades.  Our main references are \cite{Gieseker}, \cite{Kollar-Shepherd-Barron}, \cite{Kollar-Book}, \cite{Alexeev2},  \cite{KP15}, \cite{HMX14}, \cite{Hacking}.

\subsection{Index one cover of s.l.c. surfaces}\label{subsec_slc_moduli}

Let us recall the notion of stable surfaces.  Roughly speaking a stable surface is a surface which can arise as a limit of smooth surfaces under stable reduction. 

We fix some notations for the projective surface  $S$. Let $K_S$ be the canonical class of $S$, which is a Weil divisor class, and let $\omega_S$ be the dualizing sheaf.   From \cite[Appendix to \S 1]{Reid},  for any integer $N>0$ we set 
$$\omega_S^{[N]}:=\sO_S(NK_S)=(\omega_S^{\otimes N})^{\vee\vee}.$$
From \cite[Appendix to \S 1, Theorem 7]{Reid},  $\omega_S$ is a torsion-free sheaf of rank one. If $S$ is normal, $\omega_S$ is a divisorial sheaf which satisfies the equivalent conditions in \cite[Appendix to \S 1, Proposition 2]{Reid}.  
In particular, $\omega_S$ is reflexive if $S$ is normal. 

\begin{defn}\label{defn_surfaces_slc}
Let $S$ be a projective surface.  We say that $S$ has s.l.c. singularities if the following conditions hold:
\begin{enumerate}
\item the surface $S$ is reduced, Cohen-Macaulay, and has only double normal crossing singularities $(xy=0)\subset \aaa_{\mathbf{k}}^3$ away from a finite set of points;
\item  we use the notations above.  Let the pair $(S^{\nu}, \Delta^{\nu})$ be the normalization of $S$ with the inverse image of the double curve. Then $(S^{\nu}, \Delta^{\nu})$ has log canonical singularities; 
\item for some $N>0$ the $N$-th reflexive tensor power 
$\omega_S^{[N]}$ for  the dualizing sheaf $\omega_S$ is invertible. 
\end{enumerate}
\end{defn}

\begin{rmk}\label{rmk_slc_singularities}
Let us recall the type of surface singularities here.  Let $(S, P)$ be a $\qq$-Gorenstein singularity germ, and $f: Y\to S$ be a good semi-resolution of $S$ in sense of \cite[Proposition 4.13]{Kollar-Shepherd-Barron}.  Then there exists 
$N>0$ such that we can write 
$\omega_Y^N\cong f^*\omega_{S}^{[N]}\otimes \sO(\sum N a_i E_i)$, where $E_i$ are the exceptional divisors and all  $a_i$ are rational. 
Then $(S, P)$ is called 
\begin{enumerate}
\item semi-canonical if $ a_i\ge 0$,
\item semi-log-terminal if $ a_i> -1$,
\item semi-log-canonical if $ a_i\ge -1$.
\end{enumerate}
If $(S, P)$ is normal, then we get the definition of  canonical, log-terminal and log-canonical singularity with the above inequality unchanged. 
\end{rmk}

\begin{defn}\label{defn_stable_surfaces}
A stable surface is a connected projective surface $S$ such that $S$ has s.l.c. singularities and the dualizing sheaf $\omega_S$ is ample. 
\end{defn}

Let us recall the index one cover for a surface $S$ with s.l.c. singularities as in \cite[\S 2.3]{Hacking},  \cite{Reid} and \cite{Kollar-Book}.
Let $(S,P)$ be an s.l.c. surface germ. The index of $P\in S$ is the least integer $r$ such that $\omega_S^{[r]}$ is invertible around $P$.  Fix an isomorphism 
$\theta: \omega_S^{[r]}\to \sO_S$, we define 
$$Z:=\spec_{\sO_S}\left(\sO_S\oplus\omega^{[1]}_S\oplus\cdots\oplus \omega_S^{[r-1]}\right),$$
where the multiplication on $\sO_Z$ is defined by the isomorphism  $\theta$.  Then $\pi: Z\to S$ is a cyclic cover of degree $r$ which is called the index one cover of 
$S$.  This cover satisfies the properties that the inverse image of the point $P$ is a single point $Q\in Z$; the morphism $\pi$ is \'etale over $S\setminus P$; and the surface $Z$ is Gorenstein, which means that $Z$ is Cohen-Macaulay and the dualizing sheaf $\omega_Z$ is invertible.  The germ $(Z,P)$ is also s.l.c. This is uniquely determined locally in the \'etale topology. 

\subsection{Classification of s.l.c. singularities}\label{subsec_slc_classification}
We recall the classification of surface s.l.c. singularities in  \cite[Theorem 4.24]{Kollar-Shepherd-Barron}. 
The s.l.c. surface singularities  are exactly as follows:
\begin{enumerate}
\item the semi-log-terminal singularities;
\item the Gorenstein surfaces such that  every Gorenstein surface  $S$ is either semi-canonical (which is smooth, normal crossing, a pinch point or a DuVal singularity),
 or  has simple 
elliptic singularities, cusp, or degenerate cusp singularities;
\item the $\zz_2, \zz_3, \zz_4, \zz_6$ quotients of simple elliptic singularities;
\item the $\zz_2$ quotient of  cusps and degenerate cusps. 
\end{enumerate}
The  semi-log-terminal surface singularities are exactly as follows:
\begin{enumerate}
\item the quotient of $\aaa_{\mathbf{k}}^2$ by Brieskorn \cite{Brieskorn};
\item normal crossing or pinch points;
\item $(xy=0)$ modulo the group action  given by
$x\mapsto \zeta^a x$, $y\mapsto \zeta^b y$,  and $z\mapsto \zeta z$, where $\zeta$ is a primitive 
$r$-th root of unity and $(a,r)=1, (b, r)=1$;
\item $(xy=0)$ modulo the group action 
$x\mapsto \zeta^a y$, $y\mapsto x$,  and $z\mapsto \zeta z$, where $\zeta$ is a primitive 
$r$-th root of unity and $4| r$,  $(a,r)=2$;
\item $x^2=zy^2$ modulo the group action  given by $x\mapsto \zeta^{1+a} x$, $y\mapsto \zeta^a y$,  and $z\mapsto \zeta^2 z$, where $\zeta$ is a primitive 
$r$-th root of unity and $r$ odd, and  $(a,r)=1$;
\end{enumerate}
see \cite[Theorem 4.22, 4.23, 4.24]{Kollar-Shepherd-Barron}.
If there is a  finite group $\mathfrak{G}$-action on $S$, then it induces the action on s.l.c. germs.  If $(S, x)$ and $(S, x^\prime)$ are two s.l.c. germs, then the $\mathfrak{G}$-action  induces
a morphism $(S,x)\to (S,x^\prime)$ on the s.l.c. germs  which is a $\mathfrak{G}$-equivariant morphism under the above classification.

\subsection{$\qq$-Gorenstein deformations and obstruction spaces}\label{subsec_Q-Gorenstein_deformation}

Next we review  the $\qq$-Gorenstein deformation of s.l.c. surfaces and the obstruction spaces.

\subsubsection{$\qq$-Gorenstein deformations}
\begin{defn}\label{defn_Q_Gorenstein_deformation_G}
Let $S$ be a stable surface, and  $(S,x)$ be an s.l.c. surface germ. A deformation $(x\in \sS)/(0\in T)$ is $\qq$-Gorenstein if it is induced by an equivariant deformation of the index one cover of $(x\in S)$.  This means there exists a $\mu_r$-equivariant deformation  $\sZ/T$ of $Z$ whose quotient is $\sS/T$. Both $\sZ$ and $\sS$ admit $G$-actions compatible with the local $\mu_r$-action.
\end{defn}

Here is a result in \cite{Kollar-Shepherd-Barron} for one-parameter deformation families. 

\begin{lem}(\cite[Lemma 3.4]{Hacking})
Let $\sS/(0\in T)$ be a  flat family of s.l.c. surfaces over a curve $T$. 
Assume that the generic fiber is canonical, which has only Du Val singularities and the canonical line bundle 
$K_{\sS}$ is $\qq$-Cartier. Then $\sS/T$ is $\qq$-Gorenstein. 
\end{lem}

We collect some facts for the $\qq$-Gorenstein deformations.  
For a flat family $\sS/T$ of s.l.c. surfaces, let $\omega_{\sS/T}$ be the  relative dualizing sheaf. 
From \cite[\S 5.4]{Kollar-Shepherd-Barron}, \cite[Appdedix to \S 1]{Reid} and \cite[\S 3.1]{Hacking}, we have that 
$$\omega_{\sS/T}^{[N]}:=(\omega_{\sS/T}^{\otimes N})^{\vee\vee}=i_*(\omega_{\sS^0/T}^{\otimes N}),$$
where $i: \sS^{0}\hookrightarrow \sS$ is the inclusion of the Gorenstein locus; i.e., the locus where the relative dualizing sheaf $\omega_{\sS/T}$ is invertible. 
Suppose that $\omega_{\sS/T}^{[N]}$ is invertible, and 
if $(S,x)$ is an s.l.c. surface  germ  with index $r$ in the family $\sS/T$, then the index $r|N$.  

From \cite[Lemma 3.5]{Hacking}, let  $(S,x)$ be an s.l.c. surface  germ  with  index $r$, and $Z\to S$ be the index one  cover under the cyclic group $\zz_r$-action. 
Let 
$\sZ/(0\in T)$ be a $\zz_r$-equivariant deformation of $Z$ inducing a $\qq$-Gorenstein deformation $\sS/(0\in T)$ of $S$, 
then  we have that 
$$\sZ=\spec_{\sO_{\sS}}(\sO_{\sS}\oplus \omega^{[1]}_{\sS/T}\oplus\cdots\oplus \omega_{\sS/T}^{[r-1]}),$$
where the multiplication of $\sO_{\sZ}$ is given by fixing a trivialization of $\omega_{\sS/T}^{[r]}$.  If the deformation $\sS/(0\in T)$ admits a $G$-action, then  every power  $\omega_{\sS/T}^{[i]}$ is endowed with a $G$-action and the index one cover is also endowed with a $G$-action making this $\sZ/(0\in T)$ $G$-equivariant. 

The index one  cover of the s.l.c. germ $(\sS, x)$ is uniquely determined in the \'etale topology. These data of index one covers everywhere  locally on $\sS/T$ glue to define a Deligne-Mumford stack 
$\SS/T$ which we call the canonical covering (Hacking) stack, or the index one covering Deligne-Mumford stack associated with  $\sS/T$.  The dualizing sheaf $\omega_{\SS/T}$ is invertible.

Let us collect some deformation and obstruction facts about the index one covering Deligne-Mumford  stacks.  We replace $T$ by a $\mathbf{k}$-algebra $A$, and consider an infinitesimal extension $A^\prime\to A$. Let $\sS/A$ be a 
$\qq$-Gorenstein  family of s.l.c. surfaces  with $G$-action  and $\SS/A$ be its index one covering  Deligne-Mumford stack. 

\begin{defn}\label{defn_Q_deformation_stack}
A deformation of $\SS/A$ over $A^\prime$ is a  Deligne-Mumford stack  $\SS^\prime/A^\prime$ which is flat over $A^\prime$ such that 
$\SS^\prime\times_{\spec A^\prime}\spec A\cong \SS$.
\end{defn} 

Equivalently a deformation  $\SS^\prime/A^\prime$ of $\SS/A$ is a sheaf $\sO_{\SS^\prime}$ of flat $A^\prime$-algebras on the \'etale site of $\SS$ such that 
$\sO_{\SS^\prime}\otimes_{A^\prime}A=\sO_{\SS}$. Thus the deformation theory of $\SS$ is controlled by the cotangent complex 
$\ll_{\SS/A}^{\bullet}$ as in \cite{Illusie2}.   Let us fix the following notations.

Let $A$ be a $\mathbf{k}$-algebra and $J$  be a finite $A$-module.   For a flat family $\sS/A$ of schemes over $A$ let 
$\ll_{\sS/A}^{\bullet}$ be the relative cotangent complex.  Then we define 
$$T^i(\sS/A, J):=\Ext^i(\ll_{\sS/A}^{\bullet}, \sO_{\sS}\otimes_{A}J),$$ and 
$$\sT^i(\sS/A, J):=\sE xt^i(\ll_{\sS/A}^{\bullet}, \sO_{\sS}\otimes_{A}J).$$
The groups $T^i(\sS/A, J)$ control the deformation and obstruction theory of $\sS/A$. 

We are actually working on  the  $G$-equivariant $\qq$-Gorenstein deformation theory of $\sS/A$.   Thus for the $\qq$-Gorenstein  family  $\sS/A$ of s.l.c. surfaces,  let $\SS/A$ be the  family of  the index one  covering Deligne-Mumford stacks, 
and 
$\pi: \SS\to \sS$ be the map to  its coarse moduli space.  Define
$$T_{\QG}^i(\sS/A, J):=\Ext^i(\ll_{\SS/A}^{\bullet}, \sO_{\SS}\otimes_{A}J),$$
and 
$$\sT_{\QG}^i(\sS/A, J):=\pi_*\sE xt^i(\ll_{\SS/A}^{\bullet}, \sO_{\SS}\otimes_{A}J).$$ 
The following two results are proven by P. Hacking \cite[Proposition 3.7, Theorem 3.9]{Hacking}.

\begin{prop}\label{prop_deformation_S_SS_G}(\cite[Proposition 3.7]{Hacking})
Let $\sS/A$  be a  $\qq$-Gorenstein family  of s.l.c. surfaces and $\SS/A$ be  its corresponding  index one covering Deligne-Mumford  stack. 
Consider the infinitesimal extension $A^\prime \to A$,  and  let  $\sS^\prime/A^\prime$ be a  $\qq$-Gorenstein deformation of $\sS/A$, and 
$\SS^\prime/A^\prime$ be the corresponding  index one covering Deligne-Mumford  stack.  Then,  there exists a one-to-one correspondence from the set of isomorphism classes of 
$\qq$-Gorenstein deformation families of $\sS/A$ over $A^\prime$ to the set of isomorphism classes of flat deformation families $\SS^\prime/A^\prime$ over $A^\prime$.
\end{prop}

\begin{prop}\label{prop_local_deformation_obstruction_slc}
Let $\sS_0/A_0$  be a $\qq$-Gorenstein family  of s.l.c. surfaces, and let $J$ be a finite $A_0$-module.
Then we have that 
\begin{enumerate}
\item  the set of isomorphism classes of  $\qq$-Gorenstein  deformations of $\sS_0/A_0$ over 
$A_0+J$ is naturally an $A_0$-module and is canonically isomorphic to $T_{\QG}^1(\sS/A, J)$. 
Here $A_0+J$ means the ring $A_0[J]$ with $J^2=0$;
\item  let $A^\prime\to A\to A_0$ be the infinitesimal extensions,  and $J$ be the kernel of $A^\prime\to A$. 
Let $\sS/A$  be a $G$-equivariant $\qq$-Gorenstein deformation of $\sS_0/A_0$.
Then  we have 
\begin{enumerate}
\item  there exists a canonical element $\ob(\sS/A, A^\prime)\in T_{\QG}^2(\sS/A, J)$ called the obstruction class.  It vanishes if and only if there exists a 
 $\qq$-Gorenstein  deformation $\sS^\prime/A^\prime$ of $\sS/A$ over $A^\prime$. 
\item if  $\ob(\sS/A, A^\prime)=0$,  then the set of  isomorphism classes of $\qq$-Gorenstein  deformations $\sS^\prime/A^\prime$
is an affine space underlying $T_{\QG}^1(\sS_0/A_0, J)$.
\end{enumerate}
\end{enumerate}
\end{prop}
\begin{proof}
This is a basic result of deformation and obstruction theory of algebraic varieties; see \cite[Theorem 3.9]{Hacking} and \cite{Illusie2}. 
\end{proof}

\subsubsection{Higher obstruction spaces of the index one covering Deligne-Mumford stack}\label{subsubsec_higher_cohomology}

Let $S$ be an s.l.c. surface, and let $\SS\to S$ be the index one covering Deligne-Mumford  (Hacking) stack in \S \ref{subsec_Q-Gorenstein_deformation}. The spaces 
$T_{\QG}^i(S)=\Ext^i(\ll_{\SS}, \sO_{\SS})$ can be calculated by the local to global spectral sequence
$$E_2^{p,q}=H^p(\sT^{q}_{\QG}(S))\Longrightarrow T^{p+q}_{\QG}(S),$$
where $\sT^{q}_{\QG}(S):=\pi_*(\sE xt^q(\ll_{\SS}, \sO_{\SS}))$ and $\pi: \SS\to S$ is the map to its coarse moduli space. 
The spaces $T_{\QG}^i(S)$ for $i\ge 3$ classify the higher obstruction spaces for the $\qq$-Gorenstein deformations of $S$. We have that 

\begin{prop}\label{prop_slc_higher_cohomology}
Let $S$ be an s.l.c. surface satisfying  the following conditions:
\begin{enumerate}
\item $S$ is  Kawamata-log-terminal (k.l.t.);   or
\item the  possible simple elliptic singularity, the cusp and the degenerate cusp singularity of $S$, and the possible 
$\zz_2, \zz_3, \zz_4, \zz_6$ quotients of the  simple elliptic singularity, the $\zz_2$-quotient of the  cusp and the degenerate cusp singularity of $S$ all have embedding dimension at most $4$,
\end{enumerate}
then the higher obstruction spaces  $T_{\QG}^i(S)$ vanish for $i\geq 3$.  
\end{prop}
\begin{proof}
From the classification of semi-log-canonical surface singularities in \S \ref{subsec_slc_classification}, and known fact in birational geometry, 
a k.l.t. surface $S$ only has cyclic quotient singularities,  cyclic quotients of  the normal crossing, and pinch point singularities, or Du Val singularities.  
Then if the surface $S$ admits a $\qq$-Gorenstein deformation,  from \cite[Proposition 3.10]{Kollar-Shepherd-Barron}, the  quotient singularities must have the form 
$$\spec \mathbf{k}[x,y]/\mu_{r^2 s},$$
where $\mu_{r^2 s}=\langle\alpha\rangle$ and there exists a primitive $r^2 s$-th root of unity 
$\eta$ such that the action is given by
$$\alpha(x,y)=(\eta x, \eta^{dsr-1}y),$$
where $(d, r)=1$.    Thus the index one cover of $S$ locally has the quotient 
$$\spec \mathbf{k}[x,y]/\mu_{rs}$$ given by $\alpha^\prime(x,y)=(\eta^\prime x, (\eta^\prime)^{rs-1}y)$, which is an $A_{rs-1}$-singularity, and therefore is  l.c.i. 
The 
cotangent complex $\ll_{\SS}$ only has two terms concentrated in degrees $-1, 0$. Therefore,  the tangent sheaf $\sT^{q}_{\QG}(S)$ is zero for $q\geq 2$. 
By the local to global spectral sequence  $T^{i}_{\QG}(S)=0$ for $i\geq 3$. 

If an s.l.c. surface $S$ has a simple elliptic singularity, a cusp  or a  degenerate cusp singularity with embedding dimension at most $4$, then from \cite[Theorem 3.13]{Laufer},
and \cite{Stevens}, these singularities must be locally complete intersection singularities.  For the s.l.c. surfaces with $\zz_2, \zz_3, \zz_4, \zz_6$ quotients of a simple elliptic singularity,  a cusp and or   a degenerate cusp singularity such that the local embedding dimension $\leq 4$,  their index one covers 
$\SS$ locally must be l.c.i.,  and the  tangent sheaf $\sT^{q}_{\QG}(S)$ is zero for $q\geq 2$
making the global obstruction spaces  $T^{i}_{\QG}(S)=0$ for $i\geq 3$. 
\end{proof}

\begin{rmk}\label{rmk_slc_minimally_elliptic}
Recall for an s.l.c. surface $S$, the tangent sheaves $\sT^{q}_{\QG}(S)$ satisfy the following properties (see for example \cite{Hacking}):
\begin{enumerate}
\item $\sT^{0}_{\QG}(S)=\sT_S$ is the tangent sheaf of $S$;
\item  $\sT^{1}_{\QG}(S)$ supports on singular locus of $S$, which can be calculated as follows: if locally $\SS$ is given by 
$[V/\zz_r]\to U$ for an open subset $U\subset S$, we have
$$\sT^{1}_{\QG}(S)=\left(p_*\sE xt^1(\Omega_V, \sO_V)\right)^{\zz_r}$$
where $p: V\to U$ is the natural morphism;
\item  $\sT^{2}_{\QG}(S)$ supports on the locus of the index one cover $Z$ which is not a local complete intersection; 
\item $\sT^{q}_{\QG}(S)$ for $q\geq 3$ may support on non-complete intersection singularities of $S$. 
\end{enumerate}
Therefore, from the local to global spectral sequence,  to determine the higher obstruction spaces  $T_{\QG}^i(S)$ it is sufficient to know   $\sT^{q}_{\QG}(S)$ for $q\geq 3$ since for any coherent sheaf 
$F$ the cohomology spaces $H^p(S,F)$ only survive for $p=1, 2$.  
From \cite{Stevens}, if  a cusp  or a  degenerate cusp singularity has  embedding dimension $\geq 5$, then the singularity is definitely not  a complete intersection singularity.  There should exist an example of degenerate cusp singularity
$(S, p)$ such that its embedding dimension is $\geq 5$, and  the tangent sheaves $\sT^{q}_{\QG}(S)\neq 0$ for some  $q\geq 3$.   It is likely that for a cusp or degenerate  singularity germ $(S,p)$ with embedding dimension $>5$, if the tangent sheaf $\sT_{\QG}^2(S,\sO_S)\neq 0$, then $\sT_{\QG}^3(S,\sO_S)\neq 0$; see \cite{Jiang_2021}. 
In this situation, the obstruction  spaces  $T_{\QG}^i(S)$ are not zero for $i\geq 3$.   These higher obstruction spaces for the s.l.c. surface $S$ imply that there is no natural  Behrend-Fantechi style perfect obstruction theory on the moduli stack of surfaces of general type containing s.l.c. surfaces with  such type of  singularities. 
\end{rmk}

From Remark \ref{rmk_slc_minimally_elliptic}, we make the following condition for  s.l.c. surfaces.
\begin{condition}\label{condition_star}
If an s.l.c. surface $S$ has the following surface singularity $(S,x)$:   a simple elliptic singularity, a cusp or a 
degenerate cusp singularity, or the $\zz_2, \zz_3, \zz_4, \zz_6$ quotients of the simple elliptic singularity, and the $\zz_2$ quotient of a cusp or a degenerate cusp singularity, 
then $(S,x)$ has  embedding dimension 
at most $4$.  
\end{condition}

\subsection{The moduli stack of s.l.c. surfaces}\label{subsubsec_moduli_functor_G}

We define the moduli functor of s.l.c. surfaces. 
We still fix $K^2, \chi, N\in\zz_{>0}$. 
Let $$M_N:=\overline{M}_{K^2, \chi, N}:  \Sch_{\mathbf{k}}\to \text{Groupoids}$$
be the moduli functor 
sending 
\begin{equation}\label{eqn_condition_functor_M}
T\mapsto  \left\{(f:\sS\to T) \left| \begin{array}{l}
  \text{$\bullet$  $\sS\stackrel{f}{\rightarrow} T$ is a  $\qq$-Gorenstein deformation} \\
  \text{~ family of stable s.l.c. surfaces;} \\
  \text{$\bullet$ Conditions (1)-(5) hold for each geometric fiber;} \\
   \text{$\bullet$ For each geometric point $t\in T$,  we have}\\
  ~\omega^{[N]}_{\sS/T}\otimes \mathbf{k}(t)\to ~\omega_{\sS_t}^{[N]}\text{~ is an isomorphism, where}\\
   ~\omega^{[N]}_{\sS/T}=j_*(\omega^{\otimes N}_{\sS^0/T}),  \text{and~} j: \sS^0\to \sS
   \text{~ is the inclusion of}\\
   \text{~ the locus where $f$ is Gorenstein.}
     \end{array}  \right\}\right. 
\end{equation}
modulo equivalence. The Conditions (1)-(5) above  are given by
 \begin{enumerate}
 \item each fiber of $f: \sS\to T$ is a reduced projective surface; 
 \item  each $S_t$ is connected with only s.l.c. singularities;
 \item  the sheaf $\omega_{S_t}^{[N]}$ which is defined by $\omega_{S_t}^{[N]}=j_*(\omega_{(S_t)^0}^{\otimes N})$ and 
 $j: (S_t)^0\to S_t$ is the inclusion of Gorenstein locus of $S_t$, is an ample line bundle;
 \item $K_{S_t}^{2}=\frac{1}{N^2}(\omega_{S_t}^{[N]}\cdot \omega_{S_t}^{[N]})=K^2$ for any $t\in T$;
 \item $\chi(\sO_{S_t})=\chi$ for $t\in T$.
 \end{enumerate}
 
 We have that 
 
\begin{thm}\label{thm_moduli_stack_G}
 When fixing $K^2, \chi, N\in\zz_{>0}$,  the functor $M_N$ is represented by a  Deligne-Mumford stack $M_N:=M_{K^2,\chi,N}$ of finite type over $\mathbf{k}$.   Suppose that  $N>0$ is large divisible enough, then the stack 
 $M:=M_{K^2,\chi}:=M_{K^2,\chi,N}$ is a proper Deligne-Mumford stack with projective coarse moduli space.
\end{thm}
\begin{proof}
From \cite[Corollary 5.7]{Kollar-Shepherd-Barron}, 
the functor $M_N$ is coarsely represented by a separated algebraic space  of finite type.  
From \cite[Theorem 1.1, Remark 6.3]{KP15},  $M_N$ is a proper Deligne-Mumford stack if $N$ is sufficiently large enough. 

We provide a proof in detail for $M=M_N$ when $N$ is sufficiently large and divisible.     
From \cite{Alexeev2}, \cite[Theorem 1.1]{HMX14}, after fixing the data $K^2, \chi$,  any $\qq$-Gorenstein  family of s.l.c. surfaces with fixed volume is bounded,  therefore  there exists a uniform bound  $N>0$ such that 
$\omega_{\sS/T}^{[N]}$ is invertible for any flat $\qq$-Gorenstein family $\sS\to T$ of s.l.c. surfaces.  
Note that \cite{Alexeev2} did the case of surfaces which is exactly what we want.  \cite[Theorem 1.1]{HMX14} proved the case of higher dimensional log general type varieties. 
Therefore,  from \cite[\S 4.21]{DM},  to prove $M$ is a Deligne-Mumford stack, 
  one needs to show that $M$ has representable and unramified diagonal, and there is a smooth \'etale surjection from a scheme of finite type to $M$.  
  
We first show that the diagonal morphism $M\to M\times_{\mathbf{k}}M$ is representable and unramified.    Let 
$(f: \sS\to T), (f^\prime: \sS^\prime\to T)$ be two objects in $M_N(T)$.  It is sufficient to show that the isomorphism functor 
$\mathbf{Isom}_{T}(\sS, \sS^\prime)$ is represented by a quasi-projective group scheme over $T$.  But this is just from \cite[Proposition 6.8]{KP15}. 
Since we only consider stable surfaces (while \cite{KP15} studied the more general case of log stable varieties), 
 the  global line bundle $\mathscr{L}$ in \cite[Definition 6.2, Proposition 6.8]{KP15} for the family $(f:\sS\to T)$ is just the invertible sheaf $\omega_{\sS/T}^{[N]}$.  The first half of the proof in \cite[Proposition 6.8]{KP15} implies that the isomorphism functor $\mathbf{Isom}_{T}(\sS, \sS^\prime)$ is represented by a quasi-projective group scheme over $T$. 

To prove that there exists a smooth \'etale surjection from a scheme $\mathscr{C}$ of finite type to $M$,  from \cite[Proposition 6.11]{KP15}, 
we consider the Hilbert scheme $\Hilb_{K^2, \chi}$ parametrizing closed two dimensional subschemes in  a higher dimensional projective space  with the same Hilbert polynomial determined by the invariants 
$K^2, \chi$.   After fixing the necessary conditions for the stable  s.l.c. surfaces in $\Hilb_{K^2, \chi}$,  techniques in \cite[Theorem 10, Definition-Lemma 33]{Kollar_Husk} and \cite[Proposition 6.11]{KP15} imply that there exists a scheme 
$\mathscr{C}$ and a smooth \'etale morphism $\mathscr{C}\to M$. 
Thus, $M$ is a Deligne-Mumford stack of finite type over $\mathbf{k}$. 
  
If $N$ is large divisible enough, the properness of the stack $M$ is just from the boundedness  result of \cite[Theorem 1.1]{HMX14}.
Thus, from the Nakai-Moishezon criterion, for any family 
$(f: \sS\to T)$ of stable s.l.c. surfaces we need to show that,  for  a large divisible enough $N>0$,  the determinant  $\det(f_*\omega_{\sS/T}^{[N]})$ of the  pushforward   of the relative invertible sheaf 
$\omega_{\sS/T}^{[N]}$ is big.  This is obtained in   \cite[Theorem 7.1, Corollary 7.3]{KP15}.
From \cite[Theorem 1.1, Remark 6.3, Corollary 7.3]{KP15}, the Deligne-Mumford stack $M$ has a projective coarse moduli space.
\end{proof}

\section{Moduli stack of index one covers}\label{sec_moduli_index_one_cover}

In this section we construct an  obstruction theory on the moduli stack $M_N^{\ind}:=\overline{M}^{\ind}_{K^2, \chi, N}$ of   index one covers over one connected  component $M_N=\overline{M}_{K^2, \chi, N}$ of the moduli stack of s.l.c. surfaces. 
The obstruction theory is not  perfect in general, but  in some nice situation such that there is no higher obstruction spaces for the s.l.c. surfaces the obstruction theory is perfect. 

\subsection{The moduli space of   index one covers}\label{subsec_index_one_cover}
 
Recall from Section \ref{subsec_Q-Gorenstein_deformation}, a $\qq$-Gorenstein deformation family 
$$\sS\to T$$
of s.l.c. surfaces is the same as the  deformation $\SS\to T$ of the index one covering Deligne-Mumford stacks.   There is a canonical morphism $p: \SS\to \sS$ which make the following diagram 
\[
\xymatrix{
\SS\ar[rr]^{\pi}\ar[dr]&& \sS\ar[dl] \\
&T&
}
\]
commute.  The scheme  $\sS$ is the coarse moduli space of the Deligne-Mumford stack $\SS$. 
Thus, the canonical correspondence motivates us to define the moduli functor 
$$\sM_N^{\ind}=\overline{\sM}^{\ind}_{K^2, \chi, N}: \Sch_{\mathbf{k}}\rightarrow \mbox{Groupoids}$$
which sends 
$$T\mapsto \{f: \SS\to T\}$$
where  $\{f: \SS\to T\}$ represents the isomorphism classes of flat families of index one covering Deligne-Mumford stacks 
$\SS\to T$.  The coarse moduli space of the  family  $\{f: \SS\to T\}$ must satisfy the 
conditions in (\ref{eqn_condition_functor_M}).

\begin{thm}\label{thm_index_one_covering_stack}
The functor $\sM_N^{\ind}$ has representable and  unramified diagonal, therefore,  is represented by a fine Deligne-Mumford stack $M_N^{\ind}:=\overline{M}^{\ind}_{K^2,\chi,N}$.  Moreover,  there is a canonical  isomorphism 
$$f: M_N^{\ind}\to M_N.$$ 
The isomorphism  $f$ induces an isomorphism on the coarse moduli spaces. 

Fixing $K^2, \chi$, if $N$ is large divisible enough, then the stack $M^{\ind}:=M_N^{\ind}$ is a proper Deligne-Mumford stack with projective coarse moduli space, and the isomorphism 
$f: M^{\ind}\to M$ induces an isomorphism on the projective coarse moduli spaces. 
\end{thm}
\begin{proof}
We prove for the case $N$ is sufficiently large and divisible so that 
$\sM:=\overline{\sM}^{\ind}_{K^2,\chi,N}$. 
We first show that the diagonal morphism 
$$\sM^{\ind}\to \sM^{\ind}\times_{\mathbf{k}}\sM^{\ind}$$
is representable and unramified. 
Let $(f: \SS\to T)$ and $(f^\prime: \SS^\prime\to T)$ be two objects in $\sM^{\ind}(T)$, then the isomorphism functor of the two families 
$\textbf{Isom}_{T}(\SS,\SS^\prime)$ is represented by a quasi-projective group scheme 
$\Isom_{T}(\SS,\SS^\prime)$ over $T$.  We prove this statement here.  Let 
$(\overline{f}: \sS\to T)$ and $(\overline{f}^\prime: \sS^\prime\to T)$  be the $\qq$-Gorenstein  families of the corresponding s.l.c. surfaces over $T$. 
From the proof of \cite[Proposition 6.8]{KP15} and Theorem \ref{thm_moduli_stack_G}, the isomorphism functor $\textbf{Isom}_{T}(\sS,\sS^\prime)$ is represented by a quasi-projective group scheme 
$\Isom_{T}(\sS,\sS^\prime)$ over $T$.   The canonical morphisms $\SS\to \sS$ and $\SS^\prime\to \sS^\prime$ are maps to their  coarse moduli spaces.  
Consider  the following diagram
\[
\xymatrix{
\SS\ar[r]^{\cong}\ar[d]& \SS^\prime\ar[d]\\
\sS\ar[r]^{\cong}& \sS^\prime,
}
\]
any isomorphism $\SS\cong \SS^\prime$ induces an isomorphism $\sS\cong \sS^\prime$ on the coarse moduli spaces. 
 Any isomorphism $\sS\cong \sS^\prime$ of families of $\qq$-Gorenstein deformations implies the  isomorphism $\SS\cong \SS^\prime$. 
Therefore, the functor  
$\textbf{Isom}_{T}(\SS,\SS^\prime)$  is  represented by a quasi-projective group scheme $\Isom_{T}(\SS,\SS^\prime)$ and is also unramified over $T$  since its geometric fibers are finite (due to the automorphism group of each fiber 
$\SS_t$ is finite). 

From \cite[Proposition 6.11]{KP15} and Theorem \ref{thm_moduli_stack_G},  there is a cover $\varphi: \C\to M$ which is an \'etale  surjective morphism onto 
$M$ where $\C$ is a scheme of finite type.  This is because $M$ is a projective Deligne-Mumford stack. 
Also from the construction of the moduli functor there is a canonical morphism 
$f: M^{\ind}\to M$ of stacks, which sends every flat family $f: \SS\to T$ of index one covering Deligne-Mumford stacks to the corresponding 
$\qq$-Gorenstein deformation family $\sS\to T$ of s.l.c. surfaces. 

We construct the following diagram
\begin{equation}\label{eqn_diagram_surjection_cover}
\xymatrix{
\C\ar[r]^{\varphi^\prime}\ar[dr]_{\varphi}& M^{\ind}\ar[d]^{f}\\
&M.
}
\end{equation}
For each $T=\spec(A)\to \C$, 
the  $\qq$-Gorenstein deformation family $\sS\to T$ of the s.l.c. surfaces and the  corresponding family $\SS\to T$  of 
index one covering Deligne-Mumford stacks
 induce the following diagram
\[
\xymatrix{
T\ar[r]^{\varphi^\prime}\ar[dr]_{\varphi}& M^{\ind}\ar[d]^{f}\\
&M.
}
\]
This induces the diagram (\ref{eqn_diagram_surjection_cover}).  Thus,  taken as Deligne-Mumford stacks,  $M^{\ind}$ and $M$ share the same cover $\C$.

Now we show that the morphism $f: M^{\ind}\to M$ is proper by the valuative criterion for properness. 
Look at  the  following diagram
\[
\xymatrix{
\spec(K)\ar[r]\ar[d]& M^{\ind}\ar[d]^{f^{\ind}}\\
\spec(R)\ar@{-->}[ur]\ar[r]& M
}
\]
where $R$ is a valuation ring and $K$ is the field of fractions, then  any family 
$\{\sS\to \spec(R)\}$ of s.l.c. surfaces corresponds to  a unique flat  family $\{\SS\to \spec(R)\}$ of 
index one covering Deligne-Mumford stacks and the above dotted arrow exists and is unique.  
Thus, $f: M^{\ind}\to M$ is proper.

The morphism $f: M^{\ind}\to M$ is also quasi-finite, since for each geometric point 
$S=\spec(\mathbf{k})\in M$, there is a unique $\SS\in M^{\ind}$ in the preimage.   Therefore, the morphism $f: M^{\ind}\to M$ is finite. 
To prove that the Deligne-Mumford stack $M^{\ind}$ is isomorphic to the Deligne-Mumford stack $M$,  it is sufficient to show that 
for any s.l.c. surface $S$,  the automorphism group $\Aut(S)$ is isomorphic to the automorphism group 
$\Aut(\SS)$ of its index one covering Deligne-Mumford stack $\SS\to S$.  
From the canonical construction of the index one cover in \S \ref{subsec_slc_moduli}, any automorphism 
$\sigma: \SS\stackrel{\sim}{\rightarrow}\SS$ of the index one covering Deligne-Mumford stack $\SS$ induces an automorphism 
$\overline{\sigma}: S\stackrel{\sim}{\rightarrow} S$.  Thus, we get a map 
$$g: \Aut(\SS)\to \Aut(S).$$
Conversely,  for any automorphism $\overline{\sigma}: S\stackrel{\sim}{\rightarrow} S$, from the canonical construction of the index one cover, we get an $\sigma: \SS\stackrel{\sim}{\rightarrow}\SS$.  Thus, we get a map
$$h:  \Aut(S)\to \Aut(\SS).$$
The canonical construction of the index one cover implies that 
$g\circ h=1, h\circ g=1$.  Thus we get $\Aut(S)\cong \Aut(\SS)$.

The canonical isomorphism $f: M^{\ind}\to M$  induces a bijection on the coarse moduli spaces since  the index one covering Deligne-Mumford stack $\SS$ has coarse moduli space 
$S$. 
If $N$ is large divisible enough, then the stack $M$ is a proper Deligne-Mumford stack with projective coarse moduli space. Therefore the stack $M^{\ind}$ is a proper Deligne-Mumford stack with projective coarse moduli space and the isomorphism 
$f: M^{\ind}\to M$  induces an isomorphism on the projective coarse moduli spaces. 
\end{proof}

\begin{rmk}
We  point out that in the paper \cite{AH},  Abramovich-Hassett have studied the moduli functor of index one covers and constructed the moduli stack of the index one covers of stable  varieties. 
\end{rmk}

\begin{cor}\label{cor_lci_index_one}
Let $M$ be a connected component of the moduli stack of stable general type surfaces with invariants $K^2, \chi, N$.  
If each s.l.c. surface $S$ in $M$ has only l.c.i. singularities, then the moduli  stack $M^{\ind}$ of index one covers  is just the moduli stack $M$. 
\end{cor}
\begin{proof}
This is a special case. 
If an s.l.c. surface has at most l.c.i. singularities, it is Gorenstein and the dualizing sheaf $\omega_{S}$ is a line bundle.  From the construction in Section \ref{subsec_Q-Gorenstein_deformation}, the index one covering Deligne-Mumford stack 
$\SS$ is just $S$.   Therefore, from the construction of the moduli functor 
$M^{\ind}$,  $M^{\ind}$ is the same as $M$ as Deligne-Mumford stacks. 
\end{proof}

\subsection{Obstruction theory}\label{subsec_Obstruction_Theory}

Let $M$ be one connected component of the moduli stack of  s.l.c. surfaces 
with fixed invariants $K_S^2=K, \chi(\sO_S)=\chi$ and  $N\in\zz_{>0}$ as in Theorem \ref{thm_moduli_stack_G}.  Still from Theorem \ref{thm_moduli_stack_G}
there exists a universal family for the moduli stack
$$p: \scM\to M,$$
since the stack is a fine moduli stack.  From Theorem \ref{thm_index_one_covering_stack}, there also exists a universal family 
$$p^{\ind}: \scM^{\ind}\to M^{\ind},$$ 
and a commutative diagram
\begin{equation}\label{eqn_diagram_index_one_Mind_M}
\xymatrix{
\scM^{\ind}\ar[r]^{p^{\ind}}\ar[d]_{\widetilde f}& M^{\ind}\ar[d]^{f}\\
\scM\ar[r]^{p}& M.
}
\end{equation}

\begin{lem}\label{lem_pInd}
The universal family $p^{\ind}: \scM^{\ind}\to M^{\ind}$ is projective, flat and relatively Gorenstein.  Therefore the relative dualizing sheaf $\omega_{\scM^{\ind}/M^{\ind}}$ is invertible. 
\end{lem}
\begin{proof}
Since $p^{\ind}$ is a universal family for the moduli stack $M^{\ind}$, it is flat and projective.   The relative dualizing sheaf $\omega_{\scM^{\ind}/M^{\ind}}$ is invertible since it gives the dualizing sheaf 
$\omega_{\SS_t}$ of the canonical index one covering Deligne-Mumford  stack $\SS_t$ for each geometric point $t\in M^{\ind}$ and $\omega_{\SS_t}$   is invertible (due to $\SS_t$ Gorenstein).
\end{proof}

\begin{rmk}
In general,  for the universal family $p: \scM\rightarrow M$, the relative dualizing sheaf 
$\omega_{\scM/M}$ is not a line bundle since the relative dualizing sheaf $\omega_{\scM/M}$ is not a line bundle on the 
non-Gorenstein locus.   
\end{rmk}

Let $\ll^{\bullet}_{\scM^{\ind}/M^{\ind}}$ be the relative cotangent complex of $p^{\ind}$ and 
$\omega^{\ind}:=\omega_{\scM^{\ind}/M^{\ind}}[2]$. 
We consider 
$$E^{\bullet}_{M^{\ind}}:=Rp^{\ind}_{*}\left(\ll^{\bullet}_{\scM^{\ind}/M^{\ind}}\otimes\omega^{\ind}\right)[-1].$$
Here the relative dualizing sheaf $\omega_{\scM^{\ind}/M^{\ind}}$  satisfies the property
$$\omega_{\scM^{\ind}/M^{\ind}}|_{(p^{\ind})^{-1}(t)}\cong \omega_{\SS_t},$$
where the dualizing sheaf $\omega_{\SS_t}$ of the index one covering Deligne-Mumford stack $\SS_t\to S_t$, which is  locally given by $\omega_{S_t}^{[r]}$  at a singularity germ ($r$ is the index of the singular germ),  
is invertible. 

\begin{thm}\label{thm_OT_slc}
The complex $E^{\bullet}_{M^{\ind}}$ defines an  obstruction theory (in the sense of Behrend-Fantechi)
$$\phi^{\ind}: E^{\bullet}_{M^{\ind}}\to \ll_{M^{\ind}}^{\bullet}$$
induced by the Kodaira-Spencer map $\ll^{\bullet}_{\scM^{\ind}/M^{\ind}}\to (p^{\ind})^{*}\ll_{M^{\ind}}^{\bullet}[1]$. 
\end{thm}
\begin{proof}
From Lemma \ref{lem_pInd}, the universal family $p^{\ind}: \scM^{\ind}\to M^{\ind}$ is a projective, flat, relatively Gorenstein morphism between Deligne-Mumford stacks.  
Also $M^{\ind}$ is a fine moduli stack.  Thus, $\phi^{\ind}: E^{\bullet}_{M^{\ind}}\to \ll_{M^{\ind}}^{\bullet}$ gives an obstruction theory from 
Theorem  \ref{thm_BF_prop6.1} (also see \cite[Proposition 6.1]{BF}).  For completeness of the analysis of local deformation and obstruction theory of 
s.l.c. surfaces, we include the details here. 

The basic observation is that the complex 
$$\widetilde{E}^{\bullet}_{M^{\ind}}:=Rp^{\ind}_{*}\left(\ll^{\bullet}_{\scM^{\ind}/M^{\ind}}\otimes\omega^{\ind}\right),$$
when restricted to a point 
$t\in M^{\ind}$,  calculates the cohomology spaces
$H^*(\SS_t, T_{\SS_t})=T^*_{\QG}(S_t,\sO_{S_t})$ for the index one covering Deligne-Mumford stack $\SS_t$. 
Since it is of general type, $\dim H^0(\SS_t, T_{\SS_t})=0$. Over a point 
$t\in M^{\ind}$, the complex $\widetilde{E}^{\bullet}_{M^{\ind}}$ gives 
$$\widetilde{E}^{\bullet}_{M^{\ind}}|_{t}=Rp^{\ind}_{*}(\ll^{\bullet}_{\SS_t}\otimes\omega_{\SS_t}[2]), $$
and 
$$\left(\widetilde{E}^{\bullet}_{M^{\ind}}|_{t}\right)^{\vee}=Rp^{\ind}_{*}(\ll^{\bullet}_{\SS_t}, \sO_{\SS_t}).$$
Thus $\left(\widetilde{E}^{\bullet}_{M^{\ind}}|_{t}\right)^{\vee}$ is given by 
$p^{\ind}_{*}\sE xt^i(\ll^{\bullet}_{\SS_t}, \sO_{\SS_t})$ which was studied in \cite[\S 3]{Hacking}, 
Proposition \ref{prop_deformation_S_SS_G}  and Proposition \ref{prop_local_deformation_obstruction_slc}.  Therefore, 
the cohomology  spaces of $\left(\widetilde{E}^{\bullet}_{M^{\ind}}|_{t}\right)^{\vee}$ give
$$T_{\QG}^1(S_t, \sO_{S_t});  \quad T_{\QG}^2(S_t, \sO_{S_t})$$
in Proposition \ref{prop_local_deformation_obstruction_slc}.

If we have a diagram
\[
\xymatrix{
S_t\ar[r]\ar[d]& \scM^{\ind}\ar[d]^{p^{\ind}}\\
t=\spec(\mathbf{k})\ar[r]& M^{\ind},
}
\]
then from Proposition \ref{prop_local_deformation_obstruction_slc} the first order infinitesimal $\qq$-Gorenstein deformation of $\spec(\mathbf{k})\in M^{\ind}$ (i.e., the $\qq$-Gorenstein deformation of $S_t$)
 is given by 
$T_{\QG}^1(S_t, \sO_{S_t})$, and the obstruction is given by 
$T_{\QG}^2(S_t, \sO_{S_t})$. 
There may exist higher obstruction spaces $T_{\QG}^i(S_t, \sO_{S_t})$ for $i\ge 3$.
We make this more precise following  Proposition \ref{prop_local_deformation_obstruction_slc}. 
Let $A$ be a finitely generated Artinian local $\mathbf{k}$-algebra, and $\sS_A/A$ be a $\qq$-Gorenstein deformation of 
$S$ over $A$. Let $\overline{A}\to A$ be an infinitesimal extension of $A$ with kernel 
$J$.  We let $\overline{\mathfrak{m}}$ be the maximal ideal of  $\overline{A}$ and assume that 
$\overline{\mathfrak{m}}\cdot J=0$ ($J$ is a $A/\overline{\mathfrak{m}}=\mathbf{k}$ space). Then 
there is an obstruction class 
$$\ob(\sS_A/A, \overline{A})\in T_{\QG}^2(S, \sO_{S})\otimes J,$$
such that $\ob(\sS_A/A, \overline{A})=0$ if and only if there exists a $\qq$-Gorenstein deformation 
$\sS_{\overline{A}}$ of $\sS_A$ over $\overline{A}$. 
Moreover, if  $\ob(\sS_A/A, \overline{A})=0$, then the isomorphism classes of such deformations form a 
torsor under $T_{\QG}^1(S, \sO_{S})\otimes J$.

One can make this argument into a family by considering a scheme 
$T=\spec(A)\to M^{\ind}$, and the diagram
\[
\xymatrix{
\scM_T\ar[r]^{g}\ar[d]_{q}& \scM^{\ind}\ar[d]^{p^{\ind}}\\
T\ar[r]^{f}& M^{\ind}.
}
\]
Let $T\to \overline{T}$ be a square zero extension with ideal sheaf $J$. The obstruction to extending 
$\scM_T$ to a flat family over $\overline{T}$ lies in 
$\Ext^2(\ll^{\bullet}_{\scM_T/T}, q^*J)$
and if the extensions exist, they form a torsor under $\Ext^1(\ll^{\bullet}_{\scM_T/T}, q^*J)$. 
Since 
$\ll^{\bullet}_{\scM_T/T}=g^*\ll^{\bullet}_{\scM^{\ind}/M^{\ind}}$, and $p^{\ind}$ is flat,  we have that 
\begin{align*}\Ext_{\sO_{\scM_T}}^i(\ll^{\bullet}_{\scM_T/T}, q^*J)
&=\Ext_{\sO_{\scM^{\ind}}}^i(\ll^{\bullet}_{\scM^{\ind}/M^{\ind}}, Rg_*q^*J) \\
&=\Ext_{\sO_{\scM^{\ind}}}^i(\ll^{\bullet}_{\scM^{\ind}/M^{\ind}}, (p^{\ind})^*Rf_*J).
\end{align*}
Thus, 
$$\Ext_{\sO_{\scM^{\ind}}}^i(\ll^{\bullet}_{\scM^{\ind}/M^{\ind}}, (p^{\ind})^*Rf_*J)=
\Ext_{\sO_{M^{\ind}}}^{i-1}(E^{\bullet}_{M^{\ind}}, Rf_*J)=
\Ext_{\sO_{M^{\ind}}}^{i-1}(f^*E^{\bullet}_{M^{\ind}}, J),$$
where for the first isomorphism, we use Grothendieck duality since $(p^{\ind})^{!}(\sO_{M^{\ind}})$ is the dualizing sheaf 
$\omega_{\scM^{\ind}/M^{\ind}}$ which is invertible. 

Since $p^{\ind}: \scM^{\ind}\to M^{\ind}$ is a universal family for the moduli stack $M^{\ind}$, the Kodaira-Spencer map
$\ll^{\bullet}_{\scM^{\ind}/M^{\ind}}\to (p^{\ind})^*\ll_{M^{\ind}}^{\bullet}[1]$ defines a morphism 
$$\phi^{\ind}: E_{M^{\ind}}^{\bullet}\to \ll_{M^{\ind}}^{\bullet}.$$
From the above analysis, this morphism satisfies Condition (3) in Theorem \ref{thm_obstruction_theory_equ_conditions}. 
Therefore, 
$\phi^{\ind}$ defines an obstruction theory for $M^{\ind}$ in the sense of Behrend-Fantechi. 
\end{proof}

\section{Moduli stack of $\lci$ covers}\label{sec_universal_covering_DM}

In this section  we  construct  the moduli stack  $M_N^{\lci}:=\overline{M}^{\lci}_{K^2, \chi,N}$  of  $\lci$ covers  over the moduli stack  $M$ such that there is a perfect obstruction theory on 
$M_N^{\lci}$.  

\subsection{Universal abelian cover of s.l.c.  surface germs}\label{subsubsec_universal_discriminant}

Recall from Remark \ref{rmk_slc_minimally_elliptic} in \S \ref{subsubsec_higher_cohomology}, let $S$ be an s.l.c. surface and 
$\pi: \SS\to S$ be the corresponding index one covering Deligne-Mumford stack.  Except l.c.i. singularities,  the germs on the  index one covering Deligne-Mumford stack 
$\SS$ may have simple elliptic singularities, cusp or degenerate cusp singularities of embedding dimension $\ge 5$.
Locally,  the germ singularity is of the form $[Z/\mu_r]$, where $(Z,0)$ is a germ singularity which is a simple elliptic singularity, a cusp or a degenerate cusp singularity and $r$ is the index. 
Note that $r=1, 2, 3, 4, 6$. 

From the classification result in \cite[Theorem 4.24]{Kollar-Shepherd-Barron},  we consider the simple elliptic singularity, the cusp or the degenerate cusp singularity $(S,0)$,  and the  $\zz_2, \zz_3, \zz_4, \zz_6$-quotient of a simple elliptic singularity 
$(S,0)$, 
the $\zz_2$-quotient of a cusp singularity or a degenerate cusp singularity $(S,0)$.   The $\qq$-Gorenstein deformation of $(S,0)$ is equivalent to the $\zz_r$-equivariant deformation of $(Z,0)$. 

Let us focus on the surface singularity germ $(S,0)$. 
Let 
\begin{equation}\label{eqn_resolution_Z}
\sigma: X\to S
\end{equation}
be a good resolution and $A=\cup_{i=1}^{n}A_i$ be the decomposition of exceptional set $\sigma^{-1}(0)=A$
such that 
$A$ is a divisor having only simple normal crossings. A divisor supported in $A$ is called a cycle.  Let $\Sigma$ be the link of $(S,0)$ which is, by definition, the boundary $\partial U$ of a small neighborhood $U$  of the singularity $0$.  The link 
$\Sigma$ is an oriented $3$-manifold over the field  $\rr$ of real numbers .
The neighborhood  $U$ can  be made to be a tubular neighborhood of the exceptional divisor so that $\partial U=\Sigma$ is the link of the singularity.  
This can be obtained by plumbing theory of surface singularities in \cite{Neumann}. 
Then, we have that 
$$H_2(U,\zz)\cong \zz^n\subset H_2(U,\qq)\cong \qq^n,$$
where $n$ is the number of exceptional curves in $A$.  Let $\langle, \rangle$ be the intersection form on these groups and define
$$H_2(U)^{\#}=\{v\in H_2(U,\qq): \langle v, w\rangle\in \zz \text{~for all~}w\in H_2(U,\zz)\}.$$
Then the embedding $H_2(U,\zz)\to H_2(U)^{\#}$ can be identified with the map $H_2(U,\zz)\to H_2(U,\Sigma)$. So the long exact sequence in homology identifies the 
discriminant group 
$$G:=H_2(U)^{\#}/H_2(U,\zz)$$
with the torsion subgroup  $H_1(\Sigma,\zz)_{\tor}$ of $H_1(\Sigma,\zz)$.  The intersection form $\langle, \rangle$ induces on $G$ a natural non-singular pairing:
$$G\otimes G\to \qq/\zz; \quad  v\otimes w\mapsto \langle v, w \rangle/\zz$$
which is the torsion link pairing of $\Sigma$.  

If $K\subset G$ is a subgroup, then there is an induced non-singular pairing
$$K\otimes (G/K^{\perp})\to \qq/\zz$$
where $K^{\perp}$ is the orthogonal complement of $K$ under the pairing. 
The group $G/K^{\perp}$ is canonically isomorphic to the dual $\hat{K}=\Hom(K, \qq/\zz)$ and is non-canonically isomorphic to $K$ itself. 

If $\Sigma$ is a rational homology sphere,  then the universal abelian cover of $\Sigma$ is the Galois cover of $\Sigma$ determined by the natural homomorphism 
$\pi_1(\Sigma)\to H_1(\Sigma)=G$.  Thus, any subgroup $K\subset G$ determines an abelian cover of $\Sigma$;  i.e., the Galois cover with covering transformation group $G/K$. 
The Galois cover corresponding to $K^{\perp}$ is called the {\em dual cover} for $K$, with transformation group $G/K^{\perp}$. The dual cover for $G$ is thus the universal abelian cover. 

Let us  consider the  $\zz_2, \zz_3, \zz_4, \zz_6$-quotient of a simple elliptic singularity $(S,0)$, or
the $\zz_2$-quotient of a cusp singularity.   
Then 
$A$ is a tree of rational curves since the $\zz_r$-quotient of simple elliptic singularity and cusp singularity are rational singularities. 
An explicit $\zz_2$-action on cusps was given in \cite{NW}, and the $\zz_2, \zz_3, \zz_3, \zz_6$ actions on a simple elliptic singularity were given in 
\cite[\S 5.2]{Kollar-Shepherd-Barron}, \cite[\S 9.6]{Kawamata_Ann}.   All of these singularities are log-canonical. In particular,  a cyclic group 
quotient of log-canonical singularity is a rational singularity. 

For such an s.l.c. germ $(S,0)$, its link $\Sigma$ is a rational homology sphere.  
The group $G=H_1(\Sigma,\zz)$ is  a finite abelian group.   From \cite{NW2}, we take 
$$(\widetilde{S},0)\to (S,0)$$
to be 
the universal abelian cover, where the topology of the cover is determined by the link $\Sigma$. 
Let $(Z,0)\to (S,0)$ be the index one cover of the singularity germ $(S,0)$ such that $[Z/\zz_r]\cong S$ for $r=2,3, 4,6$.  Then the universal abelian cover 
$(\widetilde{S},0)\to  (S,0)$ factors through the index one cover 
\begin{equation}\label{eqn_universal_index}(\widetilde{S},0)\to (Z,0)
\end{equation}
since $(Z,0)\to (S,0)$ is an abelian cover.

The deformation of $(S,0)$ can be given by the $G$-equivariant deformation of $(\widetilde{S},0)$. Thus we have 

\begin{thm}\label{thm_minimally_elliptic_universal2}
If $(S,0)$ is the $\zz_2, \zz_3, \zz_4$, or $\zz_6$ quotient of  a simple elliptic singularity,  or the $\zz_2$ quotient of a cusp or a degenerate cusp singularity  germ, then there exists the universal abelian cover 
$(\widetilde{S},0)$ with transformation group $G$.  Moreover, 
the $G$-equivariant  deformations of $(\widetilde{S}, 0)$ gives $\qq$-Gorenstein deformations of $(S,0)$. In particular, there exists a $G$-equivariant one-parameter smoothing or deformation of $(\widetilde{S},0)$. 
\end{thm}
\begin{proof}
The cases of the $\zz_2, \zz_3, \zz_4, \zz_6$ quotients of a  simple elliptic singularity and the  $\zz_2$ quotient of a cusp are from \cite{NW}, \cite{NW2}, and (\ref{eqn_universal_index}). The $\zz_2$-quotient of degenerate cusp is given in \cite[\S 9.6]{Kawamata_Ann}, where the the degenerate cusp only has two irreducible components.  In this case we consider the following diagram
\[
\xymatrix{
(\widetilde{S}^{\norm})\ar[r]\ar[d]& S^{\norm}=S_1\sqcup S_2\ar[d]\\
(\widetilde{S})\ar[r] & S, 
}
\]
where $S^{\norm}$ is the normalization of $S$, and the two components $S_i$ have cyclic quotient singularities.  From  \cite{NW}, \cite{NW2},  $\widetilde{S}^{\norm}\to S^{\norm}$ is the universal abelian cover.  Then, 
$\widetilde{S}$ is obtained from  $\widetilde{S}^{\norm}$ by identifying the double curves.    We know that  $\widetilde{S}^{\norm}$ is l.c.i., so is $\widetilde{S}$. 
\end{proof}

\begin{rmk}\label{rmk_quotient_elliptic_cusp}
Suppose that $(S,0)$ is the $\zz_2, \zz_3, \zz_4, \zz_6$ quotient of a simple elliptic singularity, or the $\zz_2$ quotient of a cusp singularity. Let $(\widetilde{S},0)$ be the universal abelian cover.   It is interesting to study if any  
$\qq$-Gorenstein deformation of $(S,0)$ gives a $G$-equivariant  deformations of $(\widetilde{S}, 0)$.  

For instance, in the case of   $\zz_2$-quotient of  simple elliptic singularity $(S,0)$,  if the exceptional smooth elliptic curve $E$ has self-intersection number $\leq 8$,  \cite{Simonetti} proves that $(S,0)$ always admits a 
$\zz_2$-equivariant smoothing.  It is interesting to study if the universal abelian cover  $(\widetilde{S},0)$  of the quotient elliptic singularity admits a $G$-equivariant smoothing. 
\end{rmk}

\begin{example}\label{example_quotient_simple_elliptic}
Following \cite{NW_GT}, we give an interesting example of the $\zz_2$-quotient of simple elliptic singularity $(S,0)$ whose resolution graph $\Gamma$ is given by 
\begin{equation}\label{eqn_simple_graph}
\xymatrix@R=7pt@C=24pt@M=0pt@W=0pt@H=0pt{
  \overtag{\bullet}{-2}{8pt}\lineto[dr] && 
  \overtag{\bullet}{-2}{8pt}\lineto[dl]\\
  &\overtag{\bullet}{-d}{8pt}\lineto[ur]\lineto[dr]&\\
  \overtag{\bullet}{-2}{8pt}\lineto[ur] && 
  \overtag{\bullet}{-2}{8pt}\lineto[ul]}
  \end{equation}
The discriminant group $D(\Gamma)$ is our finite abelian group $G$, which has order $16d-32$ (which is given by the negative of the determinant of the intersection matrix of $\Gamma$). 
  
The corresponding splice diagram in \cite[\S 2]{NW_GT} is given by 
\[\label{eqn_simple_graph_2}
\xymatrix@R=7pt@C=24pt@M=0pt@W=0pt@H=0pt{
  \overtag{\bullet}{}{8pt}\lineto[dr]^{2} && 
  \overtag{\bullet}{}{8pt}\lineto[dl]_{2}\\
  &\overtag{\bullet}{}{8pt}\lineto[ur]\lineto[dr]&\\
  \overtag{\bullet}{}{8pt}\lineto[ur]^{2} && 
  \overtag{\bullet}{}{8pt}\lineto[ul]_{2}}
  \]
  There are four leaves in the splice diagram, which correspond to $\cc^4$ with variables $z_1, z_2, z_3, z_4$.
  Then from \cite[\S 8]{NW_GT}, the splice diagram equations are of the form
  $$
  \sum_{j=1}^{4}a_{ij}z_j^{2}+H_i(z)=0,
  $$
  for $i=1, 2$, where $(a_{ij})$ is a $2\times 4$ matrix such that all the maximal minors have full rank, and $H_i(z)$ is a convergent series in $z_i$. Let $(\widetilde{S},0)$ be the complete intersection singularity determined by these $2$-equations.  Then $G$ acts on $\widetilde{S}$ with quotient  $(S,0)$ and $(\widetilde{S},0)$ is the universal abelian cover.
\end{example}

\begin{example}\label{example_quotient_cusp}
We provide an interesting example of the $\zz_2$-quotient-cusp in \cite{NW}. 
Let $(S,0)$ be a quotient-cusp singularity.   It is the $\zz_2$-quotient  of the cusp surface singularity $(Z,0)$ whose resolution graph is given by
\begin{equation}\label{eqn_cusp_graph}
\xymatrix@R=9pt@C=24pt@M=0pt@W=0pt@H=0pt{
  &\overtag{\bullet}{-e_2}{8pt}\dashto[r]
  &&&\overtag{\bullet}{-e_{k-1}}{8pt}\dashto[l]\\
  \lefttag{\bullet}{2-2e_1}{8pt}\lineto[ur]\lineto[dr] &&&&&
  \righttag{\bullet}{2-2e_k}{8pt}\lineto[ul]\lineto[dl]\\
  &\overtag{\bullet}{-e_2}{8pt}\dashto[r]
  &&&\overtag{\bullet}{-e_{k-1}}{8pt}\dashto[l]}
  \end{equation}
where $k\ge 2$, $e_i\ge 2$ and some $e_j>2$. 
The quotient-cusp singularity  $(S,0)$ has resolution graph
\begin{equation}\label{eqn_cusp_graph2}
\xymatrix@R=7pt@C=24pt@M=0pt@W=0pt@H=0pt{
  \overtag{\bullet}{-2}{8pt}\lineto[dr] && &&&
  \overtag{\bullet}{-2}{8pt}\lineto[dl]\\
  &\overtag{\bullet}{-e_1}{8pt}\lineto[r]
  &\overtag{\bullet}{-e_2}{8pt}\dashto[r]&\dashto[r]&
  \overtag{\bullet}{-e_k}{8pt}&&\\
  \overtag{\bullet}{-2}{8pt}\lineto[ur] && &&&
  \overtag{\bullet}{-2}{8pt}\lineto[ul]}
  \end{equation}
  
There is an associated  matrix
$$B=\mat{cc}
a&b\\
c&d
\rix=B(e_1-1, e_2, \cdots, e_{k-1}, e_k-1)$$
where 
$$
B(e_1-1,e_2, \cdots, e_{k-1}, e_k-1)=
\mat{cc}
0&1\\
-1&0
\rix\mat{cc}
0&-1\\
1&e_k-1
\rix\cdots \mat{cc}
0&-1\\
1&e_1-1
\rix.
$$  

From \cite[Theorem 5.1]{NW},  the universal abelian $\lci$ cover $(\widetilde{S}, 0)\to (S,0)$ has transformation abelian group $G$ with order $16b$.   
Let $\zeta$ be a primitive $4b$-th root of unity.  We  consider the following diagonal matrices:
$$A_1=\Diag[-\zeta^a, \zeta^a, \zeta, \zeta]$$
$$A_2=\Diag[\zeta^a, -\zeta^a, \zeta, \zeta]$$
$$A_3=\Diag[\zeta, \zeta, -\zeta^d, \zeta^d]$$
$$A_4=\Diag[\zeta, \zeta, \zeta^d, -\zeta^d].$$
Then the finite abelian group is $G=\langle A_1, A_2, A_3, A_4\rangle$, which has order $16b$.  The group structure of 
$G$ depends on the parity of $c$, see \cite[Theorem 5.1]{NW}. 

The local equations of $(\widetilde{S}, 0)$ are given by:
$$x^2+y^2=u^{\alpha}v^{\beta}; \quad   u^2+v^2=x^{\gamma}y^{\delta},$$
where $\alpha, \beta, \gamma, \delta\ge 0$ satisfy the conditions
$$\alpha+\beta=2a;  \quad  \gamma+\delta=2d; \quad  \alpha\equiv \beta\equiv \gamma\equiv \delta\equiv c  ~\left(\mod ~2\right). $$

The resolution graph of the universal abelian cover $(\widetilde{S},0)$ is given by
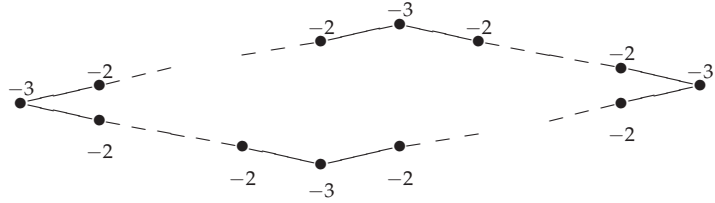
\begin{equation}\label{eqn_resolution_universal}
  \xymatrix@R=2pt@C=24pt@M=0pt@W=0pt@H=0pt{
  &&&&&\overtag{\bullet}{-3}{6pt}\\
  &&&&\overtag{\bullet}{-2}{6pt}\lineto[ur] &&\overtag{\bullet}{-2}{6pt}\lineto[ul]\\
  &&&\dashto[ur] &&&&\dashto[ul]\\
  && &&&&&&\overtag{\bullet}{-2}{6pt}\dashto[ul]\\
  &\overtag{\bullet}{-2}{6pt}\dashto[ur]
  &&&&&&&&\overtag{\bullet}{-3}{6pt}
  \lineto[ul]\\
  \overtag{\bullet}{-3}{6pt}\lineto[ur]
  &&&&&&&&\undertag{\bullet}{-2}{2pt}\lineto[ur]\\
  &\undertag{\bullet}{-2}{2pt}\lineto[ul] &&&&&&\dashto[ur]\\
  &&\dashto[ul] &&&&\\
  &&&\undertag{\bullet}{-2}{2pt}\dashto[ul]
  &&\undertag{\bullet}{-2}{2pt}\dashto[ur]\\
  &&&&\undertag{\bullet}{-3}{0pt}\lineto[ul]\lineto[ur]\\ \\ \\}\end{equation}
where
the four strings of $-2$'s are lengths $2a-3$, $2d-3$, $2a-3$, and
$2d-3$ if $a, d\neq 1$. 

If $d=1$ or $a=1$ the resolution graph is given by
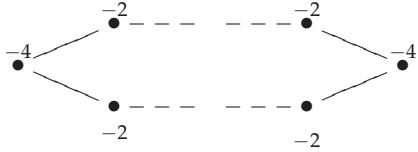
\begin{equation}\label{eqn_resolution_universal2}
\xymatrix@R=5pt@C=24pt@M=3pt@W=6pt@H=0pt{
  &\overtag{\bullet}{-2}{6pt}\dashto[r]&\dashto[r]&\overtag{\bullet}{-2}{6pt}\\
  \overtag{\bullet}{-4}{6pt}\lineto[ur]\lineto[dr]&&&&
  \overtag{\bullet}{-4}{6pt}\lineto[ul]\lineto[dl]\\
  &\undertag{\bullet}{-2}{0pt}\dashto[r]&\dashto[r]&\undertag{\bullet}{-2}{3pt}\\
  \\ \\}\end{equation}
  where the top and bottom strings are of length $2a-3$ or $2d-3$. 

From \cite[Proposition 2.5]{NW}, a cusp singularity with resolution graph $[-b_1, \cdots, -b_k]$ is a complete intersection singularity if and only if 
$$\sum_{i=1}^{k}(b_i-2)\leq 4$$
which is equivalent to the dual cusp has resolution cycle of length $\le 4$. 
It is easy to check that the above resolution graph of the  universal abelian cover cusp $(\widetilde{S},0)$ exactly satisfies this condition.  
The dual graph (of the dual cusp) of (\ref{eqn_resolution_universal}) and (\ref{eqn_resolution_universal2}) is given by
$[-2a, -2d, -2a, -2d]$ which has length $4$. 
\end{example}

\begin{example}\label{example_quotient_cusp2}
Let us look at the $\zz_2$-quotient cusp singularity $(S,0)$ in Example \ref{example_quotient_cusp} again.  
The universal abelian cover cusp $(\widetilde{S},0)$ has resolution cycle given by (\ref{eqn_resolution_universal}) and (\ref{eqn_resolution_universal2}), and it is a 
complete intersection cusp.  From \cite{GHK}, this cusp singularity  $(\widetilde{S},0)$ is smoothable if and only if the resolution cycle of its dual cusp  is the anticanonical divisor of a smooth 
rational surface. 

From \cite[(1.1) Theorem]{Looijenga}, for certain $a, d\geq 1$, there is a smooth rational surface $(X,E)$ with the anticanonical divisor $E$ given by $[-2a, -2d, -2a, -2d]$.   Thus from  \cite{GHK}, the cusp singularity  $(\widetilde{S},0)$ is smoothable, 
which induces the $\qq$-Gorenstein deformation of  $(S,0)$. 
\end{example}

\begin{example}\label{example_quotient_cusp_family}
Recall that in Example \ref{example_quotient_cusp}, the quotient-cusp singularity 
$(S,0)$ has resolution cycle (\ref{eqn_cusp_graph2}), which associates with a matrix 
$$B=\mat{cc}
a&b\\
c&d
\rix.
$$
The quotient $(\widetilde{S}/G, 0)$ is isomorphic to  $(S,0)$.   The next Proposition \ref{prop:equivariant-example-4} shows that the $\lci$ singularity  $(\widetilde{S}, 0)$ admits a one-parameter  smoothing and a two parameter smoothing. 
\end{example}

\begin{prop}
\label{prop:equivariant-example-4}
Put
\[
 F_1=x^2+y^2-u^\alpha v^\beta,\qquad
 F_2=u^2+v^2-x^\gamma y^\delta,
\]
with the notation and parity assumptions of Example~\ref{example_quotient_cusp_family}, and let
$G=\langle A_1,A_2,A_3,A_4\rangle$.  The ideal
$(F_1,F_2)$ is $G$-stable.  More precisely, if
$ad-bc=1$ and $\zeta$ is primitive of order $4b$, then the
characters of the two equations are
\[
\begin{array}{c|cccc}
 &A_1&A_2&A_3&A_4\\ \hline
\chi_1&\zeta^{2a}&\zeta^{2a}&\zeta^2&\zeta^2\\
\chi_2&\zeta^2&\zeta^2&\zeta^{2d}&\zeta^{2d}.
\end{array}
\]
Let
\[
 G_{\mathrm{fix}}:=\ker(\chi_1)\cap\ker(\chi_2).
\]
For a sufficiently general pair $(\lambda_1,\lambda_2)$, the family
\[
 \mathcal S_{\lambda}\;:=\;
 \left\{
 F_1=\lambda_1 t,\quad F_2=\lambda_2 t
 \right\}\subset \mathbb A^4\times\mathbb A^1_t
\]
is a flat one-parameter smoothing of $(\widetilde S,0)$ and is
$G_{\mathrm{fix}}$-equivariant with trivial action on $t$.  Thus
$[\mathcal S_{\lambda}/G_{\mathrm{fix}}]\to\mathbb A^1_t$ is an
$\lci$ covering Deligne-Mumford stack family.

The full group $G$ has the canonical two-parameter equivariant
family
\[
 \mathcal S_2=\{F_1=t_1,\ F_2=t_2\}
 \longrightarrow T_2:=\mathbb A^2_{t_1,t_2},
\]
where
\[
 g(t_1,t_2)=\bigl(\chi_1(g)t_1,\chi_2(g)t_2\bigr).
\]
Consequently $[\mathcal S_2/G]\to[T_2/G]$ is the full-group
equivariant smoothing.  A full-group one-parameter restriction with
base character $\psi$ exists whenever there are positive integers
$m_1,m_2$ with $\chi_i=\psi^{m_i}$; it is given by
$(t_1,t_2)=(t^{m_1},t^{m_2})$ and $g(t)=\psi(g)t$.
Without such a character relation, the two-parameter family (or the
fixed-base subgroup $G_{\mathrm{fix}}$) is the correct statement.
\end{prop}
\begin{proof}
For $A_1$ and $A_2$, the first equation has character
$\zeta^{2a}$ and the second has character $\zeta^2$.  For $A_3$
and $A_4$, the characters are $\zeta^2$ and $\zeta^{2d}$,
respectively.  For example,
\[
 (-1)^\gamma\zeta^{a(\gamma+\delta)}
 =(-1)^\gamma\zeta^{2ad}
 =\zeta^2(-1)^{\gamma+c}=\zeta^2,
\]
because $ad-bc=1$, $\zeta^{2b}=-1$, and $\gamma\equiv c\pmod2$.
The other identities are identical, using
$\alpha+\beta=2a$, $\gamma+\delta=2d$, and the stated parity
conditions.  This proves the character table and the stability of
the complete-intersection ideal.

The equations $F_i-\lambda_i t$ form a relative regular sequence
on the smoothing component.  After choosing $(\lambda_1,\lambda_2)$
in the nonempty smoothing open, generic smoothness gives a smooth
generic fibre; flatness follows from the Cohen--Macaulay/regular-base
criterion.  The subgroup $G_{\mathrm{fix}}$ fixes both equations and
the parameter, while the full group acts on the two parameters by the
displayed characters.  Quotienting the resulting lci family as a Deligne-Mumford
stack proves the assertions.
\end{proof}

\subsection{Discriminant  cover of s.l.c.  surface germs}\label{subsubsec_discriminant}

Now we assume that the s.l.c. germ $(S,0)$ is a Gorenstein simple elliptic singularity,  a cusp singularity or a degenerate cusp singularity.  Note that simple elliptic singularities and cusps are  normal surface singularities. 

\subsubsection{Cusp singularities}\label{subsubsec_cusp_singularity}

Let us first  fix to the cusp singularity case.  In this case the index one cover is just $(Z,0)=(S,0)$, and we have the good resolution 
$\sigma: X\to S$, where  $\sigma^{-1}(0)=A$ is a cycle of rational curves.
The link  $\Sigma$ is not a rational homology sphere.  The link is  a $T^2$-bundle over the circle $S^1$ and $H_1(\Sigma,\zz)=\zz\oplus G$. 
Suppose that the type of the cusp singularity is given by  $[-e_1, \cdots, -e_k]$ determined by the resolution graph of the cusp, where $e_i$ are positive integers  and $-e_i$ are the self-intersection numbers of the 
the component curves in the exceptional divisor of the minimal resolution of $(S,0)$.   Then the monodromy of the link is given by the matrix 
$$
A=
\mat{cc}
0&-1\\
1&e_k
\rix\cdots \mat{cc}
0&-1\\
1&e_1
\rix
=
\mat{cc}
a&b\\
c&d
\rix,
$$
such that $\pi_1(\Sigma)=\zz^2\rtimes_{A}\zz$. 
It is known from \cite[Proposition 2.5]{NW} that the cusp $[-e_1, \cdots, -e_k]$ is an lci cusp  if and only if
$\sum_{i=1}^k(e_i-2)\le 4$, which is equivalent to the fact that  the dual cusp has
resolution cycle of length at most $4$.

As in \cite[\S 4]{NW},  there is no natural epimorphism $\pi_1(\Sigma)\to G$, hence no natural Galois cover with transformation group $G$. 
But  different epimorphisms of $H_1(\Sigma,\zz)=\zz\oplus G\to G$ are related by automorphisms 
of $\pi_1(\Sigma)$, and hence by automorphisms of $(S,0)$.  Therefore,  there is a natural cover up to automorphisms, called the discriminant cover.  Also for any subgroup $K\subset G$
we still have the cover for $K$ and the dual cover for $K$, with transformation groups $G/K$ and $G/K^{\perp}$ respectively. 
From the proof in  \cite[\S 4]{NW},  take $K=\{1\}$ and let $(\widetilde{S},0)\to (S,0)$ be the discriminant cover of $(S,0)$, which is also the dual cusp of 
$(S,0)$. 
   
In \cite[Proposition 4.1 (2)]{NW}, Neumann and Wahl  constructed a finite  cover 
$(\widetilde{S}, 0)$ of $S$ with transformation group $G^\prime$ so that $(\widetilde{S}, 0)$ is  a hypersurface cusp, which is  l.c.i.  
Let $H$ be the subspace of $\zz^2$ generated by $\mat{c}a\\ c\rix$ and $\mat{c}1\\ 0\rix$ \footnote{Note that in the proof of \cite[Proposition 4.1 (2)]{NW}, there is a typo for the second vector.}.  We can assume $a\neq 0$, otherwise we just take $H=\zz^2$.     Then the  matrix $A$ takes the subspace $H$ to itself
by the matrix  $\mat{cc}0&-1\\ 1& t\rix$ where $t=\tr(A)=a+d$.
The finite transformation group $G^\prime$ is  given as follows:   first we take  the quotient finite group $N(H\rtimes \zz)/H\rtimes \zz$, where $N(H\rtimes \zz)$ is the normalizer.  Then the subgroup 
$H\rtimes \zz\subset \pi_1(\Sigma)$ determines a cover  of $S$. 
This cover determined by $H\rtimes \zz$ is either the cusp with resolution graph consisting of a cycle with one vertex weighted $-t$ or the dual cusp of this, according as the above basis is oriented correctly or not, i.e., whether $a< 0$ or  $a > 0$.
By  taking the discriminant cover if necessary we get the cover $(\widetilde{S}, 0)$ of $S$  with transformation group 
$G^\prime$. 
The key issue is that $(\widetilde{S}, 0)$  a complete intersection cusp.  In the case that $H\subset \pi_1(\Sigma)$ is normal, then $(\widetilde{S}, 0)$  is a Galois cover whose quotient is the cusp $(S,0)$. A non-Galois cover of cusp singularities exists, see \cite{Anderson}.

Thus, we obtain
\begin{lem}\label{lem_cusp_cover}
Let $(S,0)$ be a cusp singularity, then there exists a finite discriminant cover $(\widetilde{S}, 0)$ with transformation group $G^\prime$ and the cusp  $(\widetilde{S}, 0)$  is a complete intersection cusp.  
A deformation of the Deligne-Mumford stack $[\widetilde{S}/G^\prime]$; i.e., a $G^\prime$-equivariant deformation of $\widetilde{S}$,  induces a Gorenstein deformation of the cusp $(S,0)$ when the cover is a Galois cover. 
\end{lem}

We say that a singularity germ $(S,0)$ admits an $\lci$ lifting if there is an $\lci$ cover $(\widetilde{S}, 0)\to (S,0)$ with transformation group $G^\prime$ such that $(\widetilde{S}, 0)$ is an $\lci$ singularity. 
We say that a smoothing  $(\sS,0)\to \Delta$ of the singularity $(S,0)$ admits an $\lci$ smoothing lifting if there is a smoothing 
$\tilde{f}: (\widetilde{\sS}, 0)\to \Delta$ which induces the smoothing $(\sS,0)\to \Delta$ and the fibers of $\tilde{f}$ have only lci singularities. 
From the discriminant cover of the cusp $(S,0)$, 
a $G^\prime$-equivariant smoothings is an  {\em $\lci$ smoothing lifting} of $(S,0)$.

The smoothing of cusp singularities has a long history, see \cite{Looijenga}, \cite{GHK}, \cite{Engle}.
 In \cite[Theorem 1.3]{Jiang_cusp}, we generalize the Looijenga conjecture to the equivariant setting and prove that for any cusp singularity $(S,0)$ admitting a one-parameter smoothing, there exists an lci smoothing lifting of the singularity. 
\begin{thm}\label{thm_smoothing_lci_cusp_paper}(\cite[Theorem 1.3]{Jiang_cusp})
  Let $(S, 0)$ be a cusp singularity.   Suppose that $(S,0)$ admits a smoothing $f: (\sS,0)\to \Delta$.  Then there exists a smoothing $\tilde{f}: (\tilde{\sS},0)\to \Delta$ of an lci cusp together endowed with  a finite group $G$ action such that the quotient induces the smoothing $f: (\sS,0)\to \Delta$.
\end{thm}

\begin{rmk}\label{rmk_cusp_smoothing_analytic}
    The result in Theorem \ref{thm_smoothing_lci_cusp_paper} only works analytically, and the author doesn't know how to give an algebraic proof of the result. 
\end{rmk}

\subsubsection{Simple elliptic singularities}\label{subsubsec_simple_elliptic}

Let  $(S,0)$ be a  simple elliptic singularity. Let 
$\sigma: X\to S$ be the minimal resolution such that 
$A=\sigma^{-1}(0)$ is the exceptional elliptic curve.  Let $d:=-A\cdot A$ be the degree of $(S,0)$.    The local embedding dimension of the singularity is given by 
$\max(3, d)$.   It is known from \cite{Laufer}, that the simple elliptic singularity $(S,0)$ is an $\lci$ singularity if the negative self-intersection $d\le 4$.  If $d\ge 5$, then $(S,0)$ is never $\lci$. 
 From \cite{Pinkham2}, \cite{Kollar-Shepherd-Barron},  it admits a smoothing if and only if $1\leq d\leq 9$.   

 We list the result in \cite[Theorem 1.3]{Jiang_2023} here. 
\begin{thm}\label{thm_elliptic_singularity_lci_lifting}(\cite[Theorem 1.3]{Jiang_2023})
Let $(S,0)$ be a simple elliptic surface singularity,  and $(X, A)$ its minimal resolution. Then $(S,0)$ admits an $\lci$ smoothing lifting by a simple elliptic singularity $(\widetilde{S},0)$ of degree $\leq 4$ only when 
$d\neq 5, 6, 7$ and $1\le d\le 9$.
\end{thm}

\begin{rmk}\label{rmk_simple_elliptic_lci}
    The link $\Sigma$ of a simple elliptic singularity $(S,0)$ of degree $d$ is a $S^1$-bundle over the torus $T^2$ and $H_1(\Sigma,\zz)=\zz^2\oplus \zz_d$. The  fundamental group of the link is $\pi_1(\Sigma)=\zz^2\ltimes\zz$.  Although Theorem \ref{thm_elliptic_singularity_lci_lifting} implies that when $d=5, 6, 7$, there is no lci smoothing lifting, but there is an lci cover for these singularities by a degree one simple elliptic singularity determined by surjective morphism $H_1(\Sigma,\zz)=\zz^2\oplus \zz_d\to \zz_d$.  The transformation group $G=\zz_d$.
\end{rmk}

\subsubsection{Degenerate cusp singularities}

For a degenerate cusp singularity $(S,0)$ whose embedding dimension $\ge 5$, we prove when does it admit an lci cover.  We give an abstract proof using local ring of the singualrity. 

\begin{thm}
\label{thm_minimally_elliptic_universal}
Let $R$ be the ring of a Gorenstein complex surface germ. Assume
that a finite group $G$ acts on a surface germ with ring $B$,
and that a specified finite quotient presentation is given:
\begin{equation}\label{eq:actual-quotient}
                         R\simeq B^G.
\end{equation}
For $A\in\mathrm{Art}_k$, the category of local Artin algebras, let $B_A$ be a $G$-equivariant flat deformation
of $B$, with $G$ acting trivially on $A$. Then
\begin{equation}\label{eq:quotient-deformation}
                         R_A=B_A^G
\end{equation}
is a flat, relatively Gorenstein deformation of $R$. Formation of
$R_A$ commutes with changes of Artinian base. Consequently there
is a natural transformation
\[
\Def^G(B)\longrightarrow\Def^{\Gor}(R).
\]
If $B$ is $\lci$, then $B_A/A$ and the quotient stack
$[\spec(B_A)/G]/\spec(A)$ are relatively $\lci$, and the stack's
coarse family is $\spec(R_A)$.

For a flat finite-presentation family over a locally Noetherian
complex base, the corresponding Gorenstein conclusion holds on
a neighbourhood of the specified Gorenstein special fibre in the
quotient total space. The analogous neighbourhood assertion holds
for the relative $\lci$ property of the covering family.
\end{thm}
\begin{proof}
    Because $|G|$ is invertible in $A$, the Reynolds operator
\[
 e_G:B_A\longrightarrow B_A,\qquad
 b\longmapsto \frac{1}{|G|}\sum_{g\in G}g(b)
\]
is an $A$-linear idempotent with image $B_A^G$. Thus $R_A$ is an
$A$-module direct summand of the flat $A$-module $B_A$, and is
flat. For every homomorphism $A\to A'$ the same idempotent acts
on $B_A\otimes_A A'$; its image gives
\begin{equation}\label{eq:invariant-basechange}
 (B_A^G)\otimes_A A'
       \simeq (B_A\otimes_A A')^G.
\end{equation}
For completed rings the corresponding completed tensor products
give the same assertion. In particular
\[
R_A\otimes_A \mathbf{k}\simeq B^G\simeq R.
\]
Finite-group invariant rings in these categories are Noetherian,
and the quotient map is finite. In the algebraic category this is
the usual finite-group invariant-ring theorem; passing to the
local completion gives its complete local version. Hence the
construction is a deformation of the specified germ, not merely
a flat module.

A flat morphism with locally Noetherian fibres is Gorenstein if
and only if its fibres are Gorenstein; see
\cite[Tag 0C05]{Stack_Project}. The spectrum of $A$ has just one
geometric fibre, and the fibre of $\spec(R_A)\to A$ is the
Gorenstein germ $\spec(R)$. It follows that $R_A/A$ is relative
Gorenstein. Equivalently, its relative dualizing complex is locally
an invertible module in degree $-2$. Thus its relative dualizing
sheaf is invertible. In the usual finite-presentation setting its
formation is compatible with base change.

For an actual family over a larger base, flatness and
\eqref{eq:invariant-basechange} still hold. The relatively Gorenstein
locus is open and contains the special fibre under consideration;
see \cite[Section 48.25]{Stack_Project}. Restrict to that neighbourhood.
For a proper family one can additionally shrink the base to
exclude the image of the non-Gorenstein locus. For a germ the
neighbourhood assertion is the appropriate one.

If $B$ is lci, the flat deformation $B_A/A$ has lci geometric
fibres. The flat fibre criterion for lci morphisms, equivalently
the syntomic criterion, gives the relative lci assertion
\cite[Tags 01UB and 069K]{Stack_Project}. This can also be seen by
lifting a regular-sequence presentation of $B$ to $A$; flatness
ensures that it remains a relative complete intersection.
Finally $\spec(B_A)\to[\spec(B_A)/G]$ is a finite \'etale atlas,
so the relative lci property holds for the stack. Equation
\eqref{eq:invariant-basechange} identifies its coarse special
fiber. This proves Theorem~\ref{thm_minimally_elliptic_universal}.
\end{proof}

\begin{rmk}\label{rmk_length_5_non-lci}
    For $n\geq5$, put
\[
R_n=\mathbf{k}[[x_1,\ldots,x_n]]/
(x_i x_j:\ i,j\text{ distinct and nonadjacent in the }n\text{-cycle}).
\]
Then $S_n=\spec(R_n)$ is a Gorenstein degenerate cusp with smooth
normalization and embedding dimension $n$, but it is not $\lci$.
There is no finite surjective map of surface germs $Y\to S_n$
with $\lci$ source that is \'etale away from the vertex. In particular,
there is no $\lci$ finite-group quotient chart whose quotient stack
agrees with $S_n$ off the vertex.
\end{rmk}

Let $R$ be a reduced Gorenstein surface ring with finite normalization $N$ as in the degenerate cusp case.  Let $\mathfrak{c}$ be its conductor, and set
$D=R/\mathfrak{c}$, and $C=N/\mathfrak{c}$. The conductor square is the fiber-product 
$$
R=N\times_{C}D.
$$

\begin{prop}\label{prop:gluing}
Suppose the following data have been provided.
\begin{enumerate}
\item Finite algebras $\widetilde N$, $\widetilde C$, and
      $\widetilde D$ over $N$, $C$, and $D$, respectively, with
      compatible actions of one finite group $G$.
\item Compatible $G$-equivariant maps
      $\alpha:\widetilde N\twoheadrightarrow\widetilde C$ and
      $\beta:\widetilde D\to\widetilde C$ whose invariant diagram
      is identified with $N\to C\leftarrow D$.
\item A verification that the fibre-product ring
      \[
                     B=\widetilde N\times_{\widetilde C}
                                      \widetilde D
      \]
      is a surface ring and is lci at the points under study.
      For example, one can exhibit a regular-sequence presentation
      $B\simeq\mathbf{k}[[u_1,\ldots,u_e]]/(f_1,\ldots,f_{e-2})$
      at each such point.
\end{enumerate}
Then $B$ is finite over $R$, $B^G\simeq R$, and
$[\spec(B/G)]$ is an lci Deligne--Mumford stack with coarse space
$\spec(R)$. Every $G$-equivariant flat deformation of $B$ induces
a Gorenstein deformation of $R$ as in
Theorem~\ref{thm_minimally_elliptic_universal}.

If this stack is required to agree with $\spec(R)$ over an open
set $U$, one must additionally check that $G$ acts freely over
$U$. In particular this extra requirement is needed for
$U=\spec(R)\setminus\{0\}$ in a punctual covering construction.
\end{prop}
\begin{proof}
The compatible maps put $R$ in $B$. The latter is an $R$-submodule
of $\widetilde N\oplus\widetilde D$, which is finite over the
Noetherian ring $R$, so $B$ is finite over $R$. Taking invariant
tuples in the fibre product gives
\[
 B^G=\widetilde N^G\times_{\widetilde C^G}\widetilde D^G
          \simeq N\times_C D=R.
\]
This equality uses the specified equivariant diagram, not a
presumed action after gluing. The lci hypothesis and the finite
etale atlas give the stack assertion. A free finite-group action
over $U$ gives a torsor there, so its quotient stack is $U$.
The deformation assertion is Theorem~\ref{thm_minimally_elliptic_universal}.
\end{proof}

\begin{rmk}
    The lci verification in item (3) is essential: it cannot be
deleted or replaced by lci of $\widetilde N$. For $R_n$ with the
identity normalization covers the other data exist but item (3)
fails. Thus this proposition is a sufficient construction
criterion, not a proof of existence for every degenerate cusp.

There is a useful family version with equally explicit
hypotheses. Suppose the three algebras and both maps lift
equivariantly over an Artinian base $A$, all three lifted
algebras are $A$-flat, and $\alpha_A$ is surjective. Then
\[
 0\longrightarrow B_A\longrightarrow
 \widetilde N_A\oplus\widetilde D_A
 \xrightarrow{\alpha_A-\beta_A}\widetilde C_A
 \longrightarrow0
\]
shows that $B_A$ is $A$-flat and that its formation commutes with
base change. If its special fibre is the checked lci ring $B$,
it is a relative lci deformation. Its invariant ring gives the
quotient deformation. To obtain a \emph{prescribed} family
$R_A$, its invariant conductor diagram must be identified with
that of $R_A$. A smoothing need not admit simultaneous flat
normalization and conductor data of this kind; existence of
these lifted data is not part of the proposition.
\end{rmk}

\begin{example}\label{Example_degenerate_cusp}
    Here is a hypersurface  example of $\lci$ degenerate cusp singularity $(\widetilde S,0)$ whose $\zz_2$ quotient is a non $\lci$ degenerate cusp $(S,0)$.
    
Let $R=\mathbf{k}[x,y,z]/(xyz)$ and let $G\cong\mathbb Z/2$ act by
\[
 \sigma(x,y,z)=(-x,-y,-z).
\]
The source $(\widetilde S,0)=(\spec(R),0)$ is the degenerate cusp $T^3_{\infty,\infty}$ in \cite[Lemma 2.11 (5)]{Tziolas}. In particular it is a hypersurface, hence $\lci$.

The invariant ring $B=R^G$ is the second Veronese subring of $R$. It is generated by
\[
 a=x^2,\quad b=y^2,\quad c=z^2,\quad
 d=xy,\quad e=yz,\quad f=zx.
\]
There is an explicit presentation
\begin{align*}
 B&\cong\mathbf{k}[a,b,c,d,e,f]/I,\\
 I&=(d^2-ab,\ e^2-bc,\ f^2-ca,\ ae,\ bf,\ cd,\ de,\ ef,\ fd).
\end{align*}
To check completeness of the relations, reduce powers of $d,e,f$ using the first three equations. Products involving two distinct members of $d,e,f$ vanish, as do $ae,bf,cd$. Also $abc=cd^2=0$. The remaining monomials correspond precisely to the even-degree monomials in $x,y,z$ supported on at most two variables, with distinct images after duplicate pure powers are identified. Thus no additional relations are needed.

There are no linear relations, so the embedding dimension is $6$. The nine quadrics above are linearly independent and form a minimal generating set of $I$ at the origin. Since $\dim B=2$,
\[
 \mathrm{ht}(I)=6-2=4<9.
\]
Consequently the quotient $S=\widetilde S/G$ is \textbf{not lci}. 

To see that the quotient $(S,0)$ is a degenerate cusp, let
\[
 C=\{xyz=0\}\subset\pp^2
\]
be the triangle of three projective lines. The blowup of the vertex of its affine cone $\widetilde S$ is
\[
 \widetilde X=\Tot(\sO_C(-1))\longrightarrow \widetilde S.
\]
The lifted involution acts trivially on $C$ and by multiplication by $-1$ on each fiber. Taking the quotient squares the fiber coordinate, hence
\[
 X=\widetilde X/G
 \cong\Tot(\sO_C(-2))\longrightarrow S.
\]
This morphism is proper and an isomorphism away from the vertex. Near a node of $C$, its source has local equation $uv=0$ with a free fiber coordinate. Thus the source is semismooth. Its double curves are the fibers over the three nodes, and none is contracted to a point.

The exceptional divisor $E$ is the zero section, isomorphic to $C$. In the normalization, the three exceptional components are zero sections of $\sO_{\pp^1}(-2)$; each has self-intersection $-2$. There are no exceptional $(-1)$-curves to contract, so this is the minimal semiresolution. The quotient is nonnormal, with three irreducible components and double curves. Therefore it satisfies the definition of a degenerate cusp.

Finally, $E$ is Cartier and $\sO_E(E)=\sO_C(-2)$, so
\[
 E^2=-6.
\]
Tziolas's formula $\mathrm{embdim}_0(S)=\max\{3,-E^2\}$ \cite[Lemma 2.11(2)]{Tziolas} recovers the directly computed embedding dimension six.
\end{example}

\begin{example}\label{example_embedd_4_degenerate_cusp}
  We also have an embedding-dimension-four example. 
Take the complete intersection
\[
 (\widetilde S,0)=(\spec(R_4),0),\qquad R_4=\mathbf{k}[x,y,z,w]/(xy,zw),
\]
and let $G\cong\mathbb Z/2$ act by
\[
 \sigma(x,y,z,w)=(-x,-y,-z,-w).
\]
This is $T^4_{\infty,\infty,\infty}$ in \cite[Lemma 2.11(6)]{Tziolas}. The invariant ring $B_4=R_4^G$ is generated minimally by the eight quadrics
\[
 x^2,\ y^2,\ z^2,\ w^2,\ xz,\ xw,\ yz,\ yw.
\]
Thus the quotient $S$ has embedding dimension eight.

There is a direct equation-count proof that $(S,0)$ is not lci. Give these eight generators degree one. The Hilbert series of the original graded ring is
\[
 H_{R_4}(t)=\frac{(1-t^2)^2}{(1-t)^4}
 =\frac{(1+t)^2}{(1-t)^2},
 \qquad \dim_\mathbf{k}(R_4)_m=4m\quad(m\geq1).
\]
Hence $\dim(B_4)_1=8$ and $\dim(B_4)_2=16$. The kernel of the minimal polynomial presentation has
\[
 \binom{8+1}{2}-16=20
\]
linearly independent quadratic relations. All are minimal, since there are no linear relations. Its codimension is only $8-2=6$, proving that $(S,0)$ is not $\lci$.

The quotient is a Gorenstein degenerate cusp. This follows either from the same geometric argument using the cycle of four lines $\{xy=zw=0\}\subset\pp^3$, or directly from \cite[Corollary 2.11, case $(U_{2c})$]{Liu}. Liu considers the actions
\[
 (x,y,z,w)\longmapsto
 (\xi x,\xi^{-1}y,\xi^qz,\xi^{-q}w),\qquad (q,n)=1,
\]
where $\xi$ is a primitive $n$th root of unity, and identifies the quotient as a Gorenstein degenerate cusp. Our action is the case $n=2$, $q=1$.
\end{example}

\subsubsection{Examples of smoothing}

\begin{example}\label{example_cusp_sing}
In Example \ref{example_quotient_cusp}, there is a universal abelian cover of the  quotient-cusp which factors through the cusp in the quotient.    
We have examples of cusps which do not admit abelian covers by complete intersection cusps. 

Let $(S,0)$ be a cusp singularity whose resolution graph is given  by (\ref{eqn_cusp_graph}) in Example \ref{example_quotient_cusp}. 
Let 
$$k=4, e_1=6, e_2=3, e_3=3, e_4=2.$$
Then the resolution cycle of this specific cusp is $[-10, -3, -3, -3, -3, -2]$.   From \cite[Lemma 2.4]{NW},  the dual cusp has resolution cycle 
$$[-4, -2, -2, -2, -2, -2,-2,-2].$$
The dual cusp of the cusp corresponding to $[-4, -2, -2, -2, -2, -2,-2,-2]$ has resolution cycle $[-2, -10]$, which is a complete  intersection cusp.  Thus the cusp $(S,0)$ corresponding to the resolution cycle 
$[-10, -3, -3, -3, -3, -2]$ maybe be covered by a complete intersection cusp. 
\end{example}

\begin{example}\label{example_Pinkham}
Here is  an example of hypersurface cusp singularities with a finite abelian group action in \cite[Corollary]{Pinkham}.  Let $(\widetilde{S}, x)$ be a hypersurface cusp given by:
$$\{x^p+y^q+z^r+xyz=0\}, \quad   \frac{1}{p}+\frac{1}{q}+\frac{1}{r}<1.$$
Here $p, q, r$ are positive integers. 
The resolution cycles of such a cusp is given in \cite[Lemma 2.5]{Nakamura}.  The dual cusp  of  this cusp has resolution cycle 
$$(-(p-1), -(q-1), -(r-1)).$$

Let $\Sigma$ be the link of $(\widetilde{S}, x)$.   The torsion subgroup $G=H_1(\Sigma,\zz)_{\tor}$ is isomorphic to the group
$$\{\lambda, \mu, \nu | \lambda^p=\mu^q=\nu^r=\lambda \mu\nu\}.$$
The group $G$ acts on the hypersurface cusp singularity by
$$x\mapsto \lambda x; \quad y\mapsto \mu y; \quad z\mapsto \nu z.$$
The quotient $(\widetilde{S}, x)/G$ is the cusp $(S, x)$ whose resolution cycle is $(-(p-1), -(q-1), -(r-1))$. 
Note that if $(p-1)-2+(q-1)-2+(r-1)-2>4$, then the dual cusp $(S, x)$ is not a complete intersection cusp. 

The hypersurface cusp $(\widetilde{S}, x)$ admits a $G$-equivariant smoothing, see Proposition \ref{prop:equivariant-example-6}. 
\end{example}

\begin{prop}
\label{prop:equivariant-example-6}
Let
\[
 F=x^p+y^q+z^r+xyz,
 \qquad \frac1p+\frac1q+\frac1r<1,
\]
and let
\[
 G=\{(\lambda,\mu,\nu):
 \lambda^p=\mu^q=\nu^r=\lambda\mu\nu\}
\]
act diagonally on $(x,y,z)$.  Define the equation character
\[
 \chi(\lambda,\mu,\nu):=\lambda^p
   =\mu^q=\nu^r=\lambda\mu\nu.
\]
Then $gF=\chi(g)F$.  The fixed-base family
\[
 \mathcal Y_0=\{F=t\}\subset\mathbb A^3\times\mathbb A^1_t
 \longrightarrow\mathbb A^1_t
\]
is $G_0$-equivariant, where $G_0:=\ker\chi$ acts trivially on $t$.
Its generic fibre is smooth after deleting the finite set of critical
values of $F$, so $[\mathcal Y_0/G_0]\to\mathbb A^1_t$ is an lci
covering Deligne-Mumford stack smoothing of $(\widetilde S,0)/G_0$.

For the original full quotient by $G$, let $G$ act on the parameter
by $g(t)=\chi(g)t$.  Then
\[
 [\{F=t\}/G]=[\mathcal Y/G]\longrightarrow[\mathbb A^1_t/G]
\]
is a full-group equivariant smoothing of $[\widetilde S/G]$.  If
$m=|\chi(G)|$, the invariant coordinate $s=t^m$ gives a flat coarse
base change over $\mathbb A^1_s$; the stack quotient is the natural
lci object to retain for the perfect obstruction theory.
\end{prop}
\begin{proof}
Every monomial in $F$ has character $\chi$, so the equation is
semi-invariant.  To check smoothness, a critical point with
$xyz\ne0$ satisfies
\[
 p x^p=q y^q=r z^r=-xyz,
\]
and hence has critical value
\[
 xyz\left(1-\frac1p-\frac1q-\frac1r\right)\ne0.
\]
If one coordinate vanishes, the critical equations force all three
coordinates to vanish.  Thus $0$ is an isolated critical value and a
general nonzero fibre is smooth.  The formal family is flat because
$\mathbb C[[x,y,z]]$ is torsion-free over
$\mathbb C[[t]]$ via $t\mapsto F$, hence flat over the DVR.
The kernel $G_0$ preserves $F=t$ with fixed base.  The full group
preserves the equation after the displayed action on $t$; taking the
quotient stack gives the claimed family.  Since invariants are a
Reynolds direct summand in characteristic zero, passage to the
invariant coordinate $s=t^m$ preserves flatness of the coarse family.
\end{proof}

\subsection{More on equivariant smoothing of simple elliptic and cusp singularities}\label{subsubsec_equivariant_smoothing_cusp}

Let $(X, 0)$ be a germ of simple elliptic or cusp singularity as in \S \ref{subsubsec_discriminant}, and $(S,0)=(X,0)/\zz_r$ the quotient singularity germ in \S \ref{subsubsec_universal_discriminant}. 
Note that $r=2,3,4,6$ in the simple elliptic singularity case and $r=2$ in the cusp singularity case. 

Let $\Sigma_X$ and $\Sigma_S$ be the links of the singularity germs.  Then $\Sigma_X\to \Sigma_S$ is an unramified $r$-th fold cover. 
Since the link $\Sigma_S$ of $(S,0)$ is a rational homology sphere,  from   \S \ref{subsubsec_universal_discriminant}, let $\pi: (\widetilde{S}, 0)\to (S,0)$ be the universal abelian cover with transformation finite abelian group
$G=H_1(\Sigma_S)$.  Suppose that there is a subgroup $K\subset G$ such that we have an exact sequence
$$0\to K\rightarrow H_1(\Sigma_S)\rightarrow \zz_r\to 0,$$
then it determines a $r$-fold cover of germs $(S^\prime, 0)\to (S,0)$ such that the map $\Sigma_{S^\prime}\to \Sigma_S$ is an unramified $r$-cover of the links.  So this implies that 
$(S^\prime, 0)\cong (X,0)$ and $\Sigma_{S^\prime}\cong \Sigma_X$.
The following diagram of links 
\[
\xymatrix{
\Sigma_{\widetilde{S}}\ar[d]\ar[r]& \Sigma_{X}\ar[dl]\\
\Sigma_{S}&
}
\]
implies the  commutative diagram
\[
\xymatrix{
&0\ar[d]&0\ar[d]&0\ar[d]&\\
0\ar[r]&K^\prime\ar[r]\ar[d]_{\id}& K^{\prime\prime}\ar[r]\ar[d]& K\ar[r]\ar[d]& 0\\
0\ar[r]&K^\prime\ar[r]\ar[d]& \pi_1(\Sigma_S)\ar[r]\ar[d]& H_1(\Sigma_S)\ar[r]\ar[d]& 0\\
&0\ar[r]& \zz_r\ar[r]^{\id}\ar[d]& \zz_r\ar[r]\ar[d]& 0.\\
&&0& 0&
}
\]
The cover $\Sigma_{\widetilde{S}}\to \Sigma_{X}$ has transformation group $K$.  Thus, this induces a finite abelian cover 
$$\pi: (\widetilde{S}, 0)\to (X,0)$$
with transformation group $K$. 

Comparing with Theorem \ref{thm_minimally_elliptic_universal2}, we have 
\begin{thm}\label{thm_equivariant_smoothing_e_cusp}
If $(X,0)$ is a simple elliptic singularity germ, or a cusp singularity  germ such that there exists a quotient $((X, 0)/\zz_r,0)=(S,0)$ above, then 
the $K$-equivariant  deformations of $(\widetilde{S}, 0)$ induce $\zz_r$-equivariant  deformations of  $(X,0)$, which induce $\qq$-Gorenstein deformations of 
$(S,0)$.
\end{thm}
\begin{proof}
We only need to check that in the simple elliptic singularity and cusp singularity cases, the cyclic group $\zz_r$ for $r=2,3,4,6$
can be taken as a quotient of $H_1(\Sigma_S)$.  This is from the direct calculations for the group $H_1(\Sigma_S)$
for the simple elliptic singularities and cusps.   The group  $H_1(\Sigma_S)$ can be calculated using the resolution graphs in \cite{Kollar-Shepherd-Barron}, \cite[Theorem 9.6, (3), (4)]{Kawamata_Ann}. 
The cyclic group $\zz_r$ is a summand of $H_1(\Sigma_S)$ in the simple elliptic singularity case.  From the calculation of $H_1(\Sigma_S)$ in \cite[\S 5]{NW} and Example \ref{example_quotient_cusp} in the quotient-cusp case,  the group
$\zz_2$ can definitely be taken as a quotient of $H_1(\Sigma_S)$. 
\end{proof}

\begin{rmk}
Theorem \ref{thm_equivariant_smoothing_e_cusp} is different from Theorem \ref{thm_minimally_elliptic_universal}, since Gorenstein deformations of simple elliptic singularities and cusp singularities are different from their
$\zz_r$-equivariant deformations. 
\end{rmk}

\begin{rmk}
As we talked about the cusp singularities in \S \ref{subsubsec_discriminant},  not every cusp admits a $\zz_2$-quotient. 
Thus, not every cusp has a finite abelian cover by a complete intersection cusp.  From \cite[Proof of Proposition 4.1]{NW}, a necessary condition that a cusp singularity $(X,0)$
has no finite abelian cover by a complete intersection is that the cusp $(X,0)$ and its dual cusp are both not complete intersections.   For instance,  let $(X,0)$ be a cusp with resolution graph self-intersection sequence 
$[-2,-4,-2,-2,-5]$.   This cycle is self-dual, is not a complete intersection from \cite[Proposition 2.5]{NW}.  Thus, there is no finite abelian cover by a complete intersection for $(X,0)$.   
We have to use Theorem \ref{thm_minimally_elliptic_universal} to get a finite (not abelian) cover which is a complete intersection. 
\end{rmk}

\subsection{The  $\lci$ covering Deligne-Mumford stack over s.l.c. surfaces}\label{subsubsec_lci_slc}

Let $S$ be an s.l.c. surface such that the possible elliptic singularities, cusp and degenerate cusp singularities  in $S$  all have embedding dimension $\ge 5$;  i.e. they are not l.c.i. singularities. 
Then the arguments in Theorem \ref{thm_minimally_elliptic_universal2} and Theorem \ref{thm_minimally_elliptic_universal}  constructed the universal abelian cover or the discriminant cover of the singularity germs so that their covers 
are l.c.i.  The construction only depends on the local analytic structure of the singularity.

Similar to the construction of index one covering Deligne-Mumford stack 
$\pi: \SS\to S$,  there are only finite  singularity germs  $(S,0)$ in $S$,  such that the corresponding simple  elliptic singularities, cusp and degenerate cusp singularities  have embedding dimension $\ge 5$ (i.e., not l.c.i.).  Thus,  for each germ singularity, we perform the 
universal abelian cover or the discriminant cover  construction in \S \ref{subsubsec_universal_discriminant} and \S \ref{subsubsec_discriminant}.  Thus, we have 

\begin{prop}\label{prop_lci_cover_DM}
    For any s.l.c. surface $S$ in the KSBA moduli space $M_N$,  there is a canonical Deligne-Mumford stack 
    $$
\pi^{\lci}: \SS^{\lci}\to  S $$
with the coarse moduli space $S$ such that $\SS^{\lci}$ only has l.c.i. singularities. The map $\pi^{\lci}$ has degree one.  We call  $\SS^{\lci}$ the $\lci$ covering Deligne-Mumford stack of $S$.  
\end{prop}
\begin{proof}
Let $([Z/\mu_N],0)$ be a singularity germ of  the index one covering Deligne-Mumford stack $\SS$.  Then $\SS^{\lci}$ locally has the germ chart 
$[\widetilde{S}/G]$, where $G$ is the transformation group of the $\lci$ cover.  The Deligne-Mumford stack $\SS^{\lci}$ is Gorenstein since $\SS^{\lci}$ only has l.c.i. singularities on each chart.
The lci cover $\widetilde{S}$ can be constructed in Theorem \ref{thm_minimally_elliptic_universal2},  Lemma \ref{lem_cusp_cover}, Theorem \ref{thm_elliptic_singularity_lci_lifting}, and Remark \ref{rmk_simple_elliptic_lci}. 

Thus, we get a commutative diagram
\begin{equation}\label{eqn_diagram_ind_lci}
\xymatrix{
\SS^{\lci}\ar[r]^{\hat{\pi}}\ar[dr]_{\pi^{\lci}}&\SS\ar[d]^{\pi}\\
&S.&
}
\end{equation}
We make a summary here. Let $(S,0)$ be a singularity germ in an s.l.c. surface $S$, then we have that 
\begin{enumerate}
\item  if $(S,0)$  is a simple elliptic singularity, a cusp or a degenerate cusp singularity with embedding dimension $\ge 5$, we have 
$$\SS^{\lci}\cong [(\widetilde{S},0)/G^\prime]\to \SS=(Z,0),$$
where $(Z,0)\to (S,0)$ is the index one cover.  In this case $(Z,0)=(S,0)$ and $(\widetilde{S},0)\to (S,0)$ is the discriminant cover. 
\item  if $(S,0)$  is the $\zz_2, \zz_3, \zz_4, \zz_6$-quotient of a simple elliptic singularity,  the $\zz_2$-quotient of a cusp or a degenerate cusp singularity with embedding dimension $\ge 5$, then we have 
$$\SS^{\lci}\cong [(\widetilde{S},0)/G]\to S=(S,0),$$
where $(\widetilde{S},0)\to (S,0)$ is the universal abelian  cover.   The map factors through the index one cover map $(Z,0)\to (S,0)$.  Therefore we have the morphism
$\SS^{\lci}\cong [(\widetilde{S},0)/G]\to \SS=[(Z,0)/\zz_r]$  of stacks, where $r$ is the local index of the quotient singularity. 
\end{enumerate}
\end{proof}

We have the following result for the smoothing families. 

\begin{thm}\label{thm_lci_cover_link_smoothing}
    Let $\overline{f}: \sS\to \Delta$ be a one-parameter family of s.l.c. surfaces and $f: \SS\to \Delta$ be the associated index one covering Deligne-Mumford stack.  Suppose that the non-lci singularity germs $(\SS,x)$  are not simple elliptic singularities of degree $5$, $6$ or $7$, or cusp and degenerate cusp  singularities that can not be lifted into an equivariant smoothing of an $\lci$ cusp in the algebraic sense, then $f: \SS\to \Delta$ can be lifted to a one-parameter flat family $f^{\lci}: \SS^{\lci}\to \Delta$ of lci covering Deligne-Mumford stacks.  Moreover, $\sS$ is  the coarse moduli space of both $\SS$ and $\SS^{\lci}$.
\end{thm}
\begin{proof}
This is from the smoothing of local lci covers $[\widetilde{S}/G]$ for the singularities in 
Theorem \ref{thm_minimally_elliptic_universal2}, Theorem \ref{thm_elliptic_singularity_lci_lifting},  Theorem \ref{thm_smoothing_lci_cusp_paper}, and 
Theorem \ref{thm_minimally_elliptic_universal}
because a smoothing of the stack $[\widetilde{S}/G]$ induces a smoothing of the non-lci slc singularity.  
\end{proof}

\subsection{One-parameter  family of   $\lci$ covering Deligne-Mumford stacks}\label{subsubsec_lci_covering_DM-flat}

In the former sections we mainly talked about the one-parameter smoothing or deformation of simple elliptic and cusp, degenerate cusp singularities. In this section we prove some properties of one-parameter family of lci covering Deligne-Mumford stacks. 
For an s.l.c. surface germ $(S,0)$, we have the $\lci$ cover $\widetilde{S}\to S$ with transformation group $G$ such that the $\lci$-covering Deligne-Mumford stack 
$\SS^{\lci}$ is given by $[\widetilde{S}/G]$.  

\begin{rmk}
Comparing with Example \ref{example_quotient_cusp} and Example \ref{example_cusp_sing}, it is interesting to study the equivariant smoothing of cusp and quotient-cusp  singularities. 
We hope the equivariant Looijenga's conjecture also holds; see \cite{Engle} and \cite{GHK}.  Note that in \cite[Theorem 1.3]{Jiang_cusp} we prove that if a cusp admits a smoothing, it always admits an lci one-parameter smoothing lifting. 
\end{rmk}

Let $A$ be a one-dimensional  $\mathbf{k}$-algebra, and let $\sS/A$ be a one-parameter family of s.l.c. surfaces.  Let 
$\SS/A$ be the family of  the corresponding index one covering Deligne-Mumford stacks.  

\begin{lem}\label{lem_universal_index_one_slc}
Let $\sS/A$ be a $\qq$-Gorenstein deformation family of  s.l.c. surfaces.  Let $\pi: \SS\to \sS$ be the corresponding  index one covering Deligne-Mumford stack and $\pi^{\lci}: \SS^{\lci}\to \sS$ be the corresponding $\lci$
 covering Deligne-Mumford stack.  For the diagram 
\begin{equation}\label{eqn_diagram_ind_lci2}
\xymatrix{
\SS^{\lci}\ar[r]^{\hat{\pi}}\ar[dr]_{\pi^{\lci}}& \SS\ar[d]^{\pi}\\
&\sS,&
}
\end{equation}
we have that  $(\hat{\pi})^*\omega_{\SS/A}\cong \omega_{\SS^{\lci}/A}$.
\end{lem}
\begin{proof}
For the isomorphism of the dualizing sheaves, note that for each fiber $S_t$ of the family $\sS/A$, the dualizing sheaf of the index one covering Deligne-Mumford stack  is 
$\omega_{\SS}\cong \omega_{S_t}^{[r]}$  for each singularity germ $(S_t,0)$, where $r$ is the index of the s.l.c. surface germ  $(S_t,0)$.  
We look at the diagram (\ref{eqn_diagram_ind_lci}) at any singularity germ.   For a germ singularity$(\sS, \mathbf{0})$,
let $\pi: \sZ\to \sS$ be the index one cover such that $[\sZ/\zz_r]\cong \SS$ and the diagram (\ref{eqn_diagram_ind_lci2}) is given by
\begin{equation}\label{eqn_diagram_ind_lci_local}
\xymatrix{
\sZ^{\lci}\ar[r]^{\hat{\pi}}\ar[dr]_{\pi^{\lci}}& \sZ\ar[d]^{\pi}\\
&\sS.&
}
\end{equation}

In the case that  $(\sS, \mathbf{0})$ is a simple elliptic singularity, cusp or degenerate cusp singularity, since  $(\sS, \mathbf{0})$ is Gorenstein, then the index one cover is itself;  i.e., $\sZ=\sS$.  In this case we only have the morphism 
$\hat{\pi}: \sZ^{\lci}\to \sS$ where $\sZ^{\lci}\to \sS$ is the discriminant cover with transformation group $G$ constructed in Theorem \ref{thm_minimally_elliptic_universal} and $\SS^{\lci}=[\sZ^{\lci}/G]$.
Since $\sZ^{\lci}$ is l.c.i.,  it follows that $(\pi^{\lci})^{*}\omega_{\sS}\cong \omega_{\sZ^{\lci}}$.  This is because the dualizing sheaves $\omega_{\sS},  \omega_{\sZ^{\lci}}$ can be given by the minimal resolutions:
\begin{equation}\label{eqn_minimal_resolutions_omega}
\xymatrix{
\sX^{\lci}\ar[r]^{\pi^{\lci}}\ar[d]_{\sigma}& \sX\ar[d]^{\sigma}\\
\sZ^{\lci}\ar[r]^{\pi^{\lci}}&\sS;&
}
\end{equation}
see \cite[Lemma 1.1]{Shepherd-Barron}. 

In the case that $(\sS, \mathbf{0})$ is the $\zz_2, \zz_3, \zz_4, \zz_6$-quotients of a simple elliptic singularity or the  $\zz_2$-quotient of a cusp or a degenerate cusp singularity,  we really have the diagram 
(\ref{eqn_diagram_ind_lci_local}) such that $(\sS, \mathbf{0})$ is a rational singularity.  
Then,  $\sZ^{\lci}\to \sS$ is the universal abelian cover with the transformation group 
$G=H_1(\Sigma,\zz)$ where $\Sigma$ is the link of the singularity.  Therefore $\SS^{\lci}\cong [\sZ^{\lci}/G]$.   
In this case $\omega_{\sZ}=\omega_{\sS}^{[r]}$ where $r$ is the index of the singularity.  The dualizing sheaf $\omega_{\SS}$ is the $\zz_N$-equivariant $\omega_{\sZ}$. 
Thus,  taken as the equivariant dualizing sheaves,  $\omega_{\sZ^{\lci}}\cong (\hat{\pi})^{*}\omega_{\sZ}$, which  can be seen from the minimal resolutions
in diagram (\ref{eqn_minimal_resolutions_omega}) again and  $\omega_{\sZ}$ is constructed from $\omega_{\sX}(A)$  where $A$ is the exceptional divisor. 
\end{proof}

\subsection{Flat family of  $\lci$ covering Deligne-Mumford stacks}\label{subsubsec_lci_covering_DM}

Motivated from the above construction we introduce the definition of $\lci$ covering Deligne-Mumford stack over a general base. 

\begin{defn}\label{defn_lci_DM_base}
A flat family of $\lci$ covering Deligne-Mumford stacks $\SS^{\lci}\to T$ over a scheme $T$ is a proper Deligne-Mumford stack $\SS^{\lci}$  over $T$ such that whenever there is a discrete valuation ring $R$ we have the following Cartesian diagram 
\[
\xymatrix{
\SS^{\lci}_{1}\ar[r]\ar[d]& \SS^{\lci}\ar[d]\\
\spec(R)\ar[r]& T
}
\]
and $\SS^{\lci}_{1}\to \spec(R)$ is a one-parameter family of lci covering Deligne-Mumford stacks in \S \ref{subsubsec_lci_covering_DM-flat}.
\end{defn}

From Theorem \ref{thm_lci_cover_link_smoothing},   and the above Definition \ref{defn_lci_DM_base}, we have the following flat families over any base scheme $T$.

\begin{thm}\label{thm_lci_cover_link_smoothing_T}
    Let $\overline{f}: \sS\to T$ be a flat family of s.l.c. surfaces and $f: \SS\to T$ be the associated index one covering Deligne-Mumford stack.  Suppose that the non-lci singularity germs $(\SS_0,x)$ of the central fiber  are not simple elliptic singularities of degree $6$ or $7$, or cusp singularities and degenerate cusps that can not be lifted to an equivariant smoothing of an $\lci$ cusp algebraically,  then $f: \SS\to T$ can be lifted to a flat family $f^{\lci}: \SS^{\lci}\to T$ of lci covering Deligne-Mumford stacks.  Moreover, $\sS$ is  the coarse moduli space of both $\SS$ and $\SS^{\lci}$.
\end{thm}
\begin{proof}
    This is from Definition \ref{defn_lci_DM_base} and Theorem \ref{thm_lci_cover_link_smoothing}.
\end{proof}

\begin{rmk}
In the one-parameter family $\SS^{\lci}_{1}\to \spec(R)$ of lci covering Deligne-Mumford stacks, the lci smoothing lifting of simple elliptic singularities, cusp and degenerate cusp singularities are given by the results in \S \ref{subsubsec_universal_discriminant}, \S \ref{subsubsec_discriminant}, \S \ref{subsubsec_equivariant_smoothing_cusp}.
\end{rmk}

Let $A$ be a $\mathbf{k}$-algebra (so that $T=\spec(A)$) and   $\SS^{\lci}/A$  be a flat family of $\lci$ covering Deligne-Mumford stacks. 
Let $\ll^{\bullet}_{\SS^{\lci}/A}$ be the cotangent complex of $\SS^{\lci}/A$ and let $J$  be a finite $A$-module.  
We also let $\pi^{\lci}: \SS^{\lci}\to \sS$ be the map to its coarse moduli space. 
Define
$$\widehat{T}_{\QG}^{i}(\sS/A,J):=\Ext^i(\ll^{\bullet}_{\SS^{\lci}/A},\sO_{\SS^{\lci}}\otimes_{A}J)$$
$$\widehat{\sT}_{\QG}^{i}(\sS/A,J):=\pi^{\lci}_{*}\sE xt^i(\ll^{\bullet}_{\SS^{\lci}/A},\sO_{\SS^{\lci}}\otimes_{A}J).$$

We have  similar results as in Proposition \ref{prop_deformation_S_SS_G} and  Proposition \ref{prop_local_deformation_obstruction_slc} for $\lci$ covering Deligne-Mumford stacks.

\begin{prop}
Let $\sS/A$ be a $\qq$-Gorenstein family of  s.l.c. surfaces.  The corresponding index one covering Deligne-Mumford stack and $\lci$ covering Deligne-Mumford stack are denoted by  $\SS/A$   and  $\SS^{\lci}/A$ respectively. 
Suppose that  $A^\prime\to A$ is an infinitesimal extension.  Let  $\sS^\prime/A^\prime$ be a $\qq$-Gorenstein deformation of $\sS/A$, and 
$\SS^\prime/A^\prime$ be the index one covering Deligne-Mumford stack.  Then we have 
\begin{enumerate}
\item 
$$\sS^\prime/A^\prime\mapsto \SS^\prime/A^\prime$$
give a bijection between the set of isomorphism classes of $\qq$-Gorenstein deformations of $\sS/A$ over $A^\prime$ and the set of isomorphism classes of deformations of 
$\SS/A$. 
\item 
any isomorphism class of the deformations $(\SS^{\prime})^{\lci}/A^\prime$  of the  $\lci$ covering Deligne-Mumford stack
induces an  isomorphism class of deformations of the  index one covering Deligne-Mumford stacks 
$$ (\SS^{\prime})^{\lci}/A^\prime\mapsto \SS^\prime/A^\prime$$
which in turn induces an  isomorphism class of $\qq$-Gorenstein deformations of $\sS/A$ over $A^\prime$
$$ (\SS^{\prime})^{\lci}/A^\prime\mapsto \sS^\prime/A^\prime.$$
\end{enumerate}
\end{prop}
\begin{proof}
The case $\sS^\prime/A^\prime\mapsto \SS^\prime/A^\prime$ for the index one covering Deligne-Mumford stack  is Proposition \ref{prop_deformation_S_SS_G}.  
For the second case, from Remark \ref{rmk_quotient_elliptic_cusp}, any deformation of the $\lci$ covering   Deligne-Mumford stack induces a 
$\qq$-Gorenstein deformation of the surface singularity $\sS/A$. 
\end{proof}

\begin{rmk}
We should point out again that it is not known whether any deformation of the index one covering Deligne-Mumford stack $\SS^\prime/A^\prime$ is induced by the deformation $(\SS^\prime)^{\lci}/A^\prime$ of the $\lci$ covering   Deligne-Mumford stack.
\end{rmk}

\begin{prop}\label{prop_local_deformation_obstruction_slc_univ}
Let $\sS_0/A_0$  be a  $\qq$-Gorenstein family  of s.l.c. surfaces and let $J$ be a finite $A_0$-module.
We let  $(\SS_0)^{\lci}/A_0$ be  the corresponding $\lci$ covering Deligne-Mumford stack. 
Then we have that 
\begin{enumerate}
\item  the set of isomorphism classes of  $\qq$-Gorenstein  deformations of $\sS_0/A_0$ which are induced from the deformations of the $\lci$ covering Deligne-Mumford stack $(\SS_0)^{\lci}/A_0$  over 
$A_0+J$ is naturally an $A_0$-module and is canonically isomorphic to $\widehat{T}_{\QG}^1(\sS/A, J)$. 
Here $A_0+J$ means the ring $A_0[J]$ with $J^2=0$;
\item  let $A^\prime\to A\to A_0$ be the infinitesimal extensions and the kernel of $A^\prime\to A$ is $J$. 
Let $\sS/A$  be a  $\qq$-Gorenstein of $\sS_0/A_0$.
Then  we have 
\begin{enumerate}
\item  there exists a canonical element $\ob(\sS/A, A^\prime)\in \widehat{T}_{\QG}^2(\sS/A, J)$ called the obstruction class.  It vanishes if and only if there exists a 
$G$-equivariant $\qq$-Gorenstein  deformation $\sS^\prime/A^\prime$ of $\sS/A$ over $A^\prime$ which is induced from the deformation of the $\lci$ covering Deligne-Mumford stack $(\SS^\prime)^{\lci}/A^\prime$ . 
\item if  $\ob(\sS/A, A^\prime)=0$,  then the set of  isomorphism classes of  $\qq$-Gorenstein  deformations $\sS^\prime/A^\prime$
is an affine space underlying $\widehat{T}_{\QG}^1(\sS_0/A_0, J)$.
\end{enumerate}
\end{enumerate}
\end{prop}
\begin{proof}
From Theorem \ref{thm_minimally_elliptic_universal}, this is a basic result of deformation and obstruction theory of algebraic varieties; see  \cite{Illusie2}. 
\end{proof}

\begin{lem}\label{lem_universal_higher_tangent}
Let $S$ be an s.l.c. surface, and $\SS^{\lci}\to S$ be  the $\lci$ covering Deligne-Mumford stack.  Then  we have that 
$\widehat{T}_{\QG}^{i}(S, \sO_S)=0$ for $i\ge 3$. 
\end{lem}
\begin{proof}
There is also a local to global spectral sequence
$$E_2^{p,q}=H^p(\widehat{\sT}_{\QG}^{q}(S, \sO_S))\Rightarrow \widehat{T}_{\QG}^{p+q}(S, \sO_S).$$
Since $S$ is of general type, the higher cohomology $H^p(F)=0$ for any sheaf $F$ and $p\ge 3$.  The sheaf $\widehat{\sT}_{\QG}^{q}(S, \sO_S)=0$ when $q\ge 2$ since 
$\SS^{\lci}$ only has l.c.i. singularities.  Thus,  from the local to global spectral sequence we get the result in the lemma. 
\end{proof}

\subsection{The moduli stack of  $\lci$ covers}\label{subsubsec_universal_DM}

We consider the families $\SS^{\lci}/T$ of $\lci$ covering Deligne-Mumford  stacks. 
In general it is interesting to look at the situation that a $\lci$ covering Deligne-Mumford stack admits smoothings with coarse moduli space the smoothings of s.l.c. surfaces.  
 Extending the result in \S \ref{subsec_index_one_cover} we define the   moduli stack $M_N^{\lci}$ of $\lci$ covers over the KSBA moduli stack 
$M_N$
of s.l.c. surfaces. 

\begin{defn}\label{defn_family_lci}
We define the flat  families over a scheme $T$ in the following diagram
\begin{equation}\label{eqn_diagram_ind_lci_family}
\xymatrix{
\SS^{\lci}\ar[rr]^{\hat{\pi}}\ar[dr]^{\pi^{\lci}}\ar[ddr]_{f^{\lci}}&&\SS\ar[dl]_{\pi}\ar[ddl]^{f}\\
&\sS\ar[d]_-{\overline{f}}&\\
&T.&
}
\end{equation}
which is the generalization of  family version of  Diagram \ref{eqn_diagram_ind_lci}, 
where 
\begin{enumerate}
\item  $\overline{f}: \sS\to T$ is a $\qq$-Gorenstein deformation  family of s.l.c. surfaces;
\item $f: \SS\to T$ is the corresponding index one covering Deligne-Mumford stack;
\item $f^{\lci}: \SS^{\lci}\to T$ is the lifting   $\lci$ covering Deligne-Mumford stack of $\overline{f}$, such that the morphism $\pi^{\lci}: \SS^{\lci}\to \sS$ factors through the morphism 
$\pi: \SS\to \sS$, with $\sS$ as the coarse moduli space of both $\SS^{\lci}$ and $\SS$;
\item the isomorphism classes   $\{\overline{f}: \sS\to T\}$  of the  families   must satisfy the 
conditions in (\ref{eqn_condition_functor_M}).
\item  for the flat family $f^{\lci}: \SS^{\lci}\to T$,  let $(S, x)$ be a singularity germ in $S=\overline{f}^{-1}(0)$ such that $(\widetilde{S},x)\to (S,x)$  is the $\lci$ cover with transformation group $G$.  We make the following conditions.
\begin{enumerate}
\item suppose that the flat family $\overline{f}: \sS\to T$ lies on the smoothing component  $M^{\sm}$ (i.e., the component containing smooth surfaces) of $M=\overline{M}_{K^2, \chi, N}$.  We may assume that  $\overline{f}: \sS\to T=\spec(\mathbf{k}[t])$ is a one-parameter smoothing of the singularity 
$(S,x)$.  If the $\lci$ cover $(\widetilde{S},x)$ locally is given by 
$$\spec \mathbf{k}[x_1, \cdots, x_{\ell}]/(h_1, \cdots, h_{\ell-2}),$$
then the flat family $f^{\lci}: \SS^{\lci}\to T$ is given by the $G$-equivariant smoothing of the singularity $(\widetilde{S},x)$ which is given by:
$$\spec \mathbf{k}[x_1, \cdots, x_{\ell}, t]/(h_1-t, \cdots, h_{\ell-2}-t),$$
where $G$ acts on $t$ trivially. 
The detail definition of the  smoothing component is in \S \ref{subsec_smoothing_component}.
\item suppose that the  flat family $\overline{f}: \sS\to T$ lies on a deformation  component of $M_N=\overline{M}_{K^2, \chi, N}$ containing the same type of singularities as $(S, x)$, then  we  require that the   flat family $f^{\lci}: \SS^{\lci}\to T$ induces the family $\overline{f}: \sS\to T$.
\end{enumerate}
\item for all the singularity germs $(S, x)$ in a family  $\overline{f}: \sS\to T$,  we let  $\Box_{\germs}$ be the set of singularity germs of   simple elliptic singularities,  cusp or degenerate cusp singularities, or  cyclic quotients of them  that do not satisfy the condition in  Condition \ref{condition_star}.   Then 
if the $\lci$ liftings $(\widetilde{S}, x)$ for $(S,x)\in \Box_{\germs}$ are nontrivial such that we have the Deligne-Mumford stack 
$[\widetilde{S}/G]$, and such  singularity germs belong to the $\lci$ cover constructed in Theorem \ref{thm_minimally_elliptic_universal2}, 
Lemma \ref{lem_cusp_cover}, Theorem  \ref{thm_elliptic_singularity_lci_lifting}, Theorem \ref{thm_minimally_elliptic_universal}, Theorem \ref{thm_equivariant_smoothing_e_cusp}, then in this case  both $\SS^{\lci}$ and $\SS$ have the same coarse moduli space $\sS$.
\end{enumerate}
\end{defn}

We define the  functor: 
$$\sM_N^{\lci}=\overline{\sM}^{\lci}_{K^2, \chi, N}: \Sch_{\mathbf{k}}\rightarrow \mbox{Groupoids}$$
which sends 
$$T\mapsto \{f^{\lci}: \SS^{\lci}\to T\}$$
where  $\{f^{\lci}: \SS^{\lci}\to T\}$ is the groupoid  of isomorphism classes of families of  $\lci$ covering  Deligne-Mumford stacks 
$\SS^{\lci}\to T$.  

\begin{rmk}\label{rmk_lci_coveringDM}
From the construction of $\lci$ covering Deligne-Mumford stack $\SS^{\lci}\to S$ in \S \ref{subsubsec_universal_discriminant} and \S \ref{subsubsec_discriminant} and  the family of $\lci$ covering Deligne-Mumford stacks in 
\S \ref{subsubsec_lci_covering_DM}, we only take the $\lci$ cover for an s.l.c. surface $S$ with simple elliptic singularities, cusp or degenerate cusp singularities,  or cyclic quotients of them with local embedding dimension $> 5$. 
\end{rmk}

Let $S_t$ be an s.l.c. surface such that its index one covering Deligne-Mumford stack $\SS_t\to S_t$ is  a fiber of $f: \SS\to T^\prime$. 
Look at the diagram (\ref{eqn_diagram_ind_lci_family}) again, 
from Lemma \ref{lem_universal_index_one_slc}, 
we have $(\hat{\pi})^*\omega_{\SS/A}\cong \omega_{\SS^{\lci}/A}$ (by taking $T^\prime=\spec(A)$ as one-dimension). Thus, we have 
$$K^2=K_{S_t}^2=\frac{1}{N^2}(\omega_{S_t}^{[N]}\cdot \omega_{S_t}^{[N]})=(\omega_{\SS_t}\cdot \omega_{\SS_t})=(\omega_{\SS_t^{\lci}}\cdot \omega_{\SS_t^{\lci}}),$$
where $N\in \zz_{>0}$ can be chosen to satisfy that  $\omega_{S_t}^{[N]}$ is invertible. 

Let  $M_N:=\overline{M}_{K^2, \chi, N}$ be the moduli functor which parametrizes the flat families  $\overline{f}: \sS\to T^\prime$ of  $\qq$-Gorenstein deformations of s.l.c. surfaces induced from the flat families 
$f^{\lci}: \SS^{\lci}\to T$ of $\lci$ covering Deligne-Mumford stacks in Definition \ref{defn_family_lci}.  Then $M$ is a projective Deligne-Mumford stack when $N$ is sufficiently large.

\begin{thm}\label{thm_universal_covering_stack} 
The functor $\sM_N^{\lci}$ represents a Deligne-Mumford stack $M_N^{\lci}:=\overline{M}^{\lci}_{K^2,\chi,N}$.  Moreover,  there exists a proper morphism 
$$f^{\lci}: M_N^{\lci}\to M_N.$$ 

In particular,  if $N$ is large divisible enough and $M:=M_N$ is proper, the stack $M^{\lci}:=M_N^{\lci}$ is a proper Deligne-Mumford stack with projective coarse moduli space. 
The morphism $f^{\lci}$ in the above diagram induces a proper  morphism on their coarse moduli spaces.
\end{thm}

\begin{proof}
The proof is from the above  construction of $\lci$ covering Deligne-Mumford stacks, and has the same method as in Theorem \ref{thm_index_one_covering_stack}. We prove the result for large divisible $N$, and the functor  $M^{\lci}$.

From \cite{Stack_Project}, the functor $\sM^{\lci}$ is a stack $M^{\lci}$.  There is a natural morphism 
$f^{\lci}: M^{\lci}\to M$ of stacks by sending any family $\{\SS^{\lci}\to T\}$ to the corresponding  family $\{\sS\to T^\prime\}$ in $M$. 
To show $M^{\lci}$ is a Deligne-Mumford stack, we show that the diagonal morphism 
$$M^{\lci}\to M^{\lci}\times_{\mathbf{k}}M^{\lci}$$
is representable and unramified.  This is from the following reason. 
If we have two objects $(f: \SS^{\lci}\to T)$ and $(f^\prime: (\SS^\prime)^{\lci}\to T)$  in $M^{\lci}(T)$, then the isomorphism functor of the two families 
$\textbf{Isom}_{T}(\SS^{\lci}, (\SS^\prime)^{\lci})$ is represented by a quasi-projective group scheme 
$\Isom_{T}(\SS^{\lci}, (\SS^\prime)^{\lci})$ over $T$. Let 
$(\overline{f}: \sS\to T^\prime)$ and $(\overline{f}^\prime: \sS^\prime\to T^\prime)$  be the corresponding $\qq$-Gorenstein deformation families of  s.l.c. surfaces over $T^\prime$. 
The isomorphism functor $\mathbf{Isom}_{T^\prime}(\sS,\sS^\prime)$ is represented by a quasi-projective group scheme 
$\Isom_{T^\prime}(\sS,\sS^\prime)$ over $T^\prime$.    
Look at the following diagram
\[
\xymatrix{
\SS^{\lci}\ar[r]^{\cong}\ar[d]& (\SS^\prime)^{\lci}\ar[d]\\
\sS\ar[r]^{\cong}& \sS^\prime.
}
\]
Any isomorphism $\SS^{\lci}\cong (\SS^\prime)^{\lci}$ induces an isomorphism $\sS\cong \sS^\prime$ on the coarse moduli spaces and the isomorphisms coming from the local stacky isotropy groups induces the same isomorphism on the coarse moduli spaces. 
Thus, the functor is represented by a quasi-projective scheme 
$\Isom_{T}(\SS^{\lci},(\SS^\prime)^{\lci})$  over $\Isom_{T^\prime}(\sS,\sS^\prime)$ and is also unramified over $T$  since its geometric fibers are finite. 

From the proof of Theorem \ref{thm_index_one_covering_stack}, there is a cover $\varphi: \C\to M$. Then the fiber product $\C^{\lci}$ in the diagram 
\[
\xymatrix{
\C^{\lci}\ar[r]\ar[d]& M^{\lci}\ar[d]\\
\C\ar[r]& M
}
\]
serves as a cover over the  stack $M^{\lci}$. 
This is because for a given family of s.l.c. surface $\sS/T$, there is a family  $\SS^{\lci}/T$ of $\lci$ covering Deligne-Mumford stacks.   

Suppose that $M=M_N$ is proper, 
which means that all the s.l.c. singualrities of embedding dimension $\ge 6$ can lifted to deofrmation of lci covers. 
We show that the morphism $f^{\lci}: M^{\lci}\to M$ is proper.   We use the valuative criterion for properness and consider the following diagram
\[
\xymatrix{
\spec(K^\prime)\ar[r]\ar[d]&\spec(K)\ar[r]\ar[d]& M^{\lci}\ar[d]^{f^{\lci}}\\
\spec(R^\prime)\ar[r]\ar@{-->}[urr]&\spec(R)\ar@{-->}[ur]\ar[r]& M
}
\]
where $R$ is a valuation ring with field of fractions $K$, and residue field $\mathbf{k}$.  In this case we can take $R=\mathbf{k}[\![t]\!]$ and $K=\mathbf{k}(\!(t)\!)$. 
The morphism $\spec(R)\to M$  corresponds a flat $\qq$-Gorenstein family  $f: \sS\to \spec(R)$  of s.l.c. surfaces. 
We may assume that $\spec(R)\to M$ lies on the smoothing component of the moduli stack $M$, since if $\spec(R)\to M$ lies in other component of $M$, then from condition (6) in 
Definition \ref{defn_family_lci} we always have that the  family  $f: \sS\to \spec(R)$ is induced from a flat family of lci covering Deligne-Mumford stacks. 

Now let $S$ be the s.l.c. surface over $0=\spec(\mathbf{k})$ in the family  $f: \sS\to \spec(R)$. 
Over a singularity germ $(S, x)$ in $\sS$, we assume that the singularity is given by 
$$\spec(\mathbf{k}[x_1, \cdots, x_s]/I),$$ 
where $I$ is the ideal of the singularity.  Let $I=(g_1, \cdots, g_l)$ be the generators.  Then the singularity germ $(\sS, x)$ is given by 
$$\spec(R[x_1, \cdots, x_s, t]/I_t),$$
where $I_t=(g_1^t, \cdots, g_l^t)$ and $g_i^t$ are polynomials involving $t$.  Taking $t=0$, we get $g_i^t=g_i$.   Since the family  $f: \sS\to \spec(R)$ is flat,  the parameter $t$ can not happen in the factors of the monomial terms of $g_i^t$.   For instance, we choose 
 $I_t=(g_1-t, \cdots, g_l-t)$. 

Suppose that $f^{\lci}: \SS^{\lci}\to \spec(K)$ is the lifting of $f: \sS\to \spec(R)$ to a $\lci$-covering Deligne-Mumford stacks at the generic point. Then 
over the the singularity germ $(\sS, x)$, we have that the local  $\lci$ cover  $(\widetilde{\sS}, x)$ is given by 
$$\spec(K[x_1, \cdots, x_{\ell}]/J_t),$$
where $J_t=(h_1^t, \cdots, h_{\ell-2}^t)$ and $h_i^t$ are polynomials involving the variable $t$. 
Here the ideal $J_t$ has $\ell-2$ generators since the singularity $(\widetilde{\sS}, x)$ is an l.c.i. singularity.
The quotient of $\spec(K[x_1, \cdots, x_{\ell}]/J_t)$ by the finite  transformation group $G$ gives $\spec(K[x_1, \cdots, x_s]/I_t)$, or equivalently,  the  invariant ring $\left(K[x_1, \cdots, x_{\ell}]/J_t\right)^G$ by the  transformation group $G$ gives $K[x_1, \cdots, x_s]/I_t$.   
 
The finite group $G$ acts on the variety $\spec(K[x_1, \cdots, x_{\ell}]/J_t)$. 
The field $K$ is the fraction field of $R$ with the uniformizer $t$.   
The generators $h_j^{t}$ for $1\leq j\leq \ell-2$ may contain  powers of $t$. We let $I$ be the index set such that 
for $i\in I$, $c_i\in \zz$, and $t^{c_i}$ is a factor of some term in $h_j^{t}$.   Note that $c_i$ may be negative at the moment. 
Let $d\in \zz_{>0}$ be a large integer depending on the set $\{c_i| i\in I\}$.  We  take the  finite cover 
$$\spec(R^\prime)\to \spec(R)$$ 
by 
$$t\mapsto t^{\prime d}.$$
Let $K^\prime$ be the field of fractions of $R^\prime$.  We choose $d$ large enough so that the group $G$ acts on the parameter $t^\prime$ trivially. 
Now 
 the polynomials $h_j^{t}$ for $1\leq j\leq \ell-2$ become $h_j^{t^{\prime}}$ for $1\leq j\leq \ell-2$.
Since the singularity germ $(S, x)$ is given by an $\lci$ cover   $(\widetilde{S}, x)$, and $G$ acts on the parameter $t^\prime$ trivially, then from 
condition (5) in 
Definition \ref{defn_family_lci}, 
the $G$-equivariant smoothing of the $\lci$ cover   $(\widetilde{S}, x)$ is given by 
$$\spec(K^\prime[x_1, \cdots, x_{\ell}]/J_{t^\prime}).$$
The generators $h_j^{t^\prime}=h_j-t^\prime$. 
The morphism $\spec(K)\to M^{\lci}$ naturally extends to the morphism $\spec(K^\prime)\to M^{\lci}$. 
Therefore, taking $t^\prime=0$, we get the $\lci$ covering Deligne-Mumford stack $[\widetilde{S}/G]$ for the s.l.c. surface $S$.  This gives the  unique morphism $\spec(R^\prime)\to M^{\lci}$ which completes 
the valuative criterion for properness. 

If $N$ is large divisible enough, then the stack $M$ is a proper Deligne-Mumford stack with projective coarse moduli space.  
When we fix the volume $K^2$ of the s.l.c. surface $S$ and the $\lci$ covering Deligne-Mumford stack $\SS^{\lci}$, the families the $\lci$ covering Deligne-Mumford stacks form a bounded family. 
Therefore,  
the  morphism $f^{\lci}: M^{\lci}\to M$
in the diagram induces a proper morphism on their coarse moduli spaces since  $f^{\lci}$ is proper.   

Finally we study the fiber of the morphism $f^{\lci}: M^{\lci}\to M$.  
Let $S\in M$ be an s.l.c. surface and $(S,x)$ be an s.l.c. singularity germ.  We aim to study the fiber 
$(f^{\lci})^{-1}(S)$.  
Since there are only two cases for the log canonical surface singularities $(S,x)$ which need to take the $\lci$ covers.  We prove it by  cases. 

Case 1.  If the singularity germ $(S,x)$ has index bigger than $1$, then it is either the $\zz_2, \zz_3, \zz_4, \zz_6$-quotient of a simple elliptic singularity,  the $\zz_2$-quotient of a cusp, or the 
$\zz_2$-quotient of a degenerate cusp singularity.   Then from Theorem \ref{thm_minimally_elliptic_universal2},  the singularity $(S, x)$ is a rational singularity and the $\lci$ cover is the universal abelian cover which is unique.   Thus $(f^{\lci})^{-1}(S)$ only contains one geometric  element. 

Case 2.   If the singularity germ $(S,x)$ has index $1$, then it is either  a simple elliptic singularity,   a cusp, or  a degenerate cusp singularity.  
From Theorem \ref{thm_minimally_elliptic_universal}, since  we take the $\lci$ cover for a  degenerate cusp singularity $(S, x)$ using the $\lci$ degenerate cusp (which must have embedding dimension $3, 4$), thus the $\lci$ smoothing liftings are finite. 

For the case of a simple elliptic singularity $(S, x)$ with degree $d$,   it admits a smoothing and therefore, $1\le d\le 9$. From Theorem \ref{thm_elliptic_singularity_lci_lifting},  if $d=8, 9$, then  the $\lci$ covers 
$(\widetilde{S}, x)$ will reduce the negative self-intersection number of the exceptional elliptic curve.   Therefore, there are only finite $\lci$ covers $(\widetilde{S}, x)$ such that the self-intersection number becomes 
$1, 2, 3, 4$ which imply the singularities $(\widetilde{S}, x)$ are  l.c.i. Thus,   there are finite  $\lci$ liftings for the singularity germ $(S,x)$.

The last case is the cusp singularity $(S,x)$ which is a bit complicated. 
If the $\lci$ covers are from Theorem \ref{thm_equivariant_smoothing_e_cusp},  it is not hard to see that the $\lci$ lifting is unique. 

In other cases such that the smoothing of the cusp admits a smoothing lifting by an $\lci$ cusp, let $\Sigma$ be the link of the singularity $(S,x)$, and $\pi_1(\Sigma)=\zz^2\rtimes\zz$ be the fundamental group. 
From the proof of 
Lemma \ref{lem_cusp_cover} in \S \ref{subsubsec_discriminant}, we form the following diagram:
\[
\xymatrix{
H\rtimes\zz\ar[r]\ar@{^{(}->}[d]& \zz\oplus \tau\ar@{^{(}->}[d]\ar[r] & 0\\
\zz^2\rtimes \zz=\pi_1(\Sigma)\ar[d]\ar[r] & H_1(\Sigma,\zz)\ar[r]\ar[d]&0\\
G\ar[r]\ar[d]& K \ar[r] \ar[d]& 0\\
0& 0, 
}
\]
where 
$H\subseteq \zz^2$ is the subgroup generated by $\mat{c}a\\c\rix$, $\mat{c}0\\1\rix$. 
Here  $\mat{cc}a&b\\c&d\rix$ is the monodromy matrix of the cusp,   
$\zz\oplus \tau$ is the abelianization of $H\rtimes \zz$, and $H_1(\Sigma,\zz)=\zz\oplus (H_1(\Sigma,\zz))_{\tor}$. 
The transformation group $G^\prime$ for the $\lci$ cover 
$(\widetilde{S}, x)$ is obtained from $G$ by taking discriminant cover.  
Since there are finite morphisms $\Hom(\pi_1(\Sigma), G)$,  we conclude that there are only finite possibilities for the covers determined by 
$H\rtimes \zz$.  Therefore,  the preimage  $(f^{\lci})^{-1}(S)$ contains only finite elements.  This proves that the morphism $f^{\lci}$ is finite over such a singularity germ. 
\end{proof}

\begin{rmk}
\begin{enumerate}
\item  We only take  $\lci$ covers $(\widetilde{S}, x)\to (S,x)$ for the simple elliptic, cusp and degenerate cusp singularities $(S, x)$ with local embedding dimension $\ge 5$.  For such singularities, from the construction in \S \ref{subsubsec_universal_discriminant}, \S \ref{subsubsec_discriminant},  Example \ref{example_Pinkham} and Example \ref{example_quotient_cusp_family},  the $\lci$ cover $(\widetilde{S}, x)$ is always a locally complete intersection singularity with the transformation group $G$-action.   Since a locally complete intersection singularity admits a $G$-equivariant smoothing (which takes the action trivial on the parameter $t$), the quotient gives the $\qq$-Gorenstein smoothing of 
$(S, x)$.  The situation exactly matches the condition (5) in Definition \ref{defn_family_lci}. Thus the valuative criterion for properness always holds case by case for such singularities. 

\item  for the singularity germs $(S, x)$ with local embedding dimension $\ge 5$, if there can not have $\lci$ covers in  \S \ref{subsubsec_universal_discriminant}, \S \ref{subsubsec_discriminant},  then we will construct the stack of fake lci covers to cover all these singularities.  

\item The morphism $f^{\lci}: M^{\lci}\to M$ is not necessarily  representable.  
\end{enumerate}
\end{rmk}

\begin{cor}\label{cor_universal_index_one_M}
Let $M=\overline{M}_{K^2, \chi, N}$ be a connected component of the moduli stack of stable s.l.c. surfaces with invariants  $K^2, \chi, N$. 
If  each s.l.c. surface in $M$ satisfies Condition \ref{condition_star}, then   the  moduli stack $M^{\lci}$ of $\lci$ covers is the same as  $M^{\ind}$. 
If moreover,  every s.l.c. surface $S\in M$ is l.c.i., then  $M^{\lci}=M^{\ind}=M$.
\end{cor}
\begin{proof}
The corollary is from the construction of the $\lci$ covering Deligne-Mumford stacks. 
\end{proof}

\section{Moduli stack of birational lci enhancements}\label{sec_moduli_fake}

In this section we give a method to   construct the moduli stack
$M_N^{\mathrm{bir}}:=\overline{M}^{\mathrm{bir}}_{K^2, \chi,N}$  of birational lci enhancements and its associated algebraic  space.  We take birational $\lci$ enhancement for all non $\lci$ slc singularities in their smoothing.

\subsection{Covering surface slc singularities from birational enhancement}\label{subaec_normal_nonnormal}

Let $f: \sS\to C$ be a smoothing or deformation of non lci s.l.c. singularities.  In this section, wothout taking smoothing of  lci covers (which do not exists for some simple elliptic, cusp and degenerate cusp singularities), we use birational crepant resolutions to construct a new family. 

The useful replacement is a \emph{crepant terminalization},
not necessarily a smooth crepant resolution. In the algebraic
one-parameter setting there is a positive result that also
applies to Gorenstein degenerate cusps.

\begin{thm}
\label{thm:terminalization}
Let $f:\sS\to C$ be a flat morphism of complex algebraic varieties,
where $C$ is a smooth curve with marked point $0$. Assume that
$\sS$ is quasi-projective, the fibres are surfaces, $f$ is smooth
over $C\setminus\{0\}$, and the central fibre $S=\sS_0$ is
reduced, Gorenstein, and slc. Work near $S$, shrinking if needed.
Assume in addition that the total space $\sS$ is normal and
Gorenstein. 
Then there is a projective crepant birational morphism
\[
 q:\sX\longrightarrow \sS
\]
with $\sX$ Gorenstein terminal, which is an isomorphism over
$C\setminus\{0\}$. The composition $g=fq$ is flat, and its
fibres are reduced Gorenstein slc lci surfaces. In particular,
\[
q_0: \sX_0\longrightarrow S
\]
is a proper degree-one contraction with
\[
 \omega_{\sX/C}\simeq q^*\omega_{\sS/C},\qquad
 \omega_{\sX_0}\simeq q_0^*\omega_S.
\]
The central fibre $\sX_0$ may have extra components contracted by
$q_0$. The total-space morphism is birational; $q_0$ need not
be birational in the strict componentwise sense for reducible
surfaces.
\end{thm}
\begin{proof}
By hypothesis $\sS$ is normal and Gorenstein. The smoothness of the
generic points of the components of $S$ is the usual way the
regularity-in-codimension-one part of this hypothesis is checked
in an actual smoothing.

Since $S$ is a reduced Cartier divisor and is slc, inversion
of adjunction gives that $(\sS,S)$ is log canonical near $S$
\cite{Kawakita,FujinoHashizume}. Here adjunction to the
normalization of $S$ includes its conductor. For an exceptional
divisor $E$ over $\sS$ with center in $S$, write
$m_E=\operatorname{ord}_E(S)\geq1$. Then
\[
 a(E,\sS)=a(E,\sS,S)+m_E\geq-1+m_E\geq0.
\]
If its center is not contained in $S$, its generic centre is
in the smooth locus of $\sS$ and its discrepancy is positive.
Hence $\sS$ is canonical and Gorenstein.

Reid's crepant terminalization theorem supplies a projective
$q:\sX\to \sS$ with terminal singularities and
$K_{\sX}=q^*K_{\sS}$ \cite[Theorem (0.6)]{ReidMinimal}. It can be
taken to be an isomorphism over the smooth locus. Because
$K_{\sS}$ is Cartier, $K_{\sX}$ is Cartier. Terminal singularities
are Cohen--Macaulay, so $\sX$ is Gorenstein. Its singular points
are isolated cDV hypersurface points
\cite[Theorem (1.1)]{ReidMinimal}; thus $\sX$ is lci.

The parameter of $C$ is a nonzerodivisor on $\sX$, so $g$ is
flat. Every exceptional prime divisor of $q$ has discrepancy
zero over $\sS$. Since $(\sS,S)$ is log canonical, for such a
divisor above $S$,
\[
-\operatorname{ord}_E(S)\geq-1.
\]
Its positive integral multiplicity is therefore one. Strict
transforms of components of $S$ also have multiplicity one.
Consequently $q^*S$ is reduced as a divisor. As a Cartier
divisor on the Cohen--Macaulay space $\sX$, its scheme has no
embedded components, so the scheme-theoretic fibre $T$ is
reduced as well.

We have $K_{\sX}+\sX_0=q^*(K_{\sS}+S)$, hence $(\sX,\sX_0)$ is log canonical.
Adjunction gives that $\sX_0$ is slc. It is also an lci surface:
its local ring is obtained from an lci threefold ring by one
nonzerodivisor. Adjunction gives the displayed canonical-sheaf
identities. The remaining fibres are the original smooth
fibres because $q$ is an isomorphism over the smooth locus.
\end{proof}

\begin{example}

Let $Y$ be a smooth del Pezzo surface of degree $6$ or $7$,
and choose a smooth member $D=(s=0)\in|-K_Y|$. Set
\[
 V=\Tot(K_Y),\qquad
 t:V\longrightarrow\mathbb A^1,\quad t(v)=\langle v,s\rangle.
\]
The total space is the canonical bundle $K_Y$, not $-K_Y$.
It is smooth and has trivial canonical bundle. For the
anticanonical embedding, it maps properly to the affine cone
$C(Y)$ by contracting its zero section; this is the usual
blowup of the vertex. The section $s$ is a degree-one
function on $C(Y)$, and $t$ is its pullback. We have a diagram:
\[
\xymatrix{
\Tot(K_Y)\ar[r]^{t}\ar[d]_{\pi}& \mathbb{A}^1\\
C(Y)\ar[ur]_{t|_{C(Y)}}.&
}
\]

The central fibre is
\begin{equation}\label{eq:delpezzo-central}
 t^{-1}(0)=Y\ \cup_D\ \Tot(K_Y|_D).
\end{equation}
Near $D$, take a local equation $u$ for $D$ and a fibre
coordinate $v$ for $K_Y$. Then the family is
\begin{equation}\label{eq:delpezzo-family}
                         uv=t.
\end{equation}
Away from $D$ it is smooth. Hence the family is flat and has
lci fibres, and the central fibre has simple normal crossings.
For $t\neq0$ the fibre is isomorphic to $Y\setminus D$.
The corresponding hyperplane section of $C(Y)$ at $t=0$ is
the simple elliptic cone over $(D,-K_Y|_D)$, of degree $K_Y^2$.

This is a concrete lci \emph{birational lifting} of the cone
family, with local equation \eqref{eq:delpezzo-family}. The
map of central fibres contracts an entire surface component
and is not a finite quotient cover. It is a local model, not a
proper family of stable surfaces by itself. A compact moduli
application needs global gluing and stability data.
One may take $Y=\operatorname{Bl}_{p_1,p_2,p_3}\mathbb P^2$
with the points noncollinear in degree six, and
$Y=\operatorname{Bl}_{p_1,p_2}\mathbb P^2$ with distinct points
in degree seven.
\end{example}

\subsection{The moduli stack of birational $\lci$ enhancements}\label{subsec_moduli_birational_enhancement}

Let $M_N$ be the KSBA moduli stack of general type surfaces as before. 

\begin{defn}[Canonical enhanced objects]
\label{def:canonical-enhanced}
An enhanced object over a scheme $T$ consists of
\[
 \tau:=(\sS/T,\SS/T,\XX^{\lci}/T, q,\Delta_{\XX^{\lci}},\mathscr L,\iota, P),
\]
where $\overline{f}: \sS\to T$ is a KSBA family of slc surfaces, $f: \SS\to T$ the index one covering Deligne-Mumford stack in $M_N$, and  $f^{\lci}: \XX^{\lci}\to T$ is a proper flat tame
Deligne-Mumford family, which is $\lci$
(equivalently, in the flat finite-presentation setting, its geometric
fibres are $\lci$). 

The symbol $q$ denotes  a  proper degree-one
contraction 
$q:\XX^{\lci}\to \SS$  so that there is a commutative diagram 
\[
\xymatrix{
\XX^{\lci}\ar[r]^{q}\ar[dr]_{\pi^{\lci}} & \SS\ar[d]^{\pi}\\
& \sS.
}
\]
We call $f^{\lci}: \XX^{\lci}\to T$ a flat family of  birational $\lci$ enhancement Deligne-Mumford stacks. 

The data $\Delta_{\XX^{\lci}}$ record conductor,
branch, and marked loci, and $\iota$ records the group action,
grading, polarization, and all identifications with the special
fibre. Here $P$ is a Hilbert polynomial.

Assume that the relative log canonical algebra
\[
 \mathcal R(\XX^{\lci}/T,\Delta_{\XX^{\lci}})
   =\bigoplus_{m\geq0}
      p_*\mathcal O_{\XX^{\lci}}\bigl(m(K_{\XX^{\lci}/T}
                                      +\Delta_{\XX^{\lci}})\bigr)
\]
for all sufficiently divisible $m$ is finitely generated and
commutes with arbitrary base change, and that the coarse pair
of $(\XX^{\lci},\Delta_{\XX^{\lci}})$ is slc, while
$K_{\XX^{\lci}/T}+\Delta_{\XX^{\lci}}$ is $\mathbb Q$-Cartier and
relatively ample. (Reflexive powers are understood when the pair is
not Gorenstein.)
The object is called the \emph{canonical} model $\XX^{\mathrm{can}}$ when its full covering stack
equals the canonical covering stack associated with its relative log
canonical model. 

The polarized line bundle  $\mathscr L$ descends to a relatively very ample line bundle $L$ on the coarse family, and $\chi(X_t^{\lci},L_t^{\otimes q})=P(q)$ for every geometric point $t\to T$.  Under the proper morphism $q: \XX^{\lci}\to \SS$ the Hilbert polynomial $P$ induces the Hilbert polynomial $P^{KSBA}$ for the moduli stack $M_N$.
\end{defn}

From Theorem  \ref{thm:terminalization}, let 
$$
\overline{f}: \sS\to C
$$
be a flat family of slc surfaces so that the central fiber contains non-lci slc singualrities,  then after taking the birational crepant terminalization  we get $f: \sX\to C$, then taking the index one covering Deligne-Mumford stack we get a family 
$$
f^{\lci}: \XX^{\lci}\to C
$$
as in Definition \ref{def:canonical-enhanced} which admits a proper morphism $q: \XX^{\lci}\to \SS$, where $f: \SS\to C$ is the index one covering Deligne-Mumford stack of $\sS\to C$. 

\begin{defn}\label{defn_morphisms_enhancement}
    An isomorphism   between two objects over $T$ is an isomorphism of the birational $\lci$ enhancement stacks commuting with the contractions, carrying $\mathscr L$ to $\mathscr L'$.  Pullback along $T'\to T$ is defined componentwise. These objects and arrows form a category fibered in groupoids, denoted
\[
 M_N^{\mathrm{Bir}}:= M^{\mathrm{Bir},\mathrm{pol}}_{\tau}\longrightarrow\Sch_{\cc}.
\]
\end{defn}

\begin{thm}\label{thm_stack_Bir_enhancement}
    $M_N^{\mathrm{Bir}}$ is an algebraic Deligne-Mumford stack of finite type.   In particular, it is not separated, but as an algebraic space, it is compact when $N$ is sufficiently large and divisible. 
\end{thm}
\begin{proof}
   The deformation obstruction theory of birational $\lci$ enhancement $f^{\lci}: \XX^{\lci}\to T$ is similar to the $\lci$ covering Deligne-Mumford stacks in Proposition \ref{prop_local_deformation_obstruction_slc_univ}.   Infinitesimal automorphisms lie in $T^0$, isomorphism classes of lifts, when nonempty, form a torsor under $T^1$, and the obstruction to a lift lies in $T^2$.  

   Under the effective-stack and deformation-rigidification conventions above, every geometric object of $M_{N}^{\mathrm{Bir}}$ has a finite automorphism group. This is due to stable slc pairs with ample log canonical divisor have finite automorphism groups. Over $\cc$ its automorphism group scheme is finite and reduced, hence unramified. 

   The relatively very ample line bundle has Hilbert polynomial $P$. Hence the coarse surfaces occur in a fixed Hilbert scheme.  This implies boundedness for the moduli functor.  
   Stack axioms imply that $M_N^{\mathrm{Bir}}$ is a Deligne-Mumford stack of finite type.  

   The stack  $M_N^{\mathrm{Bir}}$ is in general not separated.  This is because different flat families $f_1^{\lci}: \XX_1^{\lci}\to T$
  and  $f_2^{\lci}: \XX_2^{\lci}\to T$ of birational $\lci$ enhancements maybe related by flops, so that they have the same generic fibers but different central fibers. 
   
   If $N$ is sufficiently large and divisible,  $M_N$ is proper, and every flat family of stable slc surfaces has a birational $\lci$ enhancement, so $M_N^{\mathrm{Bir}}$ is compact. 
\end{proof}

There is a universal family
\begin{equation}\label{eqn_universal_family_Bir}
    p^{\mathrm{Bir}}: \mathscr{X}^{\mathrm{Bir}}\to M_N^{\mathrm{Bir}}
\end{equation}
for the stack $M_N^{\mathrm{Bir}}$.

\subsection{The algebraic space}
If two flat families $f_1^{\lci}: \XX_1^{\lci}\to T$
  and  $f_2^{\lci}: \XX_2^{\lci}\to T$ of birational $\lci$ enhancements have the isomorphic generic fibers and different central fibers (for instance they are related by flops), 
we call their central fibers $\XX^{\lci}_{1,0}$ and  $\XX^{\lci}_{2,0}$ to be $\mathbb{S}$-equivalent. 

In order to get a separated moduli space of birational $\lci$ enhancements, we need to use the $\mathbb{S}$-equivalence. 
We choose  an \'etale atlas \(U\to M_N^{\mathrm{Bir}}\).
Define a relation on families\[
   R\hookrightarrow U\times U
   \]
   expressing the $\mathbb{S}$-equivalence, including ordinary isomorphisms. This is compatible with base change and descent. Here \(R\) records related pairs, rather than remembering each possible chain of flops.
Form the fppf quotient sheaf\[
   M_N^{\mathrm{alg}}=(U/R)_{\mathrm{fppf}}.
   \]
Since 
 \(R\) is represented by an algebraic space, is an equivalence relation, and both projections \(R\rightrightarrows U\) are flat and locally of finite presentation. Then the quotient is an algebraic space. These conditions are sufficient, not necessary, see \cite[Stacks Project-Tag 04S6]{Stack_Project}.
 Thus we get 

\begin{thm}\label{thm_stack_good_enhancement}
   There exists an algebraic  space  $M_N^{\mathrm{alg}}$, and a morphism 
   $$
   \pi: M_N^{\mathrm{Bir}}\to M_N^{\mathrm{alg}}
   $$
such that if $N$ is sufficiently large and divisible, the algebraic  space is proper.  Moreover, there is a contraction map $c: M_N^{\mathrm{alg}}\to M_N$ to the KSBA moduli space which is proper. 
\end{thm}

The universal family
 $p^{\mathrm{Bir}}: \mathscr{X}^{\mathrm{Bir}}\to M_N^{\mathrm{Bir}}$
in (\ref{eqn_universal_family_Bir})
induces a 
universal family
\begin{equation}\label{eqn_universal_family_good}
    p^{\mathrm{alg}}: \mathscr{M}^{\mathrm{alg}}\to M_N^{\mathrm{alg}}
\end{equation}
where over each point $P\in M_N^{\mathrm{alg}}$, the fiber $(p^{\mathrm{alg}})^{-1}(P)$ is one representative in the $\mathbb{S}$-equivalence class corresponding to $P$.

\begin{rmk}
    Another way to get {\em separated} moduli space from the flopping issues is to take the relative canonical model for the flopping families $f: \XX^{\lci}\to T$, and construct the moduli space of canonical models.  But this may receive bad slc singularities for the canonical model and a perfect obstruction theory can not exist on such moduli spaces. 
\end{rmk}

\section{The  virtual fundamental class}\label{sec_semi-POT_MG}

In this section we construct the virtual fundamental class for the KSBA moduli stacks. 

\subsection{Universal family}\label{subsec_universal_family}

The two moduli stacks $M^{\lci}_N$, and $M^{\mathrm{Bir}}_N$ are fine moduli stacks.   We have the universal families:
\begin{equation}\label{eqn_universal_lci}
    p^{\lci}: \mathscr{M}^{\lci}\to M^{\lci}:=M_N^{\lci};
\end{equation}
and 
    $$p^{\mathrm{Bir}}: \mathscr{M}^{\mathrm{Bir}}\to M_N^{\mathrm{Bir}}
$$ in (\ref{eqn_universal_family_Bir}).  The algebraic space  $M^{\mathrm{alg}}_N$ is not a fine moduli space, but there is also a family
$$
    p^{\mathrm{alg}}: \mathscr{M}^{\mathrm{alg}}\to M_N^{\mathrm{alg}}
$$ in (\ref{eqn_universal_family_good}).

The following sections work for these three  families. Each produce a virtual fundamental class in the base moduli stack.   We only work on the moduli stack of lci covers in this section. 

For the  universal family $p^{\mathrm{Bir}}: \mathscr{M}^{\mathrm{Bir}}\to M_N^{\mathrm{Bir}}$, the argument for perfect obstruction theory and virtual class work, and the only difference is that the stack $\mathscr{M}^{\mathrm{Bir}}$ is nonseparated.  But a virtual fundamental class exists on this moduli stack. 

The  issue for the family  
$p^{\mathrm{alg}}: \mathscr{M}^{\mathrm{alg}}\to M_N^{\mathrm{alg}}
$ 
in (\ref{eqn_universal_family_good}) is a little subtle.  All the argument in this section works for $p^{\mathrm{alg}}$ to produce a virtual fundamental class on $M_N^{\mathrm{alg}}$. 
But since we only choose one representative for each $\mathbb{S}$-equivalence class in $M_N^{\mathrm{alg}}$, we have to prove that the perfect obstruction theory and the virtual fundamental class are independent to the choice of the representatives in the $\mathbb{S}$-equivalence class.  This is proved in the appendix \ref{sec_POT_equivalence}.

\subsection{Perfect obstruction theory}\label{subsubsec_POT}

In this section we prove  there is a perfect obstruction theory on the moduli stack  $M^{\lci}$ of $\lci$ covers.  Similar arguments work for the stacks  $M_N^{\mathrm{Bir}}$ and $M_N^{\mathrm{alg}}$.
 
Let  $\ll^{\bullet}_{\scM^{\lci}/M^{\lci}}$ be the relative cotangent complex of $p^{\lci}$  and 
$\omega^{\lci}:=\omega_{\scM^{\lci}/M^{\lci}}[2]$. We consider 
$$E^{\bullet}_{M^{\lci}}:=Rp^{\lci}_{*}\left(\ll^{\bullet}_{\scM^{\lci}/M^{\lci}}\otimes\omega^{\lci}\right)[-1].$$
The relative dualizing sheaf $\omega_{\scM^{\lci}/M^{\lci}}$  satisfies the property
$$\omega_{\scM^{\lci}/M^{\lci}}|_{(p^{\lci})^{-1}(t)}\cong \omega_{\SS^{\lci}_t},$$
where $\omega_{\SS^{\lci}_t}$ is the dualizing sheaf of the  $\lci$ covering Deligne-Mumford stack $\SS^{\lci}_t$ which is 
invertible. 

When restricting to the smoothing component $M^{\sm}\subset M$, we get the universal family $p^{\lci,\sm}: \scM^{\lci,\sm}\to M^{\lci,\sm}$ and the complex 
$$E^{\bullet}_{M^{\lci,\sm}}:=Rp^{\lci,\sm}_{*}\left(\ll^{\bullet}_{\scM^{\lci,\sm}/M^{\lci,\sm}}\otimes\omega^{\lci}\right)[-1].$$

\begin{thm}\label{thm_POT_Muniv}
Assume that $N$ is sufficiently large divisible. 
Let $M=\overline{M}_{K^2,\chi, N}$ be a connected component of the moduli space of  stable s.l.c. general type surfaces with invariants 
$K^2, \chi, N$, and $f^{\lci}: M^{\lci}\to M$ be  the moduli stack of $\lci$ covers over $M$.   Then the complex $E^{\bullet}_{M^{\lci}}$ defines a perfect   obstruction theory (in the sense of Behrend-Fantechi)
$$\phi^{\lci}: E^{\bullet}_{M^{\lci}}\to \ll_{M^{\lci}}^{\bullet}$$
induced by the Kodaira-Spencer map $\ll^{\bullet}_{\scM^{\lci}/M^{\lci}}\to (p^{\lci})^{*}\ll_{M^{\lci}}^{\bullet}[1]$.

If we restrict the perfect obstruction theory  $\phi^{\lci}$ to the smoothing component, we get a perfect   obstruction theory 
$$\phi^{\lci,\sm}: E^{\bullet}_{M_{\eq}^{\lci,\sm}}\to \ll_{M_{\eq}^{\lci,\sm}}^{\bullet}.$$
\end{thm}
\begin{proof}
Since the universal family $p^{\lci}$ is a flat, projective and relatively Gorenstein morphism between Deligne-Mumford stacks, Theorem \ref{thm_BF_prop6.1} (\cite[Proposition 6.1]{BF}) implies that 
$\phi^{\lci}$ is an obstruction theory. Detailed  analysis is the same as Theorem \ref{thm_OT_slc}. 

To show that $\phi^{\lci}$ is a perfect obstruction theory, 
it is sufficient to show that the complex 
$$E^{\bullet}_{M^{\lci}}=Rp^{\lci}_{*}\left(\ll^{\bullet}_{\scM^{\lci}/M^{\lci}}\otimes\omega^{\lci}\right)[-1]$$ 
is of perfect amplitude contained in $[-1, 0]$.  The complex $E^{\bullet}_{M^{\lci}}$, when restricted to every geometric point $t$ in $M^{\lci}$, calculates the cohomology 
$\widehat{T}^i_{\QG}(S_t, \sO_{S_{t}})$, where $\SS^{\lci}_t\to S_t$ is the $\lci$  covering Deligne-Mumford stack corresponding to the point $t$. From Lemma 
\ref{lem_universal_higher_tangent}, the cohomology spaces   $\widehat{T}^i_{\QG}(S_t, \sO_{S_{t}})$ only survive when $i=1, 2$, and all the higher obstruction spaces  vanish.  Therefore the obstruction theory 
$\phi^{\lci}$ is perfect.  The last statement is similar. 
\end{proof}

\begin{cor}\label{thm_POT_Mind_1}
Let $M=\overline{M}_{K^2, \chi, N}$ be the moduli stack of stable surfaces of general type with invariants $K^2, \chi, N$. 
If  all the s.l.c. surfaces in $M$ satisfy the Condition  \ref{condition_star},
then the  moduli stack $M^{\lci}$ of $\lci$ covers  is the same as the moduli stack $M^{\ind}$, and the  obstruction theory  for the moduli stack $M^{\ind}$ of index one covers  in Theorem  \ref{thm_OT_slc} 
is perfect in the sense of Behrend-Fantechi.  
\end{cor}
\begin{proof}
If  the Condition  \ref{condition_star} holds, then the index one covering Deligne-Mumford stack $\SS\to S$ has only l.c.i. singularities.
Therefore,  the moduli stack $M^{\lci}=M^{\ind}$, and the obstruction theory 
in Theorem \ref{thm_OT_slc} is the same as the obstruction theory in Theorem \ref{thm_POT_Muniv}.  
\end{proof}

\begin{thm}\label{thm_POT_Mind}
Let $M=\overline{M}_{K^2, \chi, N}$ be the moduli stack of stable surfaces of general type with invariants $K^2, \chi, N$. 
If the moduli stack $M$ consists of k.l.t. surfaces, 
then the moduli stack  $M^{\lci}$  of $\lci$ covers  is the same as the moduli stack    $M^{\ind}$  of index one covers. 

Moreover,  the  obstruction theory  for the moduli stack $M^{\ind}$ of  index one covers  in Theorem  \ref{thm_OT_slc} 
is perfect in the sense of Behrend-Fantechi, and is the same as  the perfect obstruction theory on $M^{\lci}$
in Theorem \ref{thm_POT_Muniv}.  
\end{thm}
\begin{proof}
If  the s.l.c. surfaces $S$ in $M$ is k.l.t.,  then $S$ must only have cyclic quotient singularities.  From the argument in Proposition \ref{prop_slc_higher_cohomology} and  \cite[Proposition 3.10]{Kollar-Shepherd-Barron}, since  the surface $S$ admits a $\qq$-Gorenstein deformation,  the cyclic quotient singularities must have the form 
\begin{equation}\label{eqn_form_action}
\spec \mathbf{k}[x,y]/\mu_{r^2 s},
\end{equation}
where $\mu_{r^2 s}=\langle\alpha\rangle$ and there exists a primitive $r^2 s$-th root of unity 
$\eta$ such that the action is given by
$$\alpha(x,y)=(\eta x, \eta^{dsr-1}y)$$
and $(d, r)=1$.    Thus, the index one cover of $S$ locally has the quotient 
$$\spec \mathbf{k}[x,y]/\mu_{rs}$$ given by $\alpha^\prime(x,y)=(\eta^\prime x, (\eta^\prime)^{rs-1}y)$, which is an $A_{rs-1}$-singularity.  Therefore   the index one covering Deligne-Mumford stack $\SS\to S$ has only l.c.i. singularities.  
From the definition of moduli space of $\lci$ covering Deligne-Mumford stacks in \S \ref{subsubsec_universal_DM}, there is no need to take the $\lci$ covering for such an s.l.c. surface $S$.  
Thus $M^{\lci}=M^{\ind}$, and  the obstruction theory 
in Theorem \ref{thm_OT_slc} is the same as the obstruction theory in Theorem \ref{thm_POT_Muniv}, which is perfect. 
\end{proof}

\begin{cor}\label{cor_POT_MG}
Let $p: \scM\rightarrow M$ be the universal family for the moduli stack $M$ of stable s.l.c. stable surfaces, which is projective and flat. 
Assume that globally the stack $M$ consists of l.c.i. surfaces, then the relative dualizing sheaf $\omega_{\scM/M}$ is relatively Gorenstein, which means $\omega_{\scM/M}$ is a line bundle.
The complex 
$$E^{\bullet}_{M}:=Rp_{*}\left(\ll^{\bullet}_{\scM/M}\otimes\omega_{\scM/M}[2]\right)[-1]$$
defines a perfect obstruction theory
$$\phi: E^{\bullet}_{M}\rightarrow \ll^{\bullet}_{M}.$$
\end{cor}
\begin{proof}
From Corollary \ref{cor_universal_index_one_M}, $M^{\lci}=M$. 
The complex $E^{\bullet}_{M}$ is of perfect amplitude contained in $[-1,0]$.  This is because   $p$ is relatively Gorenstein, which means each fiber surface of  $p$ is Gorenstein and 
$Rp_{*}\left(\ll^{\bullet}_{\scM/M}\otimes\omega_{\scM/M}[2]\right)$ gives the cohomology spaces   $H^i(S, T_S)$ for any fiber of $p$ which vanish
except $i=1, 2$. 
\end{proof}

\begin{rmk}
It is therefore interesting to construct explicit examples of the moduli stack of $\lci$ covers using birational geometry techniques.   

Let $(S,x)$ be a simple elliptic singularity of degree $6$ or $7$, then the del Pezzo cone $f:(\sS, x)\to (\Delta,0)$ of degree $6$ or $7$ (which is the cone associated with the degree $6$ or $7$ del Pezzo surfaces)  is a smoothing of $(S,x)$.  The smoothing $f$ does not admit  an $\lci$ smoothing lifting of the same type singularity, since the link of the threefold singularity $(\sS,x)$ is simply connected, see 
\cite[Theorem 1.3]{Jiang_2023}.
It is interesting to see if  these singularities can be covered by degenerate cusp singularities.  

It is expected to investigate these singularities in  the modular compactification of K3 surfaces \cite{AEH24, AE23, AET23}.  It is possible to construct the  moduli space of Kulikov models in \cite{AE23, AET23}. Contracting the exceptional locus of a Kulikov model may yield a KSBA-stable family of polarized K3 surfaces containing simple elliptic singularities of degree $6$ or $7$. 
\end{rmk}

\subsection{Virtual fundamental class}\label{subsec_virtual_class_MUniv}

Let  $M=\overline{M}_{K^2, \chi, N}$ be a connected component of the moduli stack of s.l.c. surfaces, and $f^{\lci}: M^{\lci}\to M$ is the proper map. From Theorem \ref{thm_POT_Muniv}, the moduli stack $M^{\lci}$ of $\lci$  covers admits a perfect obstruction theory
$$
\phi^{\lci}: E^{\bullet}_{M^{\lci}}\to \ll_{M^{\lci}}^{\bullet}, 
$$
where 
$$E^{\bullet}_{M^{\lci}}:=Rp^{\lci}_{*}\left(\ll^{\bullet}_{\scM^{\lci}/M^{\lci}}\otimes\omega^{\lci}\right)[-1].$$
We follow the method in Section \ref{subsec_virtual_class} to construct the virtual fundamental class on $M^{\lci}$.

 Let $\bfc_{M^{\lci}}$ be the intrinsic normal cone of $M^{\lci}$ such that
\'etale locally on an open subset $U\subset M^{\lci}$  there exists a closed immersion
$$U\hookrightarrow Y$$
into a smooth Deligne-Mumford stack $Y$, we have $\mathbf{c}_{M^{\lci}}|_{U}=[C_{U/Y}/T_Y|_{U}]$. 
Let $N_{M^{\lci}}=h^1/h^0((\ll^{\bullet}_{M^{\lci}})^\vee)$ be the intrinsic normal sheaf of $M^{\lci}$, and there is a natural inclusion 
$\bfc_{M^{\lci}}\hookrightarrow N_{M^{\lci}}$. 

The perfect obstruction theory complex $E^{\bullet}_{M^{\lci}}$ is perfect, and we denote the corresponding bundle stack by 
$h^1/h^0((E^{\bullet}_{M^{\lci}})^\vee)$.  The perfect obstruction theory $\phi^{\lci}: E^{\bullet}_{M^{\lci}}\to \ll_{M^{\lci}}^{\bullet}$ satisfies that 
$h^{-1}(\phi^{\lci})$ is surjective, and $h^{0}(\phi^{\lci})$ is isomorphic.  Therefore it induces an inclusion of stacks 
$N_{M^{\lci}}\hookrightarrow h^1/h^0((E^{\bullet}_{M^{\lci}})^\vee)$.

\begin{defn}\label{defn_VFC_MG}
The virtual fundamental class of the perfect obstruction theory 
$\phi^{\lci}: E^{\bullet}_{M^{\lci}}\to \ll_{M^{\lci}}^{\bullet}$
is defined as
$$[M^{\lci}]^{\vir}=[M^{\lci}, \phi^{\lci}]^{\vir}:=0_{h^1/h^0((E^{\bullet}_{M^{\lci}})^\vee)}^{!}([\bfc_{M^{\lci}}])\in A_{\vd}(M^{\lci}),$$
where $\vd$ is the virtual dimension of $M^{\lci}$, and $0_{h^1/h^0((E^{\bullet}_{M^{\lci}})^\vee)}^{!}$ is the Gysin map in the intersection theory 
of Artin stacks \cite{Kretch}. 

For the morphism $f^{\lci}: M^{\lci}\to M$ which is a proper morphism  as in Theorem \ref{thm_universal_covering_stack},  from \cite[Definition 3.6 (iii)]{Vistoli} we define 
$$[M]^{\vir}:=f^{\lci}_{*}([M^{\lci}, \phi^{\lci}]^{\vir})\in A_{\vd}(M^{})$$
which is called the virtual fundamental class for the moduli stack $M$.
\end{defn}

\begin{rmk}
From   Corollary \ref{thm_POT_Mind}, if the moduli stack $\overline{M}_{K^2,\chi}$ consists of k.l.t. surfaces, then the 
morphism $f^{\lci}: M^{\lci}\to M$ is an isomorphism and the perfect obstruction theory induces a virtual fundamental class $[M]^{\vir}\in A_{\vd}(M)$.
\end{rmk}

Finally we have

\begin{cor}\label{prop_POT_MG_virtual}
Suppose  the moduli stack $M$ of s.l.c. stable surfaces only  consists of l.c.i. surfaces, then the  perfect obstruction theory
in Corollary \ref{cor_POT_MG}
$$\phi: E^{\bullet}_{M}\rightarrow \ll^{\bullet}_{M}$$
induces a virtual fundamental class 
$$[M]^{\vir}\in A_{\vd}(M).$$
\end{cor}
\begin{proof}
This is from Corollary \ref{cor_POT_MG} and the construction of virtual fundamental class in this section. 
\end{proof}

\begin{rmk}\label{rmk_vd}
The virtual dimension of $M^{\lci}$ is the same as the virtual dimension of the moduli stack 
$M$, which is 
$$\vd=\dim H^1(S, T_S)^G-\dim H^2(S, T_S)^G$$
for $S$ is a general smooth s.l.c. surface in $M$. 

In the case that $G$ is trivial, the virtual dimension of $M^{\lci}$  can be calculated by Grothendieck-Riemann-Roch theorem
\begin{align*}
\vd&=\rk(E_{M}^{\bullet})=-\chi(S, T_S)
=\int_{S}\Ch(T_S)\cdot \Td(T_S)\\
&=-(\frac{7}{6}c_1^2-\frac{5}{6}c_2)
=10\chi-2K^2.
\end{align*}
Thus,  if $10\chi-2K^2\geq 0$, the virtual dimension is nonnegative and one can define invariants by taking integration over the virtual fundamental class 
 $[\overline{M}_{K^2, \chi}]^{\vir}$ by some tautological classes. 
\end{rmk}

\begin{rmk}
Our main results Theorem \ref{thm_POT_Muniv} and Definition \ref{defn_VFC_MG} show  that for the moduli stack $M=\overline{M}_{K^2, \chi, N}$ obtained from $\qq$-Gorenstein deformations, the moduli stack   $M^{\lci}=\overline{M}^{\lci}_{K^2, \chi, N}$ of $\lci$  covers
admits a virtual fundamental class.   
This provides a strong evidence on Donaldson's conjecture for the existence of virtual fundamental class for a large class of  moduli stacks of surfaces of general type.  
In practice  people hope that there are many examples where the boundary divisors of the moduli stack $M$  consist of only l.c.i. surfaces; see for examples 
$\overline{M}^{\Gor}_{1, 3}$ and  $\overline{M}^{\Gor}_{1, 2}$ for the moduli stacks of Gorenstein surfaces in \cite{FPR}, and Donaldson's example in \S \ref{subsec_Donaldson_example}. 
Note that the moduli stack $\overline{M}^{\Gor}_{1, 3}$ and  $\overline{M}^{\Gor}_{1, 2}$ are open substacks in the moduli stack $\overline{M}_{1, 3}$ and  $\overline{M}_{1, 2}$. 
Actually for the moduli stack $M$ obtained from  $\qq$-Gorenstein deformations, the boundary divisors may only contain the $\qq$-Gorenstein deformation of class 
$T$-singularities.  Almost for all of the  known examples for  $M$ in the literature the boundary divisors were constructed using $\qq$-Gorenstein deformation of class $T$-singularities;  i.e., using  the deformation of the corresponding index one covering Deligne-Mumford stacks.  In this case,  the moduli stack   $M^{\ind}$ of index one covers  admits a virtual fundamental class. 
An interesting example is given by  the moduli stack $\overline{M}_{1, 3}$ of s.l.c.  surfaces  with $K^2=1, \chi=3$ in \cite{FPRR}, where some boundary divisors and other irreducible components in $\overline{M}_{1, 3}$ were explicitly constructed
by deformation of class $T$-singularities.  We hope that the explicit components constructed  in \cite{FPRR} completely determine the stack  $\overline{M}_{1, 3}$. 
\end{rmk}

\section{CM line bundle and tautological invariants}\label{sec_CM_tautological_invariants}

Let $M=\overline{M}_{K^2, \chi, N}$ be one connected component of KSBA moduli stack of s.l.c. stable surfaces. 
In this section we require that $N$ is large divisible enough so that $M=\overline{M}_{K^2, \chi}$ is the proper moduli stack of s.l.c. stable surfaces
with invariants $K^2, \chi$. 

\subsection{CM line bundle on the moduli stack}\label{subsec_CM_line_bundle}

From \cite[\S 2.1]{Donaldson}, over smooth part $M_{K^2, \chi}$ consisting of smooth general type surfaces $S$ with 
$K_S^2=K^2, \chi(\sO_S)=\chi$, differential geometry can define Miller-Mumford-Morita (MMM)-classes on 
$H^*(M_{K^2, \chi}, \qq)$. Donaldson \cite[\S 4]{Donaldson} proposed  a question to extend the MMM-classes to
$H^*(\overline{M}_{K^2, \chi}, \qq)$ of  the whole 
KSBA compactification $M$.  

In algebraic geometry there exists a CM line bundle on the moduli stack 
$M$ as defined in \cite{Tian97}, \cite{FR06} and \cite{PXu15}. 
We recall it here.
Let $p: \scM\rightarrow M$ be the universal family which is a projective, flat morphism with relative dimension $2$. 
Then the relative canonical sheaf $K_{\scM/M}$ is $\qq$-Cartier and relatively ample, see \cite{HMX14} and  \cite{KP15}. 
For any relatively ample line bundle 
$\sL$ on $\scM$, we have
$$\det\left(p_{!}(\sL^k)\right)=\det\left(R^*p_*(\sL^k)\right)
=\bigotimes_{i}\left(\det\left(R^ip_*(\sL^k)\right)\right)^{(-1)^{i}}.$$
As $\sL$ is relatively ample, $R^{i}p_*(\sL^k)=0$ for $i>0, k>>0$, thus 
$\det p_{!}(\sL^k)=\det p_*(\sL^k)$. 
From \cite{KM}, there exist line bundles 
$\lambda_i$ for $i=0, 1, 2, 3$ on $\overline{M}_{K^2, \chi}$, such that for all $k$, 
$$\det p_{!}(\sL^k)\cong \lambda_3^{\mat{c} k \\
3\rix}\otimes \lambda_2^{\mat{c} k \\
2\rix}\otimes \lambda_1^{\mat{c} k \\
1\rix}\otimes \lambda_0.$$
Let $\mu:=-\left(K_{S_t}\cdot \sL|_{S_t}\right)/\sL^2|_{S_t}$, then the CM line bundle (corresponding to $\sL$) is
$$\lambda_{\CM}=\lambda_{\CM}(\scM/M, \sL):=\lambda_3^{2\mu+6}\otimes \lambda_2^{-6}.$$
Using Grothendieck-Riemann-Roch theorem in \cite{FR06}, we have that 
$$
\begin{cases}
c_1(\lambda_3)=p_*(c_1(\sL)^3);\\
2c_1(\lambda_3)-2c_1(\lambda_2)=p_*(c_1(\sL)^2\cdot c_1(K_{\scM/M})).
\end{cases}
$$

Let $\sL=K_{\scM/M}$, then the CM line bundle is 
$$\lambda_{\CM}(\scM/M, K_{\scM/M}):=\lambda_3^{4}\otimes \lambda_2^{-6}=\lambda_2^2, $$
since Serre duality implies that $\lambda_3=\lambda_2^2$. 
We have that 
$$c_1(\lambda_{\CM}(\scM/M, K_{\scM/M}))=p_*\left((K_{\scM/M})^3\right).$$
We define 
$$L_{\CM}:=\lambda_{\CM}(\scM/M, K_{\scM/M}).$$
From \cite[Theorem 1.1]{PXu15}, the CM line bundle $L_{\CM}$ is ample on the KSBA moduli stack 
$M$.

\subsection{Tautological invariants}\label{subsec_tautological_invariants}
Let $M$ be one connected component of the  moduli stack of $G$-equivariant stable general type surfaces with invariants 
$K_S^2=K^2, \chi(\sO_S)=\chi$.   From Theorem \ref{thm_POT_Muniv}
the moduli  stack $M^{\lci}=\overline{M}^{\lci}_{K^2, \chi, N}$ of $\lci$ covers  admits a perfect obstruction theory
$
\phi^{\lci}: E^{\bullet}_{M^{\lci}}\to \ll_{M^{\lci}}^{\bullet}$, hence induces  a virtual fundamental class
$[M]^{\vir}$ in Definition \ref{defn_VFC_MG}.

\begin{defn}\label{defn_tautological_invariants}
Let $M$ be one connected component of the moduli stack of stable surfaces with fixed invariants 
$K^2, \chi, N$.  
We define the tautological invariant by
$$I_{\CM}=\int_{[M]^{\vir}}(c_1(L_{\CM}))^{\vd}.$$
\end{defn}

\begin{rmk}
It is  interesting to consider other tautological classes on the moduli stack $\overline{M}_{K^2, \chi}$.
\end{rmk}

\section{Examples}\label{sec_examples}

In this section we study several examples.  

\subsection{Moduli space of quintic surfaces}\label{subsec_quintic_example}

\subsubsection{General degree $d$ hypersurfaces in $\pp^3$}
Let us first consider some basic invariants for smooth hypersurfaces in $\pp^3$ of degree $d\geq 5$.
Let $\iota: S\subset \pp^3$ be a smooth hypersurface of degree $d$, then we have the exact sequence
\begin{equation}\label{eqn_exact_1}
0\to T_{S}\rightarrow T_{\pp^3}\rightarrow N_{S/\pp^3}\to 0,
\end{equation}
where $N_{S/\pp^3}=\sO_S(d)$ is the normal bundle. 
When $d\geq 5$, the surfaces $S$ is of general type. Therefore,  $H^i(S, T_S)=0$ only except $i=1, 2$. 
We calculate the dimensions of the cohomology spaces of the tangent bundle of $S$ for $d=5, 6$, 
\begin{equation}\label{eqn_cohomology_d5}
\begin{cases}
\dim H^1(S, T_S)= 40;\\
\dim H^2(S, T_S)= 0,
\end{cases}
\end{equation}
and 
\begin{equation}\label{eqn_cohomology_d6}
\begin{cases}
\dim H^1(S, T_S)=68;\\
\dim H^2(S, T_S)=6.
\end{cases}
\end{equation}

The  cohomology spaces  $H^*(S, T_S)$ are calculated by taking  the long exact sequence of the cohomology of  (\ref{eqn_exact_1})
\begin{align*}
0 &\to H^0(S, T_S)\rightarrow H^0(S, T_{\pp^3}|_{S})\rightarrow H^0(S, N_{S/\pp^3})\\
& \rightarrow H^1(S, T_S)\rightarrow H^1(S, T_{\pp^3}|_{S})\rightarrow H^1(S, N_{S/\pp^3})\\
& \rightarrow H^2(S, T_S)\rightarrow H^2(S, T_{\pp^3}|_{S})\rightarrow H^2(S, N_{S/\pp^3})\\
&\to 0, 
\end{align*}
and the long exact sequences  on the cohomology of the following  two exact sequences
$$0\to \sO_{\pp^3}\rightarrow \sO_{\pp^3}(d)\rightarrow \iota_*N_{S/\pp^3}\to 0$$
and 
$$0\to T_{\pp^3}(-d)\rightarrow T_{\pp^3}\rightarrow \iota_*(T_{\pp^3}|_{S})\to 0.$$
We omit the detailed calculation. 

\subsubsection{Moduli space of quintic surfaces}

Let us first briefly recall the moduli space of quintic surfaces in \cite{Horikawa}. 
Let $S\subset \pp^3$ be a smooth quintic surface defined by a homogeneous degree five polynomial.  It is well-known that the topological invariants of $S$ are given by 
$K_S=\sO_S(1)$, and 
$$K_S^2=5; \quad  q=\dim H^1(S, \sO_S)=0, \quad  p_g=4; \quad  \chi(\sO_S)=5.$$
Let $\overline{M}_{5,5}$ be the moduli stack of general type minimal  surfaces $S$ with 
$K_S^2=5, \chi(\sO_S)=5$.  From \cite{Horikawa},  the coarse moduli space of   the Gieseker's moduli stack 
$M_{5,5}\subset \overline{M}_{5,5}$  is a $40$-dimensional scheme
with two irreducible components $M_0\cup_{W}M_1$ meeting transversally at a $39$-dimensional scheme 
$W$, where $M_0$ is the component containing quintic surfaces in $\pp^3$ with rational double points singularities (RDP's), and  
the other components $M_1$ and $W$ consist of the following surfaces:
first from \cite[Theorem 1]{Horikawa},  for any minimal surface with $K_S^2=5; \quad  q=\dim H^1(S, \sO_S)=0, \quad  p_g=4$ and $\chi(\sO_S)=5$, the canonical system $|K_S|$ has at most one base point. 
There are three types of surfaces:

Type I:   $|K_S|$ has no base point.  The surface $S$ is birationally equivalent to $S^\prime$, where $S^\prime\subset \pp^3$ is a quintic surface with only RDP's singularities;

Type IIa:   $|K_S|$ has one base point. Let $\pi: \widetilde{S}\to S$ be the quadric transformation with center at the base point $b\in |K_S|$,  then there exists a surjective morphism $f: \widetilde{S}\to \pp^1\times\pp^1$ of degree $2$;

Type IIb:   $|K_S|$ has one base point.  In this case there exists a surjective map $f: \widetilde{S}\to \Sigma_2$ of degree two, where $\Sigma_2$ is the Hirzebruch surface of degree two, and there also exists a diagram:
\[
\xymatrix{
&
(\tilde{S})\ar[dl]_{f}\ar[dr]^{\psi}&\\
\Sigma_2\ar[rr]^{\varphi}&& \pp^3
}
\]
such that the image of $\varphi$ and $\psi$ are the quadric cone in $\pp^3$. Note that all the Type I, IIa and IIb surfaces are l.c.i. surfaces. 
The deformation of Type I, Type IIa and Type IIb surfaces are given by the deformation of the corresponding birational models in the description. 

The component $M_0$ consists of Type I surfaces;  the component $M_1$ consists of Type IIa surfaces and the intersection $W$ parametrizes type IIb surfaces.
For a surfaces $S$, from \cite{HMXu13}, $|\Aut(S)|\leq 42^2\cdot \Vol(S, K_S)$ and  if $S$ is minimal then $\Vol(S, K_S)=K_S^2$ and 
$|\Aut(S)|\leq 42^2\cdot 5$.  If we consider all the automorphism groups of $S$, we get the Deligne-Mumford stack 
$\overline{M}_{5,5}$. 

 The complete  boundary divisors of 
$\overline{M}_{5,5}$ are still not explicitly constructed;  see \cite{Rana} for an explicit construction of one boundary divisor 
$D_{\frac{1}{4}(1,1)}\subset \overline{M}_{5,5}$ corresponding to a Wahl singularity of type $\frac{1}{4}(1,1)$.
But the abstract KSBA compactification  $\overline{M}_{5,5}$ was constructed and is  a proper Deligne-Mumford stack; see \cite{Kollar-Book}.

Let us give  an example for $\overline{M}_{5,5}$ on the boundary loci consisting of s.l.c. surfaces. 
In \cite{Rana}, Rana gave a construction of one boundary divisor $D_{\frac{1}{4}(1,1)}\subset \overline{M}_{5,5}$, which consists of s.l.c. surfaces $S$ with only one Wahl type $\frac{1}{4}(1,1)$ singularity.  
This singularity has index $r=2$, and $S$ has global index $N=2$. 
The boundary divisor  $D_{\frac{1}{4}(1,1)}=D_1\cup_{W_{2b}}D_{2a}$ also contains two irreducible components, where 
$D_{1}$ is the component consisting of Type 1 surfaces; $D_{2a}$ is the component consisting of Type 2a surfaces, and $W_{2b}$ is the component consisting of Type 2b surfaces. Type 1, 2a and 2b surfaces are classified as follows.: the minimal resolution of a Type 1 surface is a double cover of $\pp^1\times\pp^1$, branched over a sextic intersecting a given diagonal tangentially at $6$ points.  The preimage of the diagonal is given by  two $(-4)$-curves intersecting at $6$ points. Contracting one of these $(-4)$-curves gives a stable numerical quintic surface of Type $1$. 
The minimal resolutions of Type 2a (respectively Type 2b) surfaces are themselves minimal resolutions of double covers of $\pp^1\times\pp^1$ (respectively a quadric cone), the branch curve of which is a sextic $B$ intersecting a given ruling at two nodes of $B$ and transversally at two other points. 
There are relations:  Type 1 (respectively Type 2a, Type 2b) surfaces are the  deformation limits of Type I  (respectively Type II a, Type II b) surfaces.
Thus $D_1, D_{2a}$ are $39$-dimensional, and $W_{2b}$ is $38$-dimensional. 

The obstruction space at such a boundary divisor was calculated. Let us take a point 
$S\in D_{2b}$ such that there exist open \'etale neighborhoods $U_{2a}\subset \overline{M}_{5,5}$ and $U_{2b}\subset \overline{M}_{5,5}$  satisfying the condition that 
$U_{2b}$ contains  elements in $D_1$ and  the boundary $W_{2b}$,   and $U_{2a}$ is only a neighborhood of $D_1$ which does not intersect with $W_{2b}$. 
From \cite[Theorem 5.1]{Rana},  $U_{2b}\subset \aaa_{\mathbf{k}}^{41}$ is   cut out by 
$H^\prime=q^\prime(t)\cdot r^\prime(t)=0$ for two holomorphic functions. 
For the surface germ $(S,P)$ of Type 2b, the obstruction space is calculated by the corresponding canonical index one covering Deligne-Mumford stack $\SS$ which contains only one orbifold point of type $\frac{1}{4}(1,1)$. 
The obstruction space $T^2_{\QG}(\ll^{\bullet}_{\SS}, \sO_{\SS})=H^2(\SS, T_{\SS})$ has dimension $1$, which  was calculated in \cite{Rana}.
 
Recall that in \cite{Kollar-Shepherd-Barron} a quotient $T$-singularity is given by  a quotient $2$-dimensional singularity of type 
$\frac{1}{dn^2}(1, dna -1)$, 
where $n > 1$ and $d, a > 0$ are integers with $a$ and $n$  coprime. These are  the quotient singularities that admit a $\qq$-Gorenstein smoothing. 
When $d=1$, these are called Wahl singularities. The s.l.c. minimal surfaces $S$ with $T$-singularities  satisfying $K_S^2=5, p_g=4$ may give some other irreducible components of $\overline{M}_{5,5}$.

\subsubsection{Discussion of the virtual fundamental class}
For a large divisible $N>0$, the KSBA compactification $\overline{M}_{5, 5, N}$ 
may contain a lot of irreducible components. 
Let us only consider the following  two irreducible components
$$\overline{P}:=\overline{M}_0\cup_{\overline{W}}\overline{M}_1$$
where  $\overline{M}_0=\overline{M}^{\quintic}$ is the closure of the component in  $\overline{M}_{5, 5}$ containing the smooth quintics,  and 
$\overline{M}_1$  is the closure of the component in $\overline{M}_{5, 5}$ containing the smooth  Type IIa surfaces in \cite{Horikawa}. 
The two Deligne-Mumford stacks $\overline{M}_0$ and $\overline{M}_1$ are 
$40$-dimensional Deligne-Mumford stacks meeting at a $39$-dimensional closed substack $\overline{W}$.
The above calculation on the cohomology spaces  $H^*(S, T_S)$ implies that the main component $\overline{M}_0$ is smooth on the open part $M_0$ consisting quintic surfaces.  From \cite{Horikawa}, the open subset $M_1\subset \overline{M}_1$ is also smooth. The singular locus of $\overline{P}$ only happens on $\overline{W}$. 
We assume that all the boundary loci of $\overline{P}$ contain  stable surfaces with class $T$-singularities, so that their index one covers  only have normal crossing and $A_n$-type singularities. 
Thus, from Corollary \ref{cor_universal_index_one_M},  $\overline{P}^{\lci}=\overline{P}^{\ind}$. 

We construct the virtual fundamental class for $\overline{P}$.  Let $f: \overline{P}^{\ind}\to \overline{P}$ be the  moduli stack of index one covers. 
Then 
$$\overline{P}^{\ind}=\overline{M}^{\ind}_0\cup_{\overline{W}^{\ind}}\overline{M}^{\ind}_1,$$
where $f^0: \overline{M}^{\ind}_0\to \overline{M}_0$ and $f^1: \overline{M}^{\ind}_1\to \overline{M}_1$ are the moduli stacks of index one covers over  the components 
$\overline{M}_0$ and $\overline{M}_1$ respectively,  and they intersects at $f_W: \overline{W}^{\ind}\to \overline{W}$ which is the moduli stack of index one covers over
$\overline{W}$. 
The morphisms $f^0, f^1$ and $f_W$ are isomorphisms except on the boundary divisors of $\overline{P}$ given by $\qq$-Gorenstein smoothing of class 
$T$-singularities.  For example, over the divisor $D_{\frac{1}{4}(1,1)}=D_1\cup_{W_{2b}}D_{2a}$ in $\overline{P}$, the fibers of $f^0, f^1$ and $f_W$ are the index one covering Deligne-Mumford stacks of the stable surfaces 
with one Wahl singularity $\frac{1}{4}(1,1)$ of Type 1, Type 2a and Type 2b respectively.

Let $p^{\ind}:  \scM^{\ind}\to \overline{P}^{\ind}$ be the universal family.  Then there is a perfect obstruction theory 
$$\phi^{\ind}: E^{\bullet}_{\overline{P}^{\ind}}\to \ll^{\bullet}_{\overline{P}^{\ind}},$$
where 
$$E^{\bullet}_{\overline{P}^{\ind}}:=Rp^{\ind}_{*}\left(\ll^{\bullet}_{\scM^{\ind}/\overline{P}^{\ind}}\otimes\omega^{\ind}\right)[-1].$$

Let $\bfc_{\overline{P}^{\ind}}$ be the intrinsic normal cone of $\overline{P}^{\ind}$. This intrinsic normal cone can be written as
$$\bfc_{\overline{P}^{\ind}}=\bfc_{\overline{M}_{0}^{\ind}}+\bfc_{\overline{M}^{\ind}_1},$$
where $\bfc_{\overline{M}_0^{\ind}}$ and $\bfc_{\overline{M}^{\ind}_1}$ are the intrinsic normal cones of the components $\overline{M}_0^{\ind}$ and $\overline{M}^{\ind}_1$ respectively. 
This can be calculated by embedding $\overline{P}^{\ind}$ into a higher dimensional smooth Deligne-Mumford stack $\sY$ and the normal cone of $C_{\overline{P}^{\ind}/\sY}$ contains two irreducible components given by the two irreducible components $\overline{M}_0^{\ind}$ and $\overline{M}^{\ind}_1$.

Look at the following diagram
\[
\xymatrix{
\bfc_{\overline{P}^{\ind}}\ar@{^{(}->}[r]\ar[d]& h^1/h^0((E^{\bullet}_{\overline{P}^{\ind}})^\vee)\ar[d]\\
\overline{cv}\ar@{^{(}->}[r]& h^1((E^{\bullet}_{\overline{P}^{\ind}})^\vee)=\Ob_{\overline{P}^{\ind}},
}
\]
where $\overline{cv}$ is the coarse moduli space of the intrinsic normal cone $\bfc_{\overline{P}^{\ind}}$, and 
$\Ob_{\overline{P}^{\ind}}$ is the obstruction sheaf of the perfect obstruction theory $\phi^{\ind}$.

Let $s: \overline{P}^{\ind}\to \Ob_{\overline{P}^{\ind}}$ be the zero section.  
From Definition \ref{defn_VFC_MG}
the virtual fundamental class $[\overline{P}^{\ind}]^{\vir}\in A_{40}(\overline{P}^{\ind})$
is obtained from the intersection of the intrinsic normal cone with the zero section of 
the bundle stack $h^1/h^0((E^{\bullet}_{\overline{P}^{\ind}})^\vee)$.
From the decomposition of the intrinsic normal cone 
$\bfc_{\overline{P}^{\ind}}=\bfc_{\overline{M}_0^{\ind}}+\bfc_{\overline{M}^{\ind}_1}$, the intersection can be calculated separately. 
Also note that both $\overline{M}_0^{\ind}$ and $\overline{M}^{\ind}_1$ are smooth,  and the coarse moduli spaces of the intrinsic normal cones 
$\bfc_{\overline{M}_0^{\ind}}$ and $\bfc_{\overline{M}^{\ind}_1}$ are just $\overline{M}_0^{\ind}$ and $\overline{M}^{\ind}_1$ respectively.  Therefore,  the intersections are just the intersections of 
$\overline{M}_0^{\ind}$ and $\overline{M}^{\ind}_1$ with the zero sections of the obstruction sheaf.   We obtain
$$[\overline{P}^{\ind}]^{\vir}=[\overline{M}_0^{\ind}]+[\overline{M}^{\ind}_1] \in A_{40}(\overline{P}^{\ind}).$$

There is a canonical morphism
$$f: \overline{P}^{\ind}\rightarrow \overline{P}$$
which is a finite morphism  and is an isomorphism except on the boundaries.  Thus we have that 
$$[\overline{P}]^{\vir}=f_*([\overline{P}^{\ind}]^{\vir})\in A_{40}(\overline{P}).$$

\begin{rmk}
It is interesting to calculate the tautological invariants for the moduli stack of quintic surfaces. 
\end{rmk}

\subsection{Donaldson's example on sextic hypersurfaces}\label{subsec_Donaldson_example}

In this section we talk about Donaldson's example on sextic hypersurfaces in $\pp^3$, and give an affirmative answer for the existence of virtual fundamental class on the moduli of $\mathfrak{G}$-equivariant sextic hypersurfaces in $\pp^3$ for some finite group $\mathfrak{G}$, thus proving Donaldson's conjecture on the existence of virtual fundamental class of this example. In this section all the surfaces are l.c.i. and the index $N=1$. The moduli stack $M_N$ is a $2$-dimensional toric Deligne-Mumford stack. 

\subsubsection{The GIT moduli space}
Let $S\subset \pp^3$ be a smooth degree $6$ hypersurface, then the formula of the  cohomology of the tangent bundle of $S$ is given in (\ref{eqn_cohomology_d6}).
Other topological invariants are given by:
$$e(S)=108;  \quad p_g=10;  \quad K_S^2=24; \quad  \chi(\sO_S)=11. $$
Let $\overline{M}_{24, 11}$ be the moduli stack of stable surfaces $S$ with invariants $K_S^2=24, \chi(\sO_S)=11$.
It is not known in the literature what this moduli stack looks like, but at least there exists one component of  $\overline{M}_{24, 11}$ containing  sextic 
surfaces in $\pp^3$.

In order to get an explicit moduli stack, Donaldson \cite[\S 5]{Donaldson} put more symmetries on the sextic surfaces. 
Let $\pp^3=\proj (\mathbf{k}[x_1, y_1, x_2, y_2])$. Let 
$\zeta\in \mu_6$ be a primitive generator of the cyclic group of order $6$, and let $\mathfrak{G}$ be the subgroup of $GL(4,\mathbf{k})$ generated by
$$
\begin{cases}
(x_1, y_1, x_2, y_2)\mapsto (\zeta x_1, \zeta^{-1} y_1, x_2, y_2);\\
(x_1, y_1, x_2, y_2)\mapsto (x_1, y_1, \zeta x_2, \zeta^{-1} y_2);\\
(x_1, y_1, x_2, y_2)\mapsto (x_2, y_2, x_1, y_1),
\end{cases}
$$
which are the actions on $\aaa_{\mathbf{k}}^4$.

Then $\mathfrak{G}$ acts on the sextic hypersurfaces  in $\pp^3$. The invariant degree $6$ homogeneous polynomials are given by
$$\alpha x_1^6+ \beta y_1^6+\alpha x_2^6+ \beta y_2^6+A Q_+^{3}+ BQ_+Q_-^{2},$$
where $Q_\pm=x_1 y_1\pm x_2 y_2$. 
Then  the $\Gm=\aaa_{\mathbf{k}}^*$-action by 
$$(x_1, y_1, x_2, y_2)\mapsto (\lambda x_1, \lambda^{-1}y_1, \lambda x_2, \lambda^{-1}y_2)$$
induces homogeneous polynomials invariant under the action of $\mathfrak{G}$.
All the invariant degree $6$ polynomials under the $\mathfrak{G}$-action give the parameter space 
$$(\alpha, \beta, A, B).$$
The $\Gm$ acts on the parameter space by
$$(\alpha, \beta, A, B)\mapsto (\lambda^6 \alpha, \lambda^{-6}\beta, A, B).$$
Let $V$ represent the vector space parametrized by $(\alpha,\beta,A,B)$. 
Then the stable points in $V$ for the above torus action are those where $\alpha, \beta$ are non-zero, and each stable orbit in $V$ contains a representative 
$$\alpha(x_1^6 +y_1^6 +x_2^6 +y_2^6)+A Q_+^3+BQ_+Q_-^2,$$
which is unique up to change of the sign of $\alpha$.  The moduli space of GIT stable locus of sextic hypersurfaces 
with $\mathfrak{G}$-action is $\aaa_{\mathbf{k}}^2/\{\pm\}$.  Here 
$\aaa_{\mathbf{k}}^2=\spec\mathbf{k}[A, B]$ and each $(A,B)$ corresponds to a hypersurface
\begin{equation}\label{eqn_S-AB}
S_{AB}=\{x_1^6 +y_1^6 +x_2^6 +y_2^6+A Q_+^3+BQ_+Q_-^2=0\}\subset \pp^3.
\end{equation}
We recall the KSBA compactification of  $\aaa_{\mathbf{k}}^2/\{\pm\}$ in \cite[\S 5]{Donaldson}. 
Before KSBA, there is a naive compactification by embedding 
$\aaa_{\mathbf{k}}^2\hookrightarrow \pp^2$ and then taking the quotient by $\mu_2=\{\pm\}$-action. 
Modulo the automorphism group of the sextic surfaces, the moduli stack is the quotient 
$M^{GIT}=[\pp^2/\mu_2]$.  The polytope description is given in \cite[\S 5]{Donaldson}.  The stacky fan in the sense of \cite{BCS}, \cite{Jiang2008} is given by
$\mathbf{\Sigma}=(N, \Sigma, \beta)$, where 
$N=\zz^2$, $\Sigma$ is the fan in  $\rr^2$ generated by rays 
$\rr_{(2,0)}, \rr_{(0,1)}$ and $\rr_{(-2,-1)}$, and 
$\beta: \zz^3\to N$ is given by $(2,0), (0,1), (-2, -1)$.
The quotient action of $\mu_2$ on the homogeneous coordinates $[x:y:z]$ of $\pp^2$ by 
$$[x:y:z]\mapsto [x:-y:-z].$$
The fixed point locus are the point $[1:0:0]$ and $\pp^1=\proj (\mathbf{k}[0:y:z])$ which is the divisor at infinity. 
The divisor $\pp^1$ in the moduli toric Deligne-Mumford stack $[\pp^2/\mu_2]$ corresponds to the following surfaces
$$\{AQ_+^3+BQ_+Q_-^2=0\}$$
for $A, B\neq 0$ at the same time. Note that there are three cases
\begin{enumerate}
\item $A, B\neq 0$, then $\{AQ_+^3+BQ_+Q_-^2=0\}$ corresponds to three quadrics meeting in four lines;
\item $B=0, A\neq 0$, this corresponds to the quadric  $\{Q_+=0\}$ with multiplicity $3$;
\item $B\neq 0, A= 0$, this corresponds to the quadric  $\{Q_+=0\}$  and the quadric $\{Q_-=0\}$ with multiplicity $2$.
\end{enumerate}

\subsubsection{The KSBA compactification}\label{subsubsec_KSBA}
Let us consider the KSBA compactification of the moduli stack $[\aaa_{\mathbf{k}}^2/\mu_2]$ of sextic hypersurfaces with $\mathfrak{G}$-action. 
We follow Donaldson's argument but using the fan structure of the toric Deligne-Mumford stack 
$M^{GIT}=[\pp^2/\mu_2]$. 

$\bullet$ Let $O:=((0,1), (-2,-1))$ be the top cone generated by $\{(0,1), (-2,-1)\}$, which corresponds to the affine toric Deligne-Mumford 
stack $[\aaa_{\mathbf{k}}^2/\mu_2]$, and the sextic surfaces in (\ref{eqn_S-AB}). One can think of the ray $\rr_{(2,0)}$ standing for the infinity divisor 
$\pp^1\subset M^{GIT}$ which is fixed under $\mu_2$. 
The ray $\rr_{(-2,-1)}$ (which corresponds to $OIII$ in Donaldson's picture in \cite[Page 20]{Donaldson}) corresponds to the surfaces 
$\{S_{A0}\}$ in  (\ref{eqn_S-AB}).

$\bullet$ The toric Deligne-Mumford stack $[\pp^2/\mu_2]=[\aaa_{\mathbf{k}}^2/\mu_2]\sqcup \pp^1$, where 
$\pp^1=\chi(\Sigma/\rr_{(2,0)})$; i.e., the toric Deligne-Mumford stack of the quotient fan $\Sigma/\rr_{(2,0)}$ modulo the ray $\rr_{(2,0)}$.
So it is enough to know what sextic surfaces the ray $\rr_{(2,0)}$ corresponds for.
As pointed out in \cite[\S 5]{Donaldson}, this ray $\rr_{(2,0)}$ corresponds to surfaces 
 $\{AQ_+^3+BQ_+Q_-^2=0\}$ for $A, B$ not zero at the same time.
Let $III:=((-2,-1), (2,0))$ be the top cone generated by $\{(-2,-1), (2,0)\}$ and 
$II:=((2,0), (0,1))$ be the top cone generated by $\{(2,0), (0,1)\}$. Note that the surface 
$\{Q_+=0\}$ with multiplicity $3$ corresponds to the origin in the cone $III$, and the surface 
$\{Q_+Q_-^2=0\}$, one quadric $\{Q_+=0\}$ and one $\{Q_-=0\}$ with multiplicity $2$ corresponds to the origin in the cone $II$.

The surfaces corresponding to the infinity divisor $\pp^1$ are not s.l.c. surfaces, and we perform weighted blow-ups on $M^{GIT}=[\pp^2/\mu_2]$ to get the KSBA compactification. 
From \cite[\S 5]{Donaldson},  in the cone $III$, the vertex corresponds to the surface  $\{S_{A0}\}$ when $A\rightarrow \infty$.
The construction is given as follows:
let $\pi: Y\to \pp^3$ be the triple cover over $\pp^3$ branched over $S_{00}$. 
There exists a section $\eta\in \pi^*\sO(2)\to Y$ such that 
$\eta^3=s$, and $s$ is the section cutting out $S_{00}$. 
Then let $W\subset Y\times\pp^1$ be the surface cut out by 
$\eta=\lambda Q_+$. Let $A=\lambda^3$.  When $A\rightarrow \infty$, we get a triple cover $S_{III}$ over $\{Q_+=0\}=\pp^1\times\pp^1$ branched over 
$\{S_{00}\cap \{Q_+=0\}$. 
This triple cover $S_{III}\to \pp^1\times\pp^1$ has an extra automorphism group $\mu_3$. 
Therefore, we do the weighted blow-up on the toric Deligne-Mumford stack $[\pp^2/\mu_2]$ by inserting the ray $\rr_{(4,-1)}$ generated by 
$(4,-1)=3(2,0)+(-2,-1)$.  This ray splits the cone $III$ into two top cones denoted by 
$III=((-2,-1), (4,-1))$ and $IV^\prime=((4,-1), (2,0))$. From \cite{BH},  this gives a new stacky fan 
$\mathbf{\Sigma}^\prime$ and a toric Deligne-Mumford stack 
$$h: \chi(\mathbf{\Sigma}^\prime)\rightarrow M^{GIT},$$
which is a weighted blow-up. 
The exceptional locus (divisor) of $h$ corresponds to the following family of surfaces:
taking affine coordinates $(s,t)$ of $\pp^1\times\pp^1$, and let 
$C_{\mu}\in |\sO(6,6)|$ be the curve  with affine equation:
$$1+s^6+t^6+s^6t^6+\mu\cdot s^3t^3=0.$$
Then the family of surfaces (corresponding to the exceptional locus $\pp^1$ by $\mu$, but $\mu\neq \infty$) is
$$\mu: \sS_{\mu}\to \pp^1\times\pp^1,$$
which are triple covers over $\pp^1\times\pp^1$ with simple branching over $C_{\mu}, C_{-\mu}$. 
The $\mu=0$ case corresponds to the surface $S_{III}$ above.  All of these surfaces are s.l.c. surfaces. 

Now we perform on the top cone $II$ similarly.  Since the vertex of the  cone $II$ corresponds to $\{Q_+Q_-^2=0\}$, there exists a $\zz_2$-symmetry. 
We do the weighted blow-up on the toric Deligne-Mumford stack  $\chi(\mathbf{\Sigma}^\prime)$ by inserting (inside the cone $II$) a ray
$\rr_{(2,1)}$ generated by $(2,1)=(2,0)+(0,1)$. This ray splits the top cone $II$ into two top cones 
$II=((2,1), (0,1))$ and $IV^{\prime\prime}=((2,0), (2,1))$. Thus we get a new stacky fan 
$\mathbf{\Sigma}^{\prime\prime}$ such that the morphisms 
$$\chi(\mathbf{\Sigma}^{\prime\prime})\stackrel{h^\prime}{\longrightarrow}\chi(\mathbf{\Sigma}^\prime)
\stackrel{h}{\longrightarrow}\chi(\mathbf{\Sigma})=M^{GIT}$$
are all weighted blow-ups. 
The exceptional divisor $\pp^1$ of $h^\prime: \chi(\mathbf{\Sigma}^{\prime\prime})\rightarrow \chi(\mathbf{\Sigma}^\prime)$ 
(also using affine coordinates $\mu$, but $\mu\neq \infty$) parametrizes the family of surfaces:
$$\widetilde{\mu}: \sS_{\mu}\to \pp^1\times\pp^1,$$
which are double covers over $\pp^1\times\pp^1$ branched over a divisor in $\sO(8,8)$ given by 
$C_{\mu}\sqcup \{s=0, s=\infty, t=0, t=\infty\}$.  Each of these four lines meets with $C_{\mu}$ in $6$ points.  Let 
$$\mbox{Bl}_{24~ pts}\sS_{\mu}\rightarrow \sS_{\mu}\rightarrow \pp^1\times\pp^1$$
be the blow-up along these $24$ points, and then let 
$$\overline{\mbox{Bl}}_{24~ pts}\sS_{\mu}\rightarrow \sS_{\mu}$$
be the  morphism by collapsing down the proper transformation of the four lines  $\{s=0, s=\infty, t=0, t=\infty\}$.
The $\mu=0$ case corresponds to the surface $S_{II}$ and all of these surfaces are  s.l.c.

For the toric Deligne-Mumford stack 
$\chi(\mathbf{\Sigma}^{\prime\prime})\to \chi(\mathbf{\Sigma})=M^{GIT}$, we collapse   down the proper transformation of the locus 
$\chi(\mathbf{\Sigma}/\rr_{(2,0)})=\pp^1$ and obtain a toric Deligne-Mumford stack 
$\chi(\overline{\mathbf{\Sigma}})$,  where the stacky fan is given by
$\overline{\mathbf{\Sigma}}=(\zz^2, \overline{\Sigma}, \beta)$.   The fan 
$\overline{\Sigma}=\{O, III, IV, II\}$ contains four top cones, where $O, III, II$ are the same as before, and 
$IV=((4,-1), (2,1))$.  This toric Deligne-Mumford stack 
$\chi(\overline{\mathbf{\Sigma}})$ is projective since the fan $\overline{\Sigma}$ is clearly complete. 

To see that $\chi(\overline{\mathbf{\Sigma}})$ is the KSBA compactification of $M^{GIT}$, note that in $M^{GIT}$, the only non-KSBA surfaces are given by the infinity divisor 
$\chi(\mathbf{\Sigma}/\rr_{(2,0)})=\pp^1$. After doing weighted blow-ups and collapsing  this infinity divisor, all surfaces parametrized by 
$\chi(\overline{\mathbf{\Sigma}})$ are  s.l.c. surfaces. 
Also, the surfaces parametrized by the top cone  $IV$ are given by (see \cite[\S 5.2]{Donaldson}) complete intersections in 
the weighted projective stack 
$\pp(1,1,1,1,2,2)$.  More precisely, let $\pp(1,1,1,1,2,2)=\proj(\mathbf{k}[x_1, y_1, x_2, y_2, h_+, h_-])$ where 
$x_1, y_1, x_2, y_2$ have degree $1$ and $h_+, h_-$ have degree $2$. We define the surfaces 
$S^{\alpha, \beta}\subset \pp(1,1,1,1,2,2)$ by
\begin{equation}
S^{\alpha,\beta}=
\begin{cases}
x_1^6+y_1^6+x_2^6+y_2^6+h_+^3+h_+h_-^2=0;\\
x_1y_1=\alpha h_++\beta h_-;\\
x_2y_2=\alpha h_+-\beta h_-.
\end{cases}
\end{equation}
The most singular one $S^{0,0}$ corresponds to the vertex in $IV$, which corresponds to the surface 
in $III$ and $II$ by taking $\mu\rightarrow \infty$. Also from \cite[\S 5.2]{Donaldson},  the surfaces in $III$ and $II$ can be obtained from the surfaces 
$S^{\alpha, \beta}$.   The surfaces $S^{\alpha, \beta}$ are complete intersections, therefore are Gorenstein; i.e., the dualizing sheaf $\omega_{S^{\alpha,\beta}}$ is a line bundle for any pair $(\alpha, \beta)$. 

\subsubsection{Virtual fundamental class}

From the construction  in \S \ref{subsubsec_KSBA}, the KSBA moduli space $M=M_N$ is the toric surface $M=\chi(\overline{\mathbf{\Sigma}})$.  There exists a universal family 
$$p: \scM\rightarrow \chi(\overline{\mathbf{\Sigma}})$$
which is projective, flat and relatively Gorenstein.  It is relatively Gorenstein since every fiber  surface $S_t$ of $p$ at $t\in \chi(\overline{\mathbf{\Sigma}})$ is a complete intersection surface.  This implies that the relative dualizing sheaf
$\omega_{\scM/\chi(\overline{\mathbf{\Sigma}})}$ is a line bundle.   Therefore from Corollary \ref{cor_POT_MG} and Corollary \ref{prop_POT_MG_virtual}, 
 we have 
\begin{prop}
Let 
$$E^{\bullet}_{\chi(\overline{\mathbf{\Sigma}})}:=Rp_{*}\left(\ll^{\bullet}_{\scM/\chi(\overline{\mathbf{\Sigma}})}\otimes\omega_{\scM/\chi(\overline{\mathbf{\Sigma}})}[2]\right)[-1].$$
Then there exists  a perfect obstruction theory
$$\phi: E^{\bullet}_{\chi(\overline{\mathbf{\Sigma}})}\rightarrow \ll^{\bullet}_{\chi(\overline{\mathbf{\Sigma}})}.$$
Therefore, there  exists a virtual fundamental class $[\chi(\overline{\mathbf{\Sigma}})]^{\vir}\in A_{\vd}(\chi(\overline{\mathbf{\Sigma}}))$.  
This proves Donaldson's conjecture for the existence of virtual fundamental class in \cite[\S 5]{Donaldson}. 
\end{prop}

The virtual dimension $\vd=1$ was calculated in \cite[\S 5]{Donaldson}.  The moduli stack $\chi(\overline{\mathbf{\Sigma}})$ is smooth of dimension $2$, but has wrong dimension.

We briefly review the calculation of the virtual dimension.  We actually have for a sextic hypersurface $S_6$, 
$$\dim H^1(S_6, T_{S_6})^\mathfrak{G}=2; \quad  \dim H^2(S_6, T_{S_6})^\mathfrak{G}=1,$$
where the calculation in \cite[\S 5]{Donaldson} is given as follows:   look at the Euler sequence
$$0\to T^*\pp^3(1)\rightarrow \sO^{\oplus 4}\rightarrow \sO(1)\to 0.$$
We have an exact sequence
$$0\to T^*\pp^3(2)\rightarrow \sO(1)^{\oplus 4}\rightarrow \sO(2)\to 0.$$
Taking sections gives
$$0\to H^0(T^*\pp^3(2))\rightarrow \sO^{\oplus 4}\otimes \sO^{\oplus 4}\rightarrow S^2(\sO^{\oplus 4})\to 0.$$
Since the canonical line bundle 
$K_{S_6}\cong \sO_{S_6}(2)$, we have 
$$H^0(T^*S_6\otimes K_{S_6})\cong \Lambda^2(\sO^{\oplus 4})$$
and the $\mathfrak{G}$-equivariant part of $\Lambda^2(\sO^{\oplus 4})$ is $1$-dimensional spanned by 
$\omega=dx_1dy_1+dx_2dy_2$. By Serre duality, the obstruction space has dimension $\dim H^2(S_6, T_{S_6})^\mathfrak{G}=1$.

The moduli stack $\chi(\overline{\mathbf{\Sigma}})$ admits an obstruction bundle which is a line bundle such that,  over a point
$t\in \chi(\overline{\mathbf{\Sigma}})$ representing a sextic surface $S_6$, it is given by the obstruction space  satisfying $\dim H^2(S_6, T_{S_6})^\mathfrak{G}=1$.
As proved  in \cite[Page 24]{Donaldson}, the obstruction bundle is given by studying the section 
$s_{\omega}\in T^*\pp^3(2)$ defined by the symplectic form $\omega$ on $\aaa_{\mathbf{k}}^4$. We omit the details and for more precise proof, we refer to  \cite[Page 24]{Donaldson}. We denote by $L_{\Ob}$ the obstruction bundle. 
Since the moduli stack $\chi(\overline{\mathbf{\Sigma}})$ is a smooth toric Deligne-Mumford stack, standard perfect obstruction theory in \cite{BF} implies that the virtual fundamental class is
$$[\chi(\overline{\mathbf{\Sigma}})]^{\vir}=e(L_{\Ob})\cap [\chi(\overline{\mathbf{\Sigma}})]\in A_1(\chi(\overline{\mathbf{\Sigma}})).$$

In the new toric Deligne-Mumford stack  $\chi(\overline{\mathbf{\Sigma}})$,  we have two divisors 
$$D_{II}:=\chi(\overline{\mathbf{\Sigma}}/\rr_{(-2,-1)});\quad  D_{III}:=\chi(\overline{\mathbf{\Sigma}}/\rr_{(0,1)}).$$
The coarse moduli space of these two substacks are all isomorphic to $\pp^1$, and is the same as the closed substack in 
$M^{GIT}=[\pp^2/\mu_2]$ corresponding to the rays $\rr_{(-2,-1)}$ and $\rr_{(0,1)}$. 
Donaldson \cite[Formula (19)]{Donaldson} calculated that 
$$\langle c_1(L_{\Ob}), D_{II}\rangle=-1/4.$$
Also \cite[Formula (17)]{Donaldson} calculated that 
$$
\langle c_1(\lambda_2), D_{II}\rangle=12.
$$
So $c_1(L_{\Ob})=-\frac{1}{48}c_1(\lambda_2)$ and 
\begin{equation}\label{eqn_index1}
\mbox{PD}[\chi(\overline{\mathbf{\Sigma}})]^{\vir}=-\frac{1}{48}c_1(\lambda_2).
\end{equation}

Also Donaldson calculated  
$$\langle c_1(\lambda_2)^2, [\chi(\overline{\mathbf{\Sigma}})])=288$$
in  \cite[Formula (18)]{Donaldson} using the property of the line bundle $\lambda_2$.
Thus 
\begin{equation}\label{eqn_index2}
\langle c_1(\lambda_2), [\chi(\overline{\mathbf{\Sigma}})]^{\vir}\rangle=-6.
\end{equation}

\subsubsection{Tautological invariants}\label{subsec_tautological_invariants}

Let us calculate one tautological invariant following \cite[\S 5.3]{Donaldson}. 
There are two MMM-classes associated to the characteristic classes $c_1^3, c_1c_2^2$.
Donaldson calculated the integration of these classes against the virtual fundamental class 
$[\chi(\overline{\mathbf{\Sigma}})]^{\vir}$.

Consider the CM line bundle $L_{\CM}:=\lambda_{\CM}(\scM/\chi(\overline{\mathbf{\Sigma}}), K_{\scM/\chi(\overline{\mathbf{\Sigma}})})$ in \S \ref{subsec_CM_line_bundle}.
We have 
$$\lambda_{\CM}(\scM/\chi(\overline{\mathbf{\Sigma}}), K_{\scM/\chi(\overline{\mathbf{\Sigma}})})=\lambda_3^{4}\otimes \lambda_2^{-6},$$
where $\lambda_2, \lambda_3$ are line bundles on $\chi(\overline{\mathbf{\Sigma}})$.
Serre duality implies that $\lambda_3\cong \lambda^2_2$. Thus 
$$\lambda_{\CM}(\scM/\chi(\overline{\mathbf{\Sigma}}), K_{\scM/\chi(\overline{\mathbf{\Sigma}})})=\lambda_2^{2}.$$
Then $L_{\CM}=\lambda_2^2$.
The tautological invariant in Definition \ref{defn_tautological_invariants} is
$$I_{\CM}=\int_{[\overline{M}_{K^2, \chi}]^{\vir}}c_1(L_{\CM})=-12$$
from (\ref{eqn_index2}).

\begin{rmk}
Donaldson \cite[\S 5.4]{Donaldson} related the KSBA compactification $\chi(\overline{\mathbf{\Sigma}})$ to some moduli space of stable maps to 
$\pp^2/(\mu_2\times\mu_2)$ and probably Gromov-Witten invariants of $\pp^2/(\mu_2\times\mu_2)$.  It is very interesting to explore its deep relationship. 
\end{rmk}

\subsection{Short discussion on the moduli stack of sextic surfaces}

For a large divisible $N>0$, let $\overline{M}_{24,11, N}$ be the KSBA moduli stack of sextic surfaces $S$ with $K_S^2=24, \chi(\sO_S)=11$.  
Although it seems hard to obtain explicitly all the boundary divisors of  $\overline{M}_{24,11,N}$ which contain s.l.c. sextic surfaces with quotient singularities, in \cite{Horikawa2} Horikawa classified all the deformations of smooth sextic hypersurfaces; i.e., the substack  for $N=1$. Let us review  \cite[Theorem 1]{Horikawa2}.  Let $S$ be a smooth sextic surface in $\pp^3$, then the line bundle $K_S$ is divisible by $2$ which we denote by $2L=K_S$.  From \cite[Lemma 2.1]{Horikawa2}, $h^0(S, L)=4$, thus,  the line bundle $L$ determines a morphism 
$$\phi_L: S\rightarrow \pp^3.$$
Then from \cite[Theorem 1]{Horikawa2}, there are totally six deformations of $S$ associated with the morphism $\phi_L$. 

Ia: $S$ is birationally equivalent to a sextic surface in $\pp^3$ with at most RDP's as singularities;

Ib: $\phi_L$ is a generically $2$-fold map onto a cubic surface in $\pp^3$;

Ic: $\phi_L$ is a generically $3$-fold map onto a quadratic  surface in $\pp^3$;

IIa: $\phi_L$ is a generically $2$-fold map onto a smooth quadratic  surface in $\pp^3$;

IIb: $\phi_L$ is a generically $2$-fold map onto a singular quadratic  surface in $\pp^3$;

III: $\phi_L$ is  composed of a pencil of curves of genus $3$ of non-hyperelliptic type.

In \cite{Horikawa2} Horikawa gave explicit constructions for each possible deformation. We list all the constructions as complete intersection surfaces 
in weighted projective spaces.

Ia: The surface $S$ of type Ia is a sextic hypersurface $S\subset \pp^3$ given by a degree $6$ homogeneous polynomial with only RDP's as singularities.

Ib:  The surface $S$ of type Ib is a complete intersection surface in $\pp(3,1,1,1,1)$ with coordinates $(w,x_0,x_1,x_2,x_3)$ of weights $(3,1,1,1,1)$ given by
$$g=0; \quad  w^2+f=0,$$
where $g=g(x_0,x_1,x_2,x_3)$ is cubic function and $f=f(x_0,x_1,x_2,x_3)$ is a degree $6$ homogeneous polynomial. 

Ic:  The surface $S$ of type Ic is a complete intersection surface in $\pp(2,1,1,1,1)$ with coordinates $(u,x_0,x_1,x_2,x_3)$ of weights $(2,1,1,1,1)$ given by
$$g=0; \quad u^3+A_2 u^2+A_4u+A_6=0,$$
where $g=g(x_0,x_1,x_2,x_3)$ is of degree $2$ and $A_{2j}=A_{2j}(x_0,x_1,x_2,x_3)$ are  degree $2j$ homogeneous polynomials.

IIa and IIb:  For a surface $S$ of type IIa or IIb, its canonical model is in the weighted projective space  $\pp(1,1,1,1,2,3)$ with coordinates $(x_0,x_1,x_2,x_3,u,w)$ of weights $(1,1,1,1,2,3)$  defined by
$$q=0; \quad x_0 u=h; \quad w^2=u^3+A_2u^2+A_4u+A_6,$$
where $q, h, A_{2j}$ are homogeneous polynomials in $x_i$ of degree $2,3,2j$ respectively. 

III:  From \cite[\S 6]{Horikawa2}, the surface of type III can be given as a subspace in the weighted projective space $\pp(1,1,1,1,2,2,2,3)$ with coordinates $(x_0,x_1,x_2,x_3,y_1, y_2,z,w)$ of weights $(1,1,1,1,2,2,2,3)$  defined by
$$\Phi_i=0; \quad \Psi_i=0; \quad \Gamma_i=0, \quad \Delta=0.$$
Here $\Phi_i (1\leq i\leq 3)$ are of degree $2$,   $\Psi_i (1\leq i\leq 3)$ are of degree $3$,   $\Gamma_i (1\leq i\leq 3)$ are of degree $4$, and 
$\Delta$ has degree $6$.  These functions can be found in \cite[\S 6]{Horikawa2}.   Although it is hard to see if the surface of type III is a global complete intersection in  $\pp(1,1,1,1,2,2,2,3)$,  \cite[\S 6]{Horikawa2} pointed out that this surface of type III is either smooth or with rational double points as singularities. 

We can perform the same calculation as in (\ref{eqn_cohomology_d6}) to calculate the dimensions of the  cohomology spaces of such complete intersection surfaces
$S$, 
$$
\begin{cases}
\dim H^1(S, T_S)=68;\\
\dim H^2(S, T_S)=6.
\end{cases}
$$

Let $\overline{M}^{\sextic}\subset \overline{M}_{24,11, N}$ be the closure of Gieseker moduli stack $M_{24,11}\subset  \overline{M}_{24,11, N}$. 

\begin{thm}\label{thm_KSBA_sextic}
Suppose that we know all the boundary divisors consisting s.l.c. sextic surfaces in $\overline{M}^{\sextic}$, then the moduli stack $\overline{M}^{\sextic}$ is an irreducible Deligne-Mumford stack of dimension $68$. 
\end{thm}
\begin{proof}
From \cite[Theorem 2]{Horikawa}, the Gieseker moduli stack $M_{24, 11}$ (without the KSBA compactification) is irreducible.  There may have some other irreducible components  in $\overline{M}_{24,11}$  consisting of singular s.l.c. sextic surfaces.  But  the closure $\overline{M}^{\sextic}$ is  irreducible. 
For the dimension of the moduli stack, note that the dimension of the homogeneous polynomials in $\mathbf{k}[x_0,x_1,x_2,x_3]$ modulo equivalence is
$$84-16=68.$$
This locus contains all the type Ia surfaces; i.e., the sextic hypersurfaces in $\pp^3$.  All the other types of deformation surfaces above should belong to the boundary divisor since the moduli stack is irreducible.  Therefore the dimension of the moduli stack is $68$.
\end{proof}

\begin{con}\label{con_second_hohomology}
Over the s.l.c. sextic surfaces $S$ in all the boundary divisors of $\overline{M}^{\sextic}$, the dimensions of the cohomology spaces of the tangent sheaf of $S$ are given by 
$$\dim H^1(S, T_S)= 68; \quad 
\dim H^2(S, T_S)= 6.$$ 
\end{con}

\begin{rmk}
In the case of moduli stack $\overline{M}_{5,5}$ of numerical quintics, the  boundary divisors consisting of a unique Wahl singularity $\frac{1}{4}(1,1)$ were found in \cite{Rana}, where the only cases of minimal surfaces with a unique Wahl singularity are of type  $\frac{1}{4}(1,1)$ and  $\frac{1}{9}(2,5)$, and the case $\frac{1}{9}(2,5)$ was  proven in \cite{Rana} to be impossible.  

In the case of sextic surfaces, from calculation there are totally possible $29$ cases of the unique Wahl singularity in the  minimal surfaces  in the    boundary divisors, which makes the calculation much more complicated. 
\end{rmk}

Let us only consider the moduli stack $\overline{M}^{\sextic}$ such that all of its boundary divisors consist of $\qq$-Gorenstein deformation of class 
$T$-singularities.  Let $f: \overline{M}_{\ind}^{\sextic}\to \overline{M}^{\sextic}$ be the moduli stack of index one covers. 
Thus from the conjecture we have that 
\begin{prop}\label{prop_virtual_class}
Under the conjecture \ref{con_second_hohomology}, there exists a rank $6$ nontrivial obstruction bundle $\Ob\to  \overline{M}_{\ind}^{\sextic}$ such that over any surface 
$S\in \overline{M}^{\sextic}$, the fiber is given by $T^2_{\QG}(S)$.
Assume that the obstruction bundle $\Ob$ is nontrivial, then 
the virtual fundamental class $[\overline{M}_{\ind}^{\sextic}]^{\vir}\in A_{62}(\overline{M}_{\ind}^{\sextic})$ is given by 
$$[\overline{M}_{\ind}^{\sextic}]^{\vir}=e(\Ob)\cap [\overline{M}_{\ind}^{\sextic}].$$
\end{prop}
\begin{proof}
Since under the conjecture the moduli stack  $\overline{M}^{\sextic}$ and   $\overline{M}_{\ind}^{\sextic}$ are  projective Deligne-Mumford stacks  and the obstruction bundle $\Ob\to  \overline{M}_{\ind}^{\sextic}$  is nontrivial, then 
standard argument in the perfect obstruction theory shows that the virtual fundamental class is just the Euler class of the obstruction bundle.  
\end{proof}

\begin{rmk}
It is very interesting to check if Conjecture \ref{con_second_hohomology} holds, and calculate the tautological invariants for the moduli stack  $\overline{M}_{24,11, N}$.
\end{rmk}


\appendix

\section{Equivalence of perfect obstruction theories}\label{sec_POT_equivalence}

Recall that in \S \ref{subsec_universal_family}, 
we have the universal family 
$$
    p^{\mathrm{alg}}: \mathscr{M}^{\mathrm{alg}}\to M_N^{\mathrm{alg}}
$$ in (\ref{eqn_universal_family_good}).
   Different choices of representatives in the $\mathbb{S}$-equivalence class give different universal families. 
In this section we prove that different choices of the representatives in the $\mathbb{S}$-equivalence class give equivalent perfect obstruction theories. 

Let $p_1: \scM_1^{\mathrm{alg}}\to M_N^{\mathrm{alg}}$ and $p_2: \scM_2^{\mathrm{alg}}\to M_N^{\mathrm{alg}}$ be two universal families, where there exist at least two one-parameter flat families 
$f_1: \XX_1^{\lci} \to \Delta\subset M_N^{\mathrm{alg}}$  and 
$f_2: \XX_2^{\lci} \to \Delta\subset M_N^{\mathrm{alg}}$ of birational $\lci$ enhancements 
so that over $\Delta$ the universal families 
$\scM_1^{\mathrm{alg}}|_{p_1^{-1}(\Delta)}$ and $\scM_2^{\mathrm{alg}}|_{p_2^{-1}(\Delta)}$ are related by a type $(M0), (M1), (M2)$ flop.  They have $\mathbb{S}$-equivalent central fibers. 

Let $\ll^{\bullet}_{\scM_1^{\mathrm{alg}}/M_N^{\mathrm{alg}}}$ and $\ll^{\bullet}_{\scM_2^{\mathrm{alg}}/M_N^{\mathrm{alg}}}$ be the relative cotangent complexes of $p_1$ and $p_2$. We take their derived duals $\ttt^{\bullet}_{\scM_1^{\mathrm{alg}}/M_N^{\mathrm{alg}}}$ and $\ttt^{\bullet}_{\scM_2^{\mathrm{alg}}/M_N^{\mathrm{alg}}}$ as the relative tangent complexes of $p_1$ and $p_2$.  Let 
$$E^{\bullet}_{M_N^{\mathrm{alg}},1}:=Rp_{1*}\left(\ll^{\bullet}_{\scM_1^{\mathrm{alg}}/M_N^{\mathrm{alg}}}\otimes\omega^{\bullet}_{\scM_1^{\mathrm{alg}}/M_N^{\mathrm{alg}}}\right)[-1]$$
and 
$$E^{\bullet}_{M_N^{\mathrm{alg}},2}:=Rp_{2*}\left(\ll^{\bullet}_{\scM_2^{\mathrm{alg}}/M_N^{\mathrm{alg}}}\otimes\omega^{\bullet}_{\scM_2^{\mathrm{alg}}/M_N^{\mathrm{alg}}}\right)[-1],$$
where $\omega^{\bullet}_{\scM_i^{\mathrm{alg}}/M_N^{\mathrm{alg}}}$  represent the relative dualizing complexes of these two universal families. 
Theorem \ref{thm_POT_Muniv} implies that there are perfect obstruction theories 
$$\phi_1^{\mathrm{alg}}: E^{\bullet}_{M_N^{\mathrm{alg}},1}\to \ll_{M_N^{\mathrm{alg}}}^{\bullet}$$
and 
$$\phi_2^{\mathrm{alg}}: E^{\bullet}_{M_N^{\mathrm{alg}},2}\to \ll_{M_N^{\mathrm{alg}}}^{\bullet}.$$

\begin{thm}\label{thm_cotangent_complex_universal-12}
 The two perfect obstruction theories $\phi_1^{\mathrm{alg}}$ and $\phi_2^{\mathrm{alg}}$ are isomorphic. 
\end{thm}
\begin{proof}
We show that the two  complexes  $E^{\bullet}_{M_N^{\mathrm{alg}},1}$ and $E^{\bullet}_{M_N^{\mathrm{alg}},2}$ are isomorphic in the derived category.
First we have 
$$(E^{\bullet}_{M_N^{\mathrm{alg}},1})^{\vee}=Rp_{1*}\left(\ttt^{\bullet}_{\scM_1^{\mathrm{alg}}/M_N^{\mathrm{alg}}}\right)[1]$$
and 
$$(E^{\bullet}_{M_N^{\mathrm{alg}},2})^{\vee}=Rp_{2*}\left(\ttt^{\bullet}_{\scM_2^{\mathrm{alg}}/M_N^{\mathrm{alg}}}\right)[1].$$
The cohomology groups $H^s((E^{\bullet}_{M_N^{\mathrm{alg}},i})^{\vee})$ (for $i=1,2$) of these two complexes don't vanish only for $s=0, 1$. 
When restricted to a point $x\in M_N^{\mathrm{alg}}$, which corresponds to  birational  lci enhancement Deligne-Mumford stacks $\XX_i^{\lci}$ in the two universal families, the cohomology groups $H^0((E^{\bullet}_{M_N^{\mathrm{alg}},i})^{\vee})|_{x}$ are the cohomology groups  $\Ext^1(\ll^{\bullet}_{\XX_i^{\lci}}, \sO)$,  while  $H^1((E^{\bullet}_{M_N^{\mathrm{alg}},i})^{\vee})|_{x}$ are the cohomology groups $\Ext^2(\ll^{\bullet}_{\XX_i^{\lci}}, \sO)$, where 
$\ll^{\bullet}_{\XX_i^{\lci}}$ are the cotangent complexes of $\XX_i^{\lci}$ for $i=1,2$. The spaces $\Ext^1(\ll^{\bullet}_{\XX_i^{\lci}}, \sO)$ and $\Ext^2(\ll^{\bullet}_{\XX_i^{\lci}}, \sO)$ are the deformation and obstruction spaces of $\XX_i^{\lci}$ for $i=1,2$. 

If over a space $T\subset M_N^{\mathrm{alg}}$, the restriction families 
    $\scM_1^{\mathrm{alg}}|_{T}=\XX_1^{\lci}$ and 
    $\scM_2^{\mathrm{alg}}|_{T}=\XX_2^{\lci}$ are 
    related by a flop, then from Definition \ref{defn_lci_DM_base}, there  exists a divisor $B\subset T$, and a smooth $\pp^1$-fibration $\mathcal{E}\to B$ for which the normal bundle to 
$\mathcal{E}$ restricts to $\sO(-1)\oplus \sO(-1)$ on every fiber. 
    We take a common blow-up for this flop
    \begin{equation}\label{eqn_diagram_flop}
    \xymatrix{
    &(\widetilde{\XX}^{\lci})\ar[dl]_{\varphi_1}\ar[dr]^{\varphi_2}\ar[dd]_{f}&\\
    \XX_1^{\lci}\ar[dr]_{f_1}&&\XX_2^{\lci}\ar[dl]^{f_2}\\
    &T &
    }
    \end{equation}
    where exceptional locus is the projective bundle $\pp(\mathcal{E})\to B$ for which the normal bundle to $\pp(\mathcal{E})$ restrict to $\sO(-1, -1)$ on every fiber. 
Therefore, there exists a universal divisor $\mathbf{B}\subset M_N^{\mathrm{alg}}$, and a 
smooth $\pp^1$-fibration $\mathcal{E}\to \mathbf{B}$ for which the normal bundle to 
$\mathcal{E}$ restricts to $\sO(-1)\oplus \sO(-1)$ on every fiber. Comparing with (\ref{eqn_diagram_flop}), there exists a universal diagram 
\begin{equation}\label{eqn_diagram_flop_Mlci}
    \xymatrix{
    &(\widetilde{\scM}^{\mathrm{alg}})\ar[dl]_{\varphi_1}\ar[dr]^{\varphi_2}\ar[dd]_{p}&\\
    \scM_1^{\mathrm{alg}}\ar[dr]_{p_1}&&\scM_2^{\mathrm{alg}}\ar[dl]^{p_2}\\
    &M_N^{\mathrm{alg}}, &
    }
    \end{equation}
where $\widetilde{\scM}^{\mathrm{alg}}\to \scM_i^{\mathrm{alg}}$ are the common blow-ups along the locus $\mathcal{E}\to \mathbf{B}$. 
Let 
$$\widetilde{E}^{\bullet}_{M_N^{\mathrm{alg}}}:=Rp_{*}\left(\ll^{\bullet}_{\widetilde{\scM}^{\mathrm{alg}}/M_N^{\mathrm{alg}}}\otimes\omega^{\bullet}_{\widetilde{\scM}^{\mathrm{alg}}/M_N^{\mathrm{alg}}}\right)[-1]$$
where $\omega^{\bullet}_{\widetilde{\scM}^{\mathrm{alg}}/M_N^{\mathrm{alg}}}$ represents the relative dualizing complex of $p$. 
Then we have 
$$(\widetilde{E}^{\bullet}_{M_N^{\mathrm{alg}}})^{\vee}=Rp_{*}\left(\ttt^{\bullet}_{\widetilde{\scM}^{\mathrm{alg}}/M_N^{\mathrm{alg}}}\right)[1].$$
From the properties of the derived pullback and pushforward,  we have 
$$E^{\bullet}_{M_N^{\mathrm{alg}},1}:=Rp_{*}\left(\varphi_1^*\ll^{\bullet}_{\scM_1^{\mathrm{alg}}/M_N^{\mathrm{alg}}}\otimes\omega^{\bullet}_{\scM_1^{\mathrm{alg}}/M_N^{\mathrm{alg}}}\right)[-1]$$
and 
$$E^{\bullet}_{M_N^{\mathrm{alg}},2}:=Rp_{*}\left(\varphi_2^*\ll^{\bullet}_{\scM_2^{\mathrm{alg}}/M_N^{\mathrm{alg}}}\otimes\omega^{\bullet}_{\scM_2^{\mathrm{alg}}/M_N^{\mathrm{alg}}}\right)[-1],$$

From Diagram (\ref{eqn_diagram_flop_Mlci}), we have morphisms 
\begin{equation}\label{eqn_Psi}
        \Psi_i:  \varphi_i^*(\ll^{\bullet}_{\scM^{\mathrm{alg}}_i/M_N^{\mathrm{alg}}})\to \ll^{\bullet}_{\widetilde{\scM}^{\mathrm{alg}}/M_N^{\mathrm{alg}}}
    \end{equation}
    and 
   \begin{equation}\label{eqn_Psi_omega}
        \Theta_i:  \varphi_i^*(\omega^{\bullet}_{\scM^{\mathrm{alg}}_i/M_N^{\mathrm{alg}}})\to \omega^{\bullet}_{\widetilde{\scM}^{\mathrm{alg}}/M_N^{\mathrm{alg}}}
    \end{equation} 
    for $i=1,2$.
  After applying derived pushforward there are morphisms 
   $$
    \Phi_i: 
    Rp_{*}(\varphi_i^*\ll^{\bullet}_{\scM^{\mathrm{alg}}_i/M_N^{\mathrm{alg}}}\otimes \varphi_i^*(\omega^{\bullet}_{\scM^{\mathrm{alg}}_i/M_N^{\mathrm{alg}}})\to 
    Rp_{*}(\ll^{\bullet}_{\widetilde{\scM}^{\mathrm{alg}}/M_N^{\mathrm{alg}}}\otimes \omega^{\bullet}_{\widetilde{\scM}^{\mathrm{alg}}/M_N^{\mathrm{alg}}})).
   $$
which  gives morphisms
  \begin{equation}\label{eqn_Phi-1}  
  \Phi_i: E^{\bullet}_{M_N^{\mathrm{alg}},i}[1]\to \widetilde{E}^{\bullet}_{M_N^{\mathrm{alg}}}[1]
 \end{equation}
 for $i=1,2$.
We further have the following commutative diagram 
  \[
\xymatrix{
Rp_{*}(\varphi_i^*\ll^{\bullet}_{\scM^{\mathrm{alg}}_i/M_N^{\mathrm{alg}}}\otimes \varphi_i^*(\omega^{\bullet}_{\scM^{\mathrm{alg}}_i/M_N^{\mathrm{alg}}})\ar[r]^-{\Phi_i}\ar[d]_{\phi_i^{\mathrm{alg}}}& Rp_{*}(\ll^{\bullet}_{\widetilde{\scM}^{\mathrm{alg}}/M_N^{\mathrm{alg}}}\otimes \omega^{\bullet}_{\widetilde{\scM}^{\mathrm{alg}}/M_N^{\mathrm{alg}}}))\ar[d]^{\phi}\\
\ll^{\bullet}_{M_N^{\mathrm{alg}}}[1]\ar[r]^{=}& \ll^{\bullet}_{M_N^{\mathrm{alg}}}[1]
}
\]
which shows that the morphisms $\Phi_i$ are compatible with the perfect obstruction theories. 
The above commutative diagram comes from  the following commutative diagram 
\[
\xymatrix{
\varphi_i^*(\ll_{\scM^{\mathrm{alg}}/M_N^{\mathrm{alg}}}^{\bullet})\ar[r]^{\Psi_i}\ar[d]&
\ll^{\bullet}_{\widetilde{\scM}^{\mathrm{alg}}/M_N^{\mathrm{alg}}}\ar[d]\\
\varphi_i^*p_i^*\ll_{M_N^{\mathrm{alg}}}^{\bullet}[1]\ar[r]& p^*\ll_{M_N^{\mathrm{alg}}}^{\bullet}[1]
}
\]
of the Kodaira-Spencer maps. 
  
We show that the morphisms $\Phi_i$ in (\ref{eqn_Phi-1}) are quasi-isomorphic. 
We need to prove that $\Phi_i$ induce isomorphic cohomology sheaves. The cohomology sheaves  of  complexes 
$E^{\bullet}_{M_N^{\mathrm{alg}},i}[1]$ and $\widetilde{E}^{\bullet}_{M_N^{\mathrm{alg}}}[1]$ only survive when $s=-1,-2$ whose duals calculate the cohomology sheaves  of the  tangent sheaves of the fibers of the universal families.

So we only need to prove $\Phi_i$ induces isomorphic cohomology sheaves, in which we can check over  schemes
$T\subset M_N^{\mathrm{alg}}$. Over a base scheme $T\subset M_N^{\mathrm{alg}}$, if the restriction families 
    $\scM_1^{\mathrm{alg}}|_{T}=\XX_1^{\lci}$ and 
    $\scM_2^{\mathrm{alg}}|_{T}=\XX_2^{\lci}$ are two families over $T$ which do not involve any birational $\lci$ enhancements,  the two universal families $\scM^{\mathrm{alg}}_i\to M_N^{\mathrm{alg}}$ are isomorphic, which definitely induce the isomorphic cohomology sheaves by $\Phi_i$ in (\ref{eqn_Phi-1})

If over $T\subset M_N^{\mathrm{alg}}$, the restriction families 
    $\scM_1^{\mathrm{alg}}|_{T}=\XX_1^{\lci}$ and 
    $\scM_2^{\mathrm{alg}}|_{T}=\XX_2^{\lci}$ are 
    related by a flop,  we have the  flop in Diagram (\ref{eqn_diagram_flop}).
To check the map $\Phi_i$ induce isomorphic cohomology sheaves, we only need to check the stalks of the cohomology sheaves of over any point $x\in M_N^{\mathrm{alg}}$.
The cohomology sheaves $\Ext^s(E^{\bullet}_{M_N^{\mathrm{alg}},i}[1], \sO)$ and  $\Ext^s(\widetilde{E}^{\bullet}_{M_N^{\mathrm{alg}}}[1], \sO)$  for $s=1,2$ are the deformation and obstruction spaces of the birational lci enhancement Deligne-Mumford stack $\XX^{\lci}_{i}$ and $\widetilde{\XX}^{\lci}$ corresponding to the point $x$ for the universal families $p_i$ and $p$ for $i=1,2$. 

The universal families 
$p_i: \scM^{\mathrm{alg}}_i\to M_N^{\mathrm{alg}}$  over the point $x\in M_N^{\mathrm{alg}}$ are different only when the corresponding birational lci enhancement Deligne-Mumford stacks $\XX_i^{\lci}$  are related by a flop. 
Therefore, we may let $T=\Delta$, and $x\in \Delta$ in a disk and the universal families restrict to one-parameter families given by flops.   

This reduces to the local case that  Diagram (\ref{eqn_diagram_flop}) contains the Atiyah type flop
    \begin{equation}\label{eqn_diagram_flop_Atiyah}
    \xymatrix{
    &Y\ar[dl]_{\varphi_1}\ar[dr]^{\varphi_2}\ar[dd]_{f}&\\
    S_1\ar[dr]_{f_1}&&S_2\ar[dl]^{f_2}\\
    &\Delta, &
    }
    \end{equation}
    where $Y$ contains $\pp^1\times\pp^1$ as exceptional divisor. 
 In terms of Diagram (\ref{eqn_diagram_flop}),  the morphisms $\Phi_i$ in (\ref{eqn_Phi-1}) are  
    \begin{equation}\label{eqn_Phi-11}
    \Phi_i: 
    Rf_{*}\varphi_i^*(\ll^{\bullet}_{\XX^{\lci}_i/\Delta}\otimes \omega^{\bullet}_{\XX^{\lci}_i/\Delta})\to 
    Rf_{*}(\ll^{\bullet}_{\widetilde{\XX}^{\lci}/\Delta}\otimes\omega^{\bullet}_{\widetilde{\XX}^{\lci}/\Delta}) 
    \end{equation}
    for $i=1,2$. 

    To prove that $\Phi_i$ induce isomorphic cohomology sheaves, 
we use the following exact sequence 
\begin{equation}\label{eqn_exact_cotangent_lci_flop}
    \varphi_i^* \Omega_{\XX^{\lci}_i/\Delta}\rightarrow \Omega_{\widetilde{\XX}^{\lci}/\Delta}\rightarrow \Omega_{\widetilde{\XX}^{\lci}/\XX^{\lci}_i} \to 0,
\end{equation}
of cotangent sheaves
for $i=1,2$ (the relative cotangnet complexes are the resolution of cotangent sheaves).  
The relative cotangent  sheaves $\Omega_{\widetilde{\XX}^{\lci}/\XX^{\lci}_i}$ only support on one factor $\pp^1$ in the exceptional divisor $\pp^1\times \pp^1$ for the common blow up $Y$ in the Atiyah type flop. So it is enough to calculate $\Omega_{Y/S_i}$.

Both $S_1$ and $S_2$ are  the total space $\text{Tot}(\sO_{\pp^1}(-1)\oplus \sO_{\pp^1}(-1))$, and $Y=\text{Tot}(\sO_{\pp^1\times\pp^1}(-1,-1))\to S_i$ are the blow-up along $\pp^1$.
Look at the diagram 
\[
\xymatrix{
\pp^1\times\pp^1\ar@{^{(}->}[r]^--{j}\ar[d]_{\pi}& Y=\sO_{\pp^1\times\pp^1}(-1,-1)\ar[d]^{\pi}\\
\pp^1\ar@{^{(}->}[r]& S_1=\sO_{\pp^1}(-1)\oplus \sO_{\pp^1}(-1)
}
\]
we have 
$$
\Omega_{Y/S_1}\cong j_{*}(\Omega_{\pp^1\times \pp^1/\pp^1})\cong j_{*}(\sO_{\pp^1\times\pp^1}(0,-2)).
$$
For $S_2$, we have 
$$
\Omega_{Y/S_2}\cong j_{*}(\Omega_{\pp^1\times \pp^1/\pp^1})\cong j_{*}(\sO_{\pp^1\times\pp^1}(-2,0)).
$$

We need to calculate the cohomology sheaves $\Ext^s(E^{\bullet}_{M_N^{\mathrm{alg}},i}[1], \sO)$ and  $\Ext^s(\widetilde{E}^{\bullet}_{M_N^{\mathrm{alg}}}[1], \sO)$  for $s=1,2$.  These two cohomology sheaves are given by 
$\Ext^s(\Omega_{S_i}, \sO)$ and  $\Ext^s(\Omega_{Y}, \sO)$  in the case of Atiyah type flops, and they are isomorphic from the long exact sequence in cohomology from the exact sequence in (\ref{eqn_exact_cotangent_lci_flop}), and the vanishing   
$$\Ext^s(\Omega_{Y/S_i},\sO)=0$$ 
for $s> 0$. 
This shows that (\ref{eqn_Phi-11}) induce isomorphic cohomology sheaves, which imply that 
the morphisms $\Phi_i$ in (\ref{eqn_Phi-1}) are quasi-isomorphic. 
Thus, $E^{\bullet}_{M_N^{\mathrm{alg}},1}$ and $E^{\bullet}_{M_N^{\mathrm{alg}},2}$ are quasi-isomorphic. 
\end{proof}



\subsection*{}

\end{document}